\documentclass{amsart}
\usepackage[english]{babel}
\usepackage{graphicx}
\usepackage[hidelinks]{hyperref} 
\usepackage{amssymb}
\usepackage{amsmath}
\usepackage[foot]{amsaddr}
\usepackage{amsfonts}
\usepackage{xcolor}
\usepackage{mathtools}
\usepackage{algpseudocode}
\usepackage{algorithm}
\usepackage{bm,mathrsfs,braket}
\usepackage[top=1in, bottom=1.25in, left=1.25in, right=1.25in]{geometry}
\usepackage{comment}
\usepackage{mathtools}
\usepackage{xparse}
\usepackage[normalem]{ulem}
\usepackage{empheq}

\graphicspath{{images/}}
\usepackage{float}

\theoremstyle{plain}
\newtheorem{thm}{Theorem}[section]	

\newtheorem{cor}[thm]{Corollary}
\newtheorem{prop}[thm]{Proposition}

\newtheorem{rmk}{Remark}[section]

\theoremstyle{definition}

\newcommand{\R}{\mathbb{R}}

\newcommand{\N}{\mathbb{N}}

\newcommand{\dmu}{\mathrm{d}\mu}
\newcommand{\dxi}{\mathrm{d}\xi}
\newcommand{\E}{\mathbb{E}}
\newcommand{\RV}{\operatorname{RV}(I)}

\newcommand{\NPC}{N_{PC}}

\newcommand{\func}[1]{\mathbf{#1}}
\newcommand{\vect}[1]{\mathbf{#1}}

\makeatletter
\algnewcommand{\LineComment}[1]{\Statex \hskip\ALG@thistlm {\color{white}\textbf{Input:}} #1}
\makeatother

\begin{document}

\title[Stochastic bifurcation analysis via polynomial chaos]{Stochastic bifurcation analysis via polynomial chaos: consistency and convergence of branch-approximating solutions}

\author{Giacomo Venier$^1$,
Isabella Carla Gonnella$^1$,
Federico Pichi$^1$,
Gianluigi Rozza$^1$}
\address{$^1$ mathLab, Mathematics Area, SISSA, via Bonomea 265, I-34136 Trieste, Italy}

\begin{abstract}
Parameter-dependent dynamical systems that exhibit bifurcations pose significant computational challenges, as traditional continuation methods require repeated, costly simulations across large ranges of parameter values to capture sudden qualitative changes in the solution.
In this work, we propose a systematic approach to perform a one-shot reconstruction of entire bifurcation branches in the whole parametric domain leveraging generalized Polynomial Chaos (PC) expansion. 
By treating the parameter as a random variable, we cast the deterministic parameter-dependent model in a weak stochastic form, and then use a Galerkin projection to recover bifurcation branches globally across the parameter domain without iterative pointwise continuation.
We show that the resulting Galerkin system, in the non-uniqueness regime, produces many discrete algebraic roots that naturally split into two classes: highly oscillatory solutions and branch-approximating ones.
We develop a rigorous theoretical framework that establishes consistency, proves convergence of the branch-approximating solutions to the true steady states, and guarantees uniqueness of the Galerkin solution under suitable assumptions.
Finally, we confirm these theoretical results with numerical experiments on several parameter-dependent ordinary differential equations (ODEs), demonstrating the accuracy and computational efficiency of our single-run framework in capturing complex bifurcation diagrams for both scalar and vector-valued systems.

\end{abstract}

\maketitle

\section{Introduction}\label{sec:intro}

The study of parameter-dependent non-linear dynamical systems plays a central role in applied mathematics. 
Specifically, bifurcation analysis studies time-dependent attractors, e.g.\ limit cycles and quasi-periodic motions, and more in general how equilibrium manifolds change w.r.t.\ parameter variations, identifying critical values where qualitative transitions occur.

Focusing on stationary equilibria, such topological transitions of the steady states are classified by local bifurcation theory into canonical normal forms~\cite{wiggins2003, Kuznetsov2023}, such as the \textit{saddle-node} bifurcation, where two equilibria collide and annihilate, the \textit{transcritical  bifurcation}, which involves the exchange of stability of two steady states, and the \textit{pitchfork bifurcation}, which is characterized by the emergence of symmetric branches. 

Traditionally, these equilibrium manifolds are called \emph{solution branches}, and are numerically identified through continuation methods, such as predictor-corrector algorithms and pseudo-arclength continuation~\cite{allgower1990numerical, UeckerNumericalContinuationBifurcation2021,Seydel2010}. 
While these iterative procedures are well-established, they suffer from inherent limitations. They require fine-tuning of stepping parameters, exhibit severe sensitivity near turning points or bifurcation singularities, and are vulnerable to undesirable branch jumping. 
Furthermore, since they evaluate the system point-by-point, this many-query application can become computationally prohibitive.

To overcome the limitations of sequential stepping, several advanced numerical techniques have been developed to enhance the detection of bifurcating branches. 
For instance, deflated continuation methods~\cite{FarrellDeflationTechniquesFinding2015, pintore2020efficient,PichiDeflationbasedCertifiedGreedy2025a,BoulleBifurcationAnalysisTwodimensional2022,KumarBifurcationCurveDetection2026} have been introduced to systematically discover disconnected branches by penalizing previously computed solutions. 

In the context of parametrized partial differential equations, reduced order modeling (ROM) techniques~\cite{Rozza2014, QuarteroniReducedBasisMethods2016} have proven to be highly effective in accelerating bifurcation analysis by projecting the full-order problem onto a low-dimensional manifold~\cite{DengLoworderModelSuccessive2020a,PichiDrivingBifurcatingParametrized2022,OlshanskiiApproximatingBranchSolutions2025,HerreroRBReducedBasis2013}. 
Furthermore, data-driven approaches, including artificial neural networks~\cite{PichiArtificialNeuralNetwork2023,LiDatadrivenModelingBifurcation2025,PiaSurrogateNormalformsNumerical2025,PichiGraphConvolutionalAutoencoder2024,ShahabNeuralNetworksBifurcation2025} and sparse identification~\cite{BruntonDiscoveringGoverningEquations2016,TomadaSparseIdentificationBifurcating2025,ContiReducedOrderModeling2023}, have recently been employed to construct efficient surrogates for complex bifurcating dynamics.

In parallel to deterministic model reduction, bifurcations have also been defined in stochastic systems, which are widely used in real-world applications to account for uncertainties, such as unknown physical parameters or noisy initial and boundary conditions~\cite{KuehnUncertaintyQuantificationAnalysis2024,VenturiStochasticBifurcationAnalysis2010, xiu2003modeling, lemaitre2009, patil2023}. While the broader field of stochastic bifurcations typically investigates the dynamic response of systems to additive or multiplicative noise, recent work~\cite{gonnella2024stochastic} has demonstrated that parameter randomization can also be strategically leveraged as a numerical tool to detect \emph{deterministic} bifurcations. 

Specifically, it has been shown that treating the bifurcation parameter as a random variable, and exploiting Polynomial Chaos (PC) expansions~\cite{ghanem1991stochastic, xiu2010}, allows to get an approximation of 
multiple deterministic branches simultaneously.
Indeed, the original parameter-dependent problem is cast into a weak form over both the spatial and parameter space, and the obtained stochastic Galerkin system admits solutions with multi-modal probability density functions, which exhibit peaks in correspondence to the values of the deterministic bifurcation branches.

However, while this PC-based framework has shown remarkable empirical success even for non-trivial partial differential equations (PDEs) modeling fluid-dynamic instabilities, applying such strategy to highly non-linear dynamical systems introduces profound, yet unexplored, mathematical challenges~\cite{debusschere2005}. 
For instance, as dictated by Bézout's Theorem for polynomial systems~\cite{Fulton1998, SommeseWampler2005}, projecting a non-linear vector field onto a high-degree polynomial subspace leads to a combinatorial explosion in the number of discrete algebraic roots, and thus the stochastic Galerkin system admits many additional solutions on top of the multi-modal ones found in \cite{gonnella2024stochastic}.

\subsubsection*{Contributions.} The main result of this work is the rigorous identification and mathematical analysis of a particular class of stochastic Galerkin solutions, which we call \textit{branch-approximating}.
They form a sparse subclass of the full Galerkin solution set, converging to the continuous deterministic bifurcation branches entirely supported in the parameter domain.

In this work, we develop a novel comprehensive theoretical framework that establishes consistency and convergence of branch-approximating solutions to the true deterministic branches for scalar/vector non-linear parameter-dependent ODEs. 
Furthermore, we demonstrate under suitable assumptions that the Galerkin projection preserves key qualitative and topological properties of the original system, such as the non-existence or uniqueness of solutions, alongside limitations on partially supported branches.

Building on these theoretical guarantees, we also provide a degree-continuation algorithm to isolate and reconstruct the branch-approximating solutions from the larger set of Galerkin solutions.
We validate our framework through extensive numerical experiments, successfully computing the solution branches for canonical normal-form models such as the pitchfork and the s-shaped bifurcation (which consists of two saddle-node bifurcations), assessing its computational efficiency against classical and pseudo-arclength continuation. We then extend this validation to the multi-dimensional double-pitchfork of the genetic toggle switch system~\cite{gardner2000, bollenbach2023}, and to the Lorenz system~\cite{lorenz1963deterministic, sparrow1982lorenz} for chaotic dynamics.
For this latter benchmark, we also verify our uniqueness result, identifying a non-bifurcating parameter regime and proving that the PCE-Galerkin formulation admits a unique branch-approximating solution.
Finally, we showcase the robustness of our framework by analyzing a degenerate dynamical system, demonstrating its capability to capture an infinite number of steady states across a 2-dimensional equilibrium manifold.

The remainder of the paper is organized as follows. 
Section~\ref{sec:ode-bifurcations} formulates the parameter-dependent ODE problem, introduces the weak formulation via Polynomial Chaos, and highlights the roots proliferation induced by the Galerkin projection.
Section~\ref{sec:theoretical-results} is devoted to the theoretical analysis, proving the consistency, convergence, and well-posedness of the branch-approximating solutions. 
Section~\ref{sec:numerical-results} presents the numerical validation on the canonical normal forms and the genetic toggle switch, followed by the rigorous analysis of the Lorenz system and the degenerate spherical manifold. 
Conclusions are drawn in Section~\ref{sec:conclusions}.
\section{Bifurcation analysis via Polynomial Chaos Expansion}\label{sec:ode-bifurcations}

Let $I\subset \R$ be a compact interval and let
$$
\func{f}:\R^n\times I\to \R^n,\qquad \text{s.t.}\qquad (\vect{u},\mu)\mapsto \func{f}(\vect{u},\mu),
$$
be a continuous and locally Lipschitz function in the first variable, uniformly with respect to the parameter, i.e.\ for every $R>0$ there exists $L_R>0$ such that
\begin{equation}
\|\func{f}(\vect{u}_1,\mu)-\func{f}(\vect{u}_2,\mu)\|_{\R^n} \le L_R\,\|\vect{u}_1-\vect{u}_2\|_{\R^n}
\quad \text{for all }\ \|\vect{u}_1\|_{\R^n}, \|\vect{u}_2\|_{\R^n} \le R, \ \mu\in I.
\tag{A1}
\label{eq:LipLocUnif} 
\end{equation}
Under these assumptions, by considering the time-dependent function $\func{u}:t\to\func{u}(t)\in\R^n$, the Cauchy problem defined by the autonomous parameter-dependent differential equation
$$
\dot{\func{u}}=\func{f}(\func{u},\mu),
$$
admits a unique local solution, for each fixed parameter $\mu\in I$ and each initial datum, by the Cauchy--Lipschitz theorem~\cite{arnold1992ordinary,rudin1976,wiggins2003}. A key aspect in the analysis of such systems is the identification of their \emph{steady states} $\vect{u}_*\in\R^n$ for any given fixed parameter $\mu_*\in I$, which satisfy
$$
\func{f}(\vect{u}_*,\mu_*) = \vect{0}.
$$
If $\func{f}\in\mathcal{C}^1(\R^n\times I; \R^n)$ and the Jacobian matrix $D_{\vect{u}} \func{f}(\vect{u}_*,\mu_*)$ is non-singular, then the Implicit Function Theorem~\cite{dini1878, rudin1976} yields an open neighborhood $I_{\mu_*}\subset I$ of $\mu_*$ and a unique function $\bar{\func{u}}\in \mathcal{C}^1(I_{\mu_*}; \R^n)$, referred to as a \emph{solution branch}, such that
\begin{equation}
    \func{f}(\bar{\func{u}}(\mu),\mu)=\vect{0} \quad \forall\,\mu\in I_{\mu_*}, \qquad\text{with }\quad \bar{\func{u}}(\mu_*)=\vect{u}_*. \tag{PB1}
    \label{eq:PB1}
\end{equation}
This branch can be continued uniquely as long as the non-degeneracy condition on the Jacobian persists. In general, either $I_{\mu_*}=I$, or at least one endpoint of $I_{\mu_*}$ corresponds to a singularity of the Jacobian matrix $D_{\vect{u}} \func{f}(\bar{\func{u}}(\mu),\mu)$. This typically signals the presence of a \emph{bifurcation}, possibly leading to qualitative changes in the set of steady states~\cite{wiggins2003,Seydel2010,Kuznetsov2023} (e.g.\ loss of uniqueness, changes of stability, or variations in their number).
In this setting, retrieving the solution branches of~\eqref{eq:PB1} over the full interval $I$ by standard numerical continuation~\cite{allgower1990numerical, Seydel2010} is computationally costly, as it requires many-query and large scale simulations, and is likely to suffer from numerical vulnerabilities such as step-size sensitivity, convergence failures at turning points, and undesired branch jumping.
To overcome these limitations, we shift the approach from continuation-based pointwise evaluations to global reconstruction, meaning that we aim at recovering a continuous solution branch over the whole parameter interval $I$. Specifically, we introduce a weak formulation of the problem in the parameter space, allowing us to approximate the solution branches as continuous curves $\bar{\func{u}}\in\mathcal{C}(I; \R^n)$ over the interval $I$, eliminating the need for iterative procedures.

Now, consider a curve in the parameter space $\func{u}\in\mathcal{C}(I;\R^n)$ supported in the whole compact interval $I$, and define the \emph{Nemytskii operator}~\cite{appell1990nonlinear} denoted with $\mathcal{F}:\mathcal{C}(I; \R^n)\to \mathcal{C}(I; \R^n)$ and associated to $\func{f}$ as
$$
\mathcal{F}(\func{u})(\mu) \coloneq \func{f}\bigl(\func{u}(\mu),\mu\bigr), \qquad \mu\in I,
$$
which is well defined by the continuity of $\func{f}$, i.e.\ $\mathcal{F}(\func{u})\in \mathcal{C}(I; \R^n)$ for every $\func{u}\in\mathcal{C}(I; \R^n)$. 
Moreover, under assumption~\eqref{eq:LipLocUnif}, it is locally Lipschitz on $\mathcal{C}(I; \R^n)$ with respect to the sup norm $\|\cdot\|_\infty$, and in particular Lipschitz on bounded sets of $\mathcal{C}(I; \R^n)$ (the proof is reported in Proposition~\ref{prop:LocalLipschitzNemytskii}).

Using the standard $L^2(I; \R^n)$ inner product on $\mathcal{C}(I; \R^n)$, we say that $\bar{\func{u}}\in\mathcal{C}(I; \R^n)$ is a \emph{weak solution} of problem \eqref{eq:PB1} if
\begin{equation}
    \int_I \langle \mathcal{F}(\bar{\func{u}})(\mu), \func{v}(\mu) \rangle_{\R^n}\,\dmu = 0 \qquad \forall\,\func{v}\in\mathcal{C}(I; \R^n).
    \label{eq:WeakFormulation}
\end{equation}
It is straightforward to prove the following equivalence result.

\begin{prop}[Strong--weak equivalence]\label{prop:strong-weak-equivalence}
Let $\bar{\func{u}}\in\mathcal{C}(I; \R^n)$, then
\begin{align*}
\mathcal{F}(\bar{\func{u}})(\mu)=\vect{0} \quad \forall\,\mu\in I
\quad \iff \quad
\int_I \langle \mathcal{F}(\bar{\func{u}})(\mu), \func{v}(\mu) \rangle_{\R^n}\,\dmu = 0 \quad \forall\,\func{v}\in \mathcal{C}(I; \R^n).
\end{align*}
\end{prop}
\begin{proof}
The implication ``$\Rightarrow$'' is immediate. For the converse, assume that $\int_I \langle \mathcal{F}(\bar{\func{u}})(\mu), \func{v}(\mu) \rangle_{\R^n}\,\dmu = 0$ for all $\func{v}\in \mathcal{C}(I; \R^n)$. If $\mathcal{F}(\bar{\func{u}})(\tilde{\mu})\neq \vect{0}$ for some $\tilde{\mu}\in I$, then by continuity there exists an open interval $J\subset I$ containing $\tilde{\mu}$ such that $\mathcal{F}(\bar{\func{u}})\neq\vect{0}$ on $J$. Choosing $\func{v}\coloneq\mathcal{F}(\bar{\func{u}})$ yields
$$
\int_I \|\mathcal{F}(\bar{\func{u}})(\mu)\|_{\R^n}^2\,\dmu>0,
$$
which is a contradiction. Therefore, $\mathcal{F}(\bar{\func{u}})(\mu)=\vect{0}$ for all $\mu\in I$.
\end{proof}

In order to pass from~\eqref{eq:PB1} to its stochastic counterpart, so to perform spectral approximations, we endow the parameter space with a probability measure. Specifically, in this work, we model the parameter as a random variable with finite second moment, namely
\begin{align*}
\mu \in \RV \coloneq \left\{ \mu : \Omega \to I \;\middle|\; \E\big[\big(\mu - \E[\mu]\big)^2\big] < +\infty \right\},
\end{align*}
where $(\Omega, \Sigma, \mathbb{P})$ is the underlying probability space. A random variable $\mu: \Omega \to I$ induces a push-forward probability measure $\nu_\mu$ on the Borel $\sigma$-algebra $\mathcal{B}(I)$, defined by $\nu_\mu(B) \coloneq \mathbb{P}(\mu^{-1}(B))$ for any $B \in \mathcal{B}(I)$, which is the probability law of $\mu$. Let $\bar{\mu} \coloneq \E[\mu] = \int_I \mu(x)\,\mathrm{d}\nu_\mu(x)$ be its expectation and $\sigma \coloneq \E[(\mu-\bar{\mu})^2]^{1/2} = \left( \int_I (\mu(x) - \bar{\mu})^2 \,\mathrm{d}\nu_\mu(x) \right)^{1/2}$ its standard deviation.

To work with normalized random variables, we introduce the re-parametrization $\xi \coloneq \frac{\mu-\bar{\mu}}{\sigma}$, supported on $I_\xi$, and let $\nu_\xi$ denote its probability law. Consequently, the formulation in~\eqref{eq:WeakFormulation} can be equivalently re-written as finding $\tilde{\func{u}}\in\mathcal{C}(I_\xi;\R^n)$ such that
\begin{equation}
    \int_{I_\xi} \langle \widetilde{\mathcal{F}}(\tilde{\func{u}})(z), \tilde{\func{v}}(z) \rangle_{\R^n} \,\mathrm{d}\nu_\xi(z) = 0 \quad \forall \tilde{\func{v}} \in \mathcal{C}(I_\xi; \R^n),
    \label{eq:weak_form_probability}
\end{equation}
where the operator $\widetilde{\mathcal{F}}:\mathcal{C}(I_{\xi};\R^n)\to \mathcal{C}(I_{\xi};\R^n)$ is defined as $\widetilde{\mathcal{F}}(\tilde{\func{u}})(z)\coloneq\func{f}(\tilde{\func{u}}(z),\bar{\mu}+\sigma z)$ for all $z\in I_{\xi}$.

To numerically address the resolution of the problem defined by Equation~\eqref{eq:weak_form_probability}, we employ Polynomial Chaos (PC) expansion. This allows us to approximate functions in the space
\begin{align*}
    L^2_{\xi}(I_\xi;\R^n)\coloneq\Big\{\tilde{\func{v}}:I_\xi\to\R^n\ \big|\ \int_{I_\xi} \|\tilde{\func{v}}(z)\|_{\R^n}^2\,\mathrm{d}\nu_\xi(z) <+\infty \Big\},
\end{align*}
endowed with the inner product $\langle\tilde{\func{u}},\tilde{\func{v}}\rangle_{L^2_{\xi}}=\int_{I_\xi} \langle\tilde{\func{u}}(z),\tilde{\func{v}}(z)\rangle_{\R^n}\,\mathrm{d}\nu_\xi(z)$, through a truncated series of orthogonal polynomials w.r.t.\ the law measure of $\xi$, yielding spectral convergence~\cite{xiu2010, xiu2002wiener, xiu2003modeling, cameron1947orthogonal, ghanem1991stochastic}.

In particular, given a continuous function $\tilde{\func{v}}\in \mathcal{C}(I_\xi; \R^n)\subset L^2_{\xi}(I_\xi;\R^n)$, its PC expansion $\tilde{\func{v}}_{\NPC}$ is given by:
\begin{equation}\label{eq:PC_representation}
\begin{split}
    \tilde{\func{v}}_{\NPC}(\xi)&\coloneq\sum_{k=0}^{\NPC} \hat{\vect{v}}_k\,\Phi_k(\xi)=
    \begin{pmatrix} 
    \sum_{k=0}^{\NPC} \hat{v}_{k,1} \Phi_k(\xi) \\
    \vdots \\
    \sum_{k=0}^{\NPC} \hat{v}_{k,n} \Phi_k(\xi) 
    \end{pmatrix},\\
    \hat{\vect{v}}_k &\coloneq \frac{\langle \tilde{\func{v}}, \Phi_k \rangle_{L^2_{\xi}}}{\|\Phi_k\|_{L^2_{\xi}}^2}=(\hat{v}_{k,1}, \dots, \hat{v}_{k,n}),
\end{split}
\end{equation}
where the inner product $\langle \tilde{\func{v}}, \Phi_k \rangle_{L^2_{\xi}}$ is intended component-wise. The basis functions $\{\Phi_k\}_{k=0}^{+\infty}$ are polynomials of degree $k$ strictly orthogonal w.r.t.\ the law of $\xi$, satisfying $\int_{I_\xi} \Phi_i(z)\Phi_j(z)\,\mathrm{d}\nu_\xi(z)=\delta_{ij}\|\Phi_i\|_{L^2_\xi}^2$ for all $i,j$.

In this work, we investigate the properties of the solutions to the weak problem in \eqref{eq:weak_form_probability} restricting the functions $\func{u},\func{v}$ to the subspace spanned by the truncated PC expansion.
Specifically, for a fixed integer $\NPC > 0$, we take the solution space to be the finite-dimensional polynomial subspace
$$
\Pi_{\NPC}^n \coloneq \left(\mathrm{span}\{\Phi_0, \dots, \Phi_{\NPC}\}\right)^n \subset L^2_{\xi}(I_{\xi}; \R^n),
$$
and, by a Galerkin projection, we impose that the residual is orthogonal to $\Pi_{\NPC}^n$. 
Thus, the problem is reduced to finding $\tilde{\func{u}}_{\NPC}\in\Pi_{\NPC}^n$ such that
\begin{equation}
\tag{PB2}\label{eq:PB2}
    \int_{I_\xi} \langle \widetilde{\mathcal{F}}(\tilde{\func{u}}_{\NPC})(z), \tilde{\func{v}}_{\NPC}(z) \rangle_{\R^n} \,\mathrm{d}\nu_\xi(z) =0,\qquad\forall \tilde{\func{v}}_{\NPC}\in\Pi_{\NPC}^n\subset\mathcal{C}(I_{\xi}; \R^n).
\end{equation}
Note that, as $\NPC\to\infty$, the infinite-dimensional weak problem \eqref{eq:weak_form_probability} is recovered.

Finally, problem~\eqref{eq:PB2} corresponds to a system of $n(\NPC+1)$, possibly non-linear, equations. 
Defining the operator $\func{F}_{\NPC}:\R^{n(N_{PC}+1)}\to\R^{n(N_{PC}+1)}$ in the unknown vector $\vect{\hat{u}}$, we obtain the compact block-vector form:
\begin{equation}
\func{F}_{\NPC}(\vect{\hat{u}})\coloneq
\begin{pmatrix}
\langle \widetilde{\mathcal{F}}(\tilde{\func{u}}_{\NPC}), \Phi_0 \rangle_{L^2_{\xi}}\\
\vdots\\
\langle \widetilde{\mathcal{F}}(\tilde{\func{u}}_{\NPC}), \Phi_{\NPC} \rangle_{L^2_{\xi}}
\end{pmatrix}
=
\begin{pmatrix}
\vect{0}\\
\vdots\\
\vect{0}
\end{pmatrix},
\label{eq:GalerkinSystem}
\end{equation}
where the $k$-th block $\langle \widetilde{\mathcal{F}}(\tilde{\func{u}}_{\NPC}), \Phi_k \rangle_{L^2_{\xi}} \in \R^n$ represents the projection of the $n$-dimensional residual onto the $k$-th basis polynomial.

\begin{rmk}
Notice that PCE is well suited for this kind of investigation compared to other stochastic representations. For example, using the Karhunen--Lo\`eve (KL)~\cite{ghanem1991stochastic, xiu2010} expansion requires knowing the eigenfunctions of the output covariance operator, and thus prior knowledge of the unknown bifurcating behavior, whereas PCE relies only on the orthogonal basis determined by the law $\nu_\xi$.
\end{rmk}

In the following, we will consider uniformly distributed parameters $\mu\sim\mathcal{U}(a,b)$ in $I\coloneqq[a,b]$, so that $\xi$ is uniformly distributed\footnote{The construction can be extended to more general random variables~\cite{xiu2002wiener, xiu2010}, possibly with unbounded support.} on $I_{\xi}\coloneqq[-\sqrt{3},\sqrt{3}]$.
In this setting, the PC orthogonal polynomials $\Phi_k$ are suitably scaled Legendre polynomials on $I_{\xi}$.

The PCE-Galerkin system~\eqref{eq:GalerkinSystem} is highly nonlinear and may admit a large family of solutions. The following remark serves as a motivation for our work, highlighting the need to handle this vast set of algebraic roots and identifying the specific ones we aim to isolate.

\begin{rmk}
Let us consider the case where $n=1$ and $f(u,\mu)$ is a polynomial of degree $d$ in the first variable, i.e.
$$
f(u,\mu)=\sum_{i=0}^{d}\theta_i(\mu)\,u^i,\qquad \theta_i\in\mathcal{C}(I;\R).
$$
We aim to characterize the solutions of the system in Equation~\eqref{eq:GalerkinSystem}, whose $k$-th equation reads as:
\begin{equation*}
    \langle \widetilde{\mathcal{F}}(\tilde{u}_{\NPC}),\Phi_k \rangle_{L^2_{\xi}}=\sum_{i=0}^{d} \sum_{j_1, \dots, j_i = 0}^{\NPC} A^{(k)}_{j_1,\dots,j_i}\,\hat{u}_{j_1}\cdots \,\hat{u}_{j_i}=0,
\end{equation*}
where the coefficients $A^{(k)}_{j_1,\dots,j_i}$ for $k=0, \dots, \NPC$ are defined by
$$
A^{(k)}_{j_1,\dots,j_i}=\int_{I_\xi} \theta_i(\bar{\mu}+\sigma z)\,\Phi_{j_1}(z)\cdots\Phi_{j_i}(z)\,\Phi_k(z)\,\mathrm{d}\nu_\xi(z).
$$

The system consists of $\NPC+1$ polynomial equations, each of these of degree up to $d$ in the unknown $\bm{\hat{u}}$. 
Assuming that equations in~\eqref{eq:GalerkinSystem} define algebraic curves without common components, since each of the equations has at most degree $d$, the system admits at most $d^{\NPC+1}$ real solutions, by the ``real affine'' version of Bézout's Theorem~\cite{Fulton1998, SommeseWampler2005} reported in Appendix~\ref{thm:Bezout}.
Consequently, the number of solutions of the PCE-Galerkin system exceeds the one of the original problem, growing exponentially in $\NPC$.

A canonical example within this class of bifurcating systems is the pitchfork bifurcation, which is characterized by a single equilibrium branch that splits into three at a critical parameter value. Its normal form, defined as the simplest polynomial model exhibiting this qualitative behavior, is described by the equation
\begin{align*}
    f(u,\mu) = -u^3+\mu u,
\end{align*}
whose bifurcation point is $(u,\mu)=(0,0)$ and its equilibrium branches are shown alongside their PC approximations in Figure~\ref{fig:oscillating_vs_branch}.

We observe that the majority of PC solutions oscillates between different branches of the bifurcation diagram. 
However, a smaller set of solutions seems to have a branch-approximating behavior, accurately matching the equilibria.

\begin{figure}[t]
    \centering
    \includegraphics[width=0.48\textwidth]{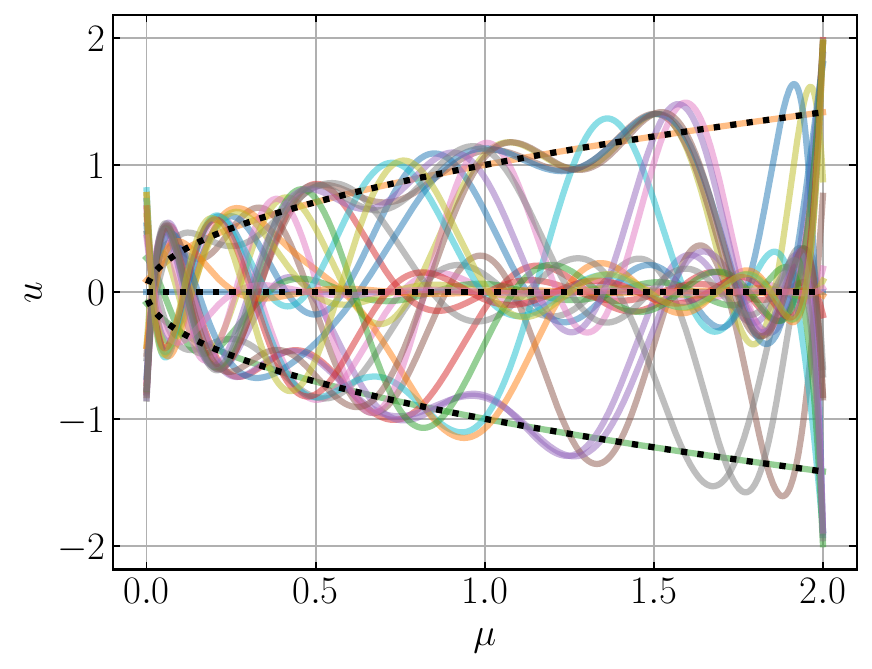}
    \includegraphics[width=0.48\textwidth]{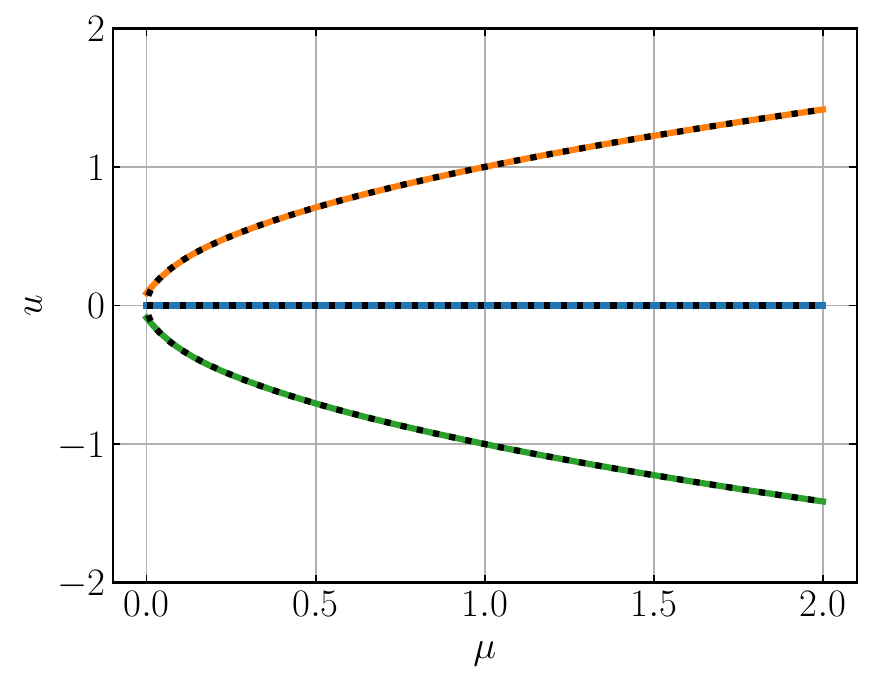}
    \caption{Oscillating and branch-approximating solutions for the pitchfork bifurcation, left and right respectively, for $\NPC=10$, $\bar{\mu}=1.0$, and $\sigma=1.0$.}
    \label{fig:oscillating_vs_branch}
\end{figure}
\end{rmk}

Motivated by this, in the next section we investigate the analytical properties of the solutions to problem \eqref{eq:PB2} which exhibits a \textit{branch--approximating} behavior, proving them to be a very powerful tool.
\section{Theoretical Results}\label{sec:theoretical-results}

In this section, we establish rigorous convergence results and error estimates for the \emph{branch approximating} solutions to problem \eqref{eq:PB2}. 
Let us suppose that there exist $r$ distinct solution branches $\bar{\func{u}}^{(1)},\dots,\bar{\func{u}}^{(r)}\in\mathcal{C}(I; \R^n)$ of~\eqref{eq:PB1} within the domain $I$, and that they are of class $\mathcal{C}^1$ away from bifurcation points. We remark that, while they may lose differentiability at critical parameter values, continuity is typically preserved. 

In the following, for the sake of readability, we drop the tilde notation since we perform our analysis exclusively in the re-parametrized setting over $I_\xi$. Therefore, $\bar{\func{u}}^{(1)},\dots,\bar{\func{u}}^{(r)}\in\mathcal{C}(I_{\xi};\R^n)$ and $\mathcal{F}:\mathcal{C}(I_{\xi};\R^n)\to \mathcal{C}(I_{\xi};\R^n)$ will denote the re-parametrized branches and Nemytskii operator, respectively. 
Moreover, denoting by $N$ the degree of the PC expansion, we refer to $\func{u}^{(k)}_{LS,N}$ as the best $N$-degree polynomial approximation in $L^2_{\xi}(I_\xi; \R^n)$ of the $k$-the branch, and to $\func{u}_N$ as a generic solution to~\eqref{eq:PB2}, i.e.
$$
\func{u}_{N}(\xi)=\sum_{j=0}^N\hat{\vect{u}}_{N,j}\,\Phi_j(\xi)\quad\text{s.t. }\int_{I_\xi} \langle \mathcal{F}(\func{u}_N)(z), \func{v}_N(z) \rangle_{\R^n} \,\mathrm{d}\nu_\xi(z) =0,\ \forall \func{v}_N\in\Pi_{N}^n.
$$
For the remainder of this section, we focus on a single solution branch of the original problem~\eqref{eq:PB1}, as the analysis can be carried out independently for each branch, and thus we fix
$$
\bar{\func{u}}\coloneq \bar{\func{u}}^{(k)},\qquad \func{u}_{LS,N}\coloneq \func{u}^{(k)}_{LS,N}.
$$
Note that, since in the uniform case the probability law of $\xi$ is equivalent to the Lebesgue measure on $I_\xi$ up to a multiplicative constant, i.e.
$$
\langle \func{u},\func{v}\rangle_{L^2_{\xi}}=\frac{1}{2\sqrt{3}}\int_{I_{\xi}}\langle \func{u}(z),\func{v}(z)\rangle_{\R^n}\mathrm{d}z= \frac{1}{2\sqrt{3}}\langle \func{u},\func{v}\rangle_{L^2},
$$
then the inner product $\langle\cdot,\cdot\rangle_{L^2_{\xi}}$ is equivalent to the standard $L^2(I_\xi; \R^n)$ inner product, which we will use in the following. For simplicity of notation, we may omit the domain and codomain, referring to the spaces as $L^2$ and $L^\infty$. 
Specifically, for a vector-valued function $\func{v} \in L^2(I_\xi; \R^n) \cap L^\infty(I_\xi; \R^n)$, its functional norms are defined through the standard Euclidean norm $\|\cdot\|_{\R^n}$:
$$
\|\func{v}\|_{L^2} \coloneq \left( \int_{I_\xi} \|\func{v}(z)\|_{\R^n}^2 \,\mathrm{d}z \right)^{1/2}, \qquad \|\func{v}\|_{L^\infty} \coloneq \operatorname{ess\,sup}_{z \in I_\xi} \|\func{v}(z)\|_{\R^n}.
$$
Similarly, the $H^s$ Sobolev norms for vector-valued functions are defined by applying the standard scalar $H^s$ norm definitions to the Euclidean norm of the function and its weak derivatives.

To simplify the reading process across the theoretical results and numerical experiments, we collect here the main assumptions employed in the paper. While uniform local Lipschitzianity~\eqref{eq:LipLocUnif}, introduced in Section~\ref{sec:ode-bifurcations}, is assumed throughout the manuscript (also for $D_{\vect{u}}\func{f}$ whenever $\func{f}$ is differentiable with respect to $\vect{u}$), the following hypotheses are used to prove consistency (Theorem~\ref{thm:consistency}), convergence (Theorem~\ref{thm:existence_of_PC_ls_sol}) and uniqueness (Theorem~\ref{thm:uniqueness}).

First, the target branch must satisfy the Sobolev regularity conditions:
\begin{empheq}[left={\bar{\func{u}} \in H^s(I_\xi; \R^n), \quad \text{with} \quad \empheqlbrace}]{align*}
s &\ge 1 \quad \text{for Theorem~\ref{thm:consistency}}, \tag{A2a}\label{ass:A2a} \\
s &> 1 \quad \text{for Theorems~\ref{thm:existence_of_PC_ls_sol}--\ref{thm:uniqueness}}. \tag{A2b}\label{ass:A2b}
\end{empheq}

Moreover, given a symmetric positive-definite matrix $P \in \R^{n \times n}$ inducing the weighted inner product $\langle \vect{v}, \vect{w} \rangle_P \coloneq \vect{v}^\top P \vect{w}$ and norm $\|\vect{w}\|_P \coloneq \sqrt{\langle \vect{w}, \vect{w} \rangle_P}$, we assume a local non-degeneracy condition along the branch for convergence and uniqueness Theorems~\ref{thm:existence_of_PC_ls_sol}--\ref{thm:uniqueness}: there exists $\lambda_{\min} > 0$ such that the Jacobian is either uniformly positive definite,
\begin{equation}
    \langle D_{\vect{u}}\func{f}(\bar{\func{u}}(z), \bar{\mu}+\sigma z) \vect{w}, \vect{w} \rangle_P \geq \lambda_{\min} \|\vect{w}\|_P^2, \quad \forall z \in I_\xi, \;\forall \vect{w} \in \R^n,
    \tag{A3a}\label{eq:uniform_positive_definiteness}
\end{equation}
or uniformly negative definite,
\begin{equation}
    \langle D_{\vect{u}}\func{f}(\bar{\func{u}}(z), \bar{\mu}+\sigma z) \vect{w}, \vect{w} \rangle_P \leq -\lambda_{\min} \|\vect{w}\|_P^2, \quad \forall z \in I_\xi, \;\forall \vect{w} \in \R^n.
    \tag{A3b}\label{eq:uniform_negative_definiteness}
\end{equation}

Finally, concerning Theorem~\ref{thm:uniqueness}, we introduce two asymptotic geometric constraints under Assumption~\eqref{eq:uniform_positive_definiteness} with respect to the same $P$-inner product: for any $\rho > 0$, there exist constants $m_\rho > 0$ and $C_\rho > 0$ such that, for all $\vect{u} \in \R^n$ with $\|\vect{u} - \bar{\func{u}}(z)\|_P \geq \rho$ and all $z \in I_\xi$, the vector field satisfies the uniform radial coercivity condition and the relative growth bound\footnote{If~\eqref{eq:uniform_negative_definiteness} holds instead of~\eqref{eq:uniform_positive_definiteness}, the inequalities are reversed simultaneously: $\langle \func{f}(\vect{u},\bar{\mu}+\sigma z), \vect{u} - \bar{\func{u}}(z) \rangle_P \leq -m_\rho \|\vect{u} - \bar{\func{u}}(z)\|_P$, and $\|\func{f}(\vect{u},\bar{\mu}+\sigma z)\|_P \leq -C_\rho \langle \func{f}(\vect{u},\bar{\mu}+\sigma z), \vect{u} - \bar{\func{u}}(z) \rangle_P$.}
\begin{gather*}
\langle \func{f}(\vect{u},\bar{\mu}+\sigma z), \vect{u} - \bar{\func{u}}(z) \rangle_P \geq m_\rho \|\vect{u} - \bar{\func{u}}(z)\|_P, \tag{A4a}\label{ass:A4a}\\
\|\func{f}(\vect{u},\bar{\mu}+\sigma z)\|_P \leq C_\rho \langle \func{f}(\vect{u},\bar{\mu}+\sigma z), \vect{u} - \bar{\func{u}}(z) \rangle_P. \tag{A4b}\label{ass:A4b}
\end{gather*}

\subsection{Consistency}
We start by presenting a consistency result, showing that the (re-parametrized) solution branch $\bar{\func{u}}$ becomes, in a suitable sense, an asymptotic solution of~\eqref{eq:PB2} as $N\to+\infty$. More precisely, we study the behavior of the least-squares approximation residual in the PCE-Galerkin system~\eqref{eq:GalerkinSystem} and show that it converges to zero as $N\to+\infty$, with rates consistent with the regularity of the branch (see Theorem~\ref{thm:convergence_ls} in Appendix \ref{appendix}).

\begin{thm}[Consistency]\label{thm:consistency}
Let $\bar{\func{u}} \in H^s(I_\xi; \R^n)$ with $s\ge1$~\eqref{ass:A2a}.
Under the local Lipschitz assumptions~\eqref{eq:LipLocUnif} on $\mathcal{F}$, there exist two constants $C_2, C_{\infty}>0$, independent of $N$, such that the residual of the least-squares approximation $\func{u}_{LS,N}$ satisfies the strong error bounds:
\begin{gather}
\forall\,\delta>0,\quad\|\mathcal{F}(\func{u}_{LS,N})\|_{L^{\infty}} \leq C_{\infty}\|\bar{\func{u}}-\func{u}_{LS,N}\|_{L^{\infty}}=\mathcal{O}(N^{\frac{3}{4}+\delta-s})\label{eq:Linfresidual},\\
\|\mathcal{F}(\func{u}_{LS,N})\|_{L^2} \leq C_2\|\bar{\func{u}}-\func{u}_{LS,N}\|_{L^2}=\mathcal{O}(N^{-s}).\label{eq:L2residual} 
\end{gather}
In particular, $\func{u}_{LS,N}$ asymptotically satisfies the PCE-Galerkin system, i.e.\ for all $l\in\N$ it holds
$$
\lim_{N\to+\infty}\langle \mathcal{F}(\func{u}_{LS,N}),\Phi_l\rangle_{L^2} = \vect{0}.
$$
\end{thm}

\begin{proof}
As a first step, we prove that the sequence of projections $\{\func{u}_{LS,N}\}_{N\in\mathbb{N}}$ converges to $\bar{\func{u}}$ uniformly on the interval $I_\xi$. To do so, we rely on standard spectral and polynomial approximation theory in Sobolev spaces~\cite{canuto2006spectral, shen2011spectral}. The truncation errors for the $L^2$-orthogonal projection, when evaluated in the $L^2$ and $H^1$ norms, can be bounded as
\begin{gather*}
\|\bar{\func{u}}-\func{u}_{LS,N}\|_{L^2}\leq C_0N^{-s}\|\bar{\func{u}}\|_{H^s},\\
\|\bar{\func{u}}-\func{u}_{LS,N}\|_{H^1}\leq C_1N^{\frac{3}{2}-s}\|\bar{\func{u}}\|_{H^s},
\end{gather*}
for some positive constants $C_0, C_1$ independent of $N$. Notice that the $H^1$ estimate exhibits a loss of half an order of convergence, which is due to the fact that the projection is orthogonal in $L^2$ and not in $H^1$.

Next, we establish uniform convergence by bounding the error in $L^\infty(I_\xi; \R^n)$. Given that the $H^1$ error may not provide a sufficiently tight bound for the supremum norm, we appeal to an intermediate regularity. Since the parameter domain $I_\xi$ is one-dimensional, the Sobolev embedding~\cite{adams2003sobolev} $H^k(I_\xi; \R^n) \hookrightarrow L^\infty(I_\xi; \R^n)$ holds for any fractional order $k > 1/2$, regardless of the state space dimension $n$. Consequently, we employ the Gagliardo--Nirenberg interpolation inequality~\cite{adams2003sobolev, bergh1976interpolation, tartar2007sobolev, brenner2008mathematical, brezis2018gagliardo} for $k \in (1/2, 1)$, which yields
\begin{equation*}
    \| \bar{\func{u}} - \func{u}_{LS,N} \|_{H^k} \leq C_k \| \bar{\func{u}} - \func{u}_{LS,N} \|_{H^1}^k \| \bar{\func{u}} - \func{u}_{LS,N} \|_{L^2}^{1-k}.
\end{equation*}
Substituting the previously established polynomial error bounds into this interpolation inequality, we obtain
$$
\|\bar{\func{u}}-\func{u}_{LS,N}\|_{H^k}\leq C_k\left(C_1 N^{\frac{3}{2}-s}\right)^k\left(C_0 N^{-s}\right)^{1-k}\|\bar{\func{u}}\|_{H^s}=\widehat{C}N^{\frac{3}{2}k-s}\|\bar{\func{u}}\|_{H^s},
$$
where the new constant $\widehat{C}$ is the product of all the constants independent of $N$.

To ensure both the Sobolev embedding, the interpolation inequality and the asymptotic decay of the $H^k$ error, we require $\frac{1}{2} < k < \min\left(1, \frac{2}{3}s\right)$. Since $s \ge 1$ by assumption, we have $\frac{2}{3}s \ge \frac{2}{3}$, so we can choose $k = \frac{1}{2} + \frac{2}{3}\delta$ for any sufficiently small $\delta \in \left(0, \frac{1}{4}\right)$. Substituting this into the $H^k$ estimate yields
$$
\|\bar{\func{u}}-\func{u}_{LS,N}\|_{L^\infty}=\mathcal{O}\left(N^{\frac{3}{4}+\delta-s}\right).
$$
Since $s \ge 1$ and $\delta < \frac{1}{4}$, the exponent is strictly negative, proving uniform convergence and ensuring that $\{\func{u}_{LS,N}\}_{N\in\mathbb{N}}$ is uniformly bounded in $L^\infty(I_\xi; \R^n)$.
Hence, there exists a radius $R>0$, independent of $N$, such that both $\bar{\func{u}}(z)$ and $\func{u}_{LS,N}(z)$ lie within the closed ball $\bar{B}_R(\vect{0}) \subset \R^n$ for all $z \in I_\xi$ and for all $N \in \mathbb{N}$.
Since $\func{f}(\func{u},\mu)$ is locally Lipschitz continuous in its first argument, uniformly with respect to the parameter $\mu$, there exists a local Lipschitz constant $L_R>0$ (see Proposition~\ref{prop:LocalLipschitzNemytskii}) such that
$$
\|\mathcal{F}(\func{v}_1)-\mathcal{F}(\func{v}_2)\|_{L^\infty}\leq L_R\|\func{v}_1-\func{v}_2\|_{L^\infty}, \qquad \forall\,\func{v}_1,\func{v}_2\; \text{s.t.}\; \|\func{v}_1\|_{L^\infty}, \|\func{v}_2\|_{L^\infty}\le R.
$$
Applying this property to the functions $\bar{\func{u}}, \func{u}_{LS,N}$, and recalling that $\mathcal{F}(\bar{\func{u}})\equiv\func{0}$, we finally obtain the first bound of the theorem, setting $C_\infty = L_R$, as
$$
\|\mathcal{F}(\func{u}_{LS,N})\|_{L^\infty}\leq L_R\|\bar{\func{u}}-\func{u}_{LS,N}\|_{L^\infty}=\mathcal{O}(N^{\frac{3}{4}+\delta-s}).
$$

Similarly, we square both sides of the pointwise Lipschitz inequality~\eqref{eq:LipLocUnif} and integrate over $I_\xi$ with respect to the probability law of $\xi$, and obtain
$$
\int_{I_\xi}\|\mathcal{F}(\func{u}_{LS,N})(z)\|_{\R^n}^2\,\mathrm{d}z\leq L_R^2\int_{I_\xi}\|\bar{\func{u}}(z)-\func{u}_{LS,N}(z)\|_{\R^n}^2\,\mathrm{d}z<+\infty.
$$
Taking the square root of this integral expression yields the corresponding bound in the $L^2$ norm
$$
\|\mathcal{F}(\func{u}_{LS,N})\|_{L^2}\leq L_R\|\bar{\func{u}}-\func{u}_{LS,N}\|_{L^2}=\mathcal{O}(N^{-s}),
$$
which establishes the second inequality of the theorem with $C_2=L_R$. 

Finally, we address the weak convergence of the residual. Using the Cauchy--Schwarz inequality for vector-valued integrals, the Euclidean norm of the weak projection along any scalar basis function $\Phi_l$ is bounded by the $L^2$ norms
$$
\|\langle\mathcal{F}(\func{u}_{LS,N}),\Phi_l\rangle_{L^2}\|_{\R^n}\leq\|\mathcal{F}(\func{u}_{LS,N})\|_{L^2}\|\Phi_l\|_{L^2}.
$$
Using the $L^2$ bound just proved, and since the Legendre polynomials satisfy $\|\Phi_l\|_{L^2}\sim l^{-1/2}$, the asymptotic result follows as
$$
\lim_{N\to+\infty}\langle\mathcal{F}(\func{u}_{LS,N}),\Phi_l\rangle_{L^2}=\vect{0}\qquad\forall l\in\mathbb{N}.
$$
\end{proof}

\begin{rmk}\label{rmk:superconvergence}
Although Theorem~\ref{thm:consistency} requires the assumption $\bar{\func{u}} \in H^s(I_\xi; \R^n)$ with $s \ge 1$ to apply the $H^1$ error bound, the Sobolev embedding $H^k(I_\xi; \R^n) \hookrightarrow L^\infty(I_\xi; \R^n)$ and the convergence condition $\frac{1}{2} < k < \frac{2}{3}s$ only require $s > \frac{3}{4}$. Moreover, empirical evidence often shows uniform convergence even for functions with lower fractional regularity ($0 < s < 1$), as we report in Section~\ref{subsec:pitchfork} for the pitchfork bifurcation.

This phenomenon is theoretically justified when the lack of regularity is due to specific algebraic singularities. Recent results in approximation theory demonstrate that functional bounds such as those just developed overestimate the error for functions with localized fractional smoothness~\cite{wang2021how}. 

For example, if $n=1$ and the exact solution exhibits an interior singularity of fractional order, such as
$\bar{u}(x)=|x-x_0|^\alpha$ with $\alpha>0$, the $L^\infty$-norm of the Legendre projection error actually decays at the optimal rate of $\mathcal{O}(N^{-\alpha})$. Furthermore, if the fractional singularity is located at the endpoints of the domain, e.g., $\bar{u}(x)=(1\pm x)^\alpha$ in the interval $I=[-1,1]$, the uniform error may decay even faster, at a rate of $\mathcal{O}(N^{-2\alpha})$. 

Consequently, the $L^2$-orthogonal projection can converge uniformly at a rate matching the best possible polynomial approximation. This explains the empirically observed stability and weak convergence of the residual, despite the exact solution failing the strict $s \ge 1$ Sobolev condition required for the interpolation bounds.
\end{rmk}

\subsection{Convergence}
Motivated by the previous consistency result, we now address the existence of PCE-Galerkin solutions that approximate the branch. Since, for sufficiently large $N$, the least-squares approximation $\func{u}_{LS,N}$ makes the Galerkin residual arbitrarily small, it is natural to seek a sequence of $\func{u}_N$ such that the PCE-Galerkin system $\func{F}_N(\func{u}_N)=\vect{0}\ \forall N$, which converges to $\func{u}_{LS,N}$ for $N\to+\infty$, and hence to $\bar{\func{u}}$.

Assuming $\func{f}$ is differentiable with respect to its first argument, we introduce the Fr\'echet derivative of the Nemytskii operator as the operator $T_{\func{u}}\coloneq D\mathcal{F}(\func{u}): L^2(I_\xi; \R^n) \to L^2(I_\xi; \R^n)$ acting as a linear matrix multiplication operator:
$$
T_{\func{u}}(\func{v})(z) \coloneq D\mathcal{F}(\func{u})[\func{v}](z)=D_{\vect{u}}\func{f}(\func{u}(z), \bar{\mu}+\sigma z) \func{v}(z), \quad \forall z \in I_\xi.
$$
Then, suppose that the Jacobian matrix evaluated along the branch satisfies the uniform positive definiteness condition~\eqref{eq:uniform_positive_definiteness} in the standard Euclidean inner product, with $P = I_n$\footnote{In the case of $n=1$, a continuous non-vanishing derivative maintains a constant sign, ensuring this property. For $n \ge 2$, uniform positive definiteness is required to preserve the coercivity of the operator through the Galerkin projection.}. Note that all invertibility and convergence proofs in this section hold identically under the uniform negative definiteness condition~\eqref{eq:uniform_negative_definiteness}.

This condition strictly excludes bifurcation points inside the interval\footnote{Since $\mathcal{F}(\bar{\func{u}})=\vect{0}$, this positive definiteness ensures the Fréchet derivative $D\mathcal{F}(\bar{\func{u}})=T_{\bar{\func{u}}}$ is strictly monotone and therefore continuously invertible, satisfying the key hypothesis of the Implicit Function Theorem.}. 
By the Cauchy--Schwarz inequality, equation~\eqref{eq:uniform_positive_definiteness} ensures that the operator is uniformly bounded away from zero. It follows from standard operator theory~\cite{rudin1991functional} that $T_{\bar{\func{u}}}$ is invertible with the following estimates:
$$
\|T_{\bar{\func{u}}}\func{v}\|_{L^2} \geq \lambda_{\min} \|\func{v}\|_{L^2} \quad \forall \func{v} \in L^2(I_\xi; \R^n), \qquad \|T_{\bar{\func{u}}}^{-1}\|_{\mathcal{L}(L^2, L^2)} \leq \frac{1}{\lambda_{\min}}.
$$
To prove the convergence, we aim to apply the Newton--Kantorovich Theorem (see Appendix \ref{thm:NewtonKantorovich}) to a family of solutions $\func{u}_N$ beyond a fixed value $N_{\text{conv}}$. Thus, we now seek to exploit the invertibility of $T_{\bar{\func{u}}}$ to obtain the required uniform bound on the norm of the Jacobian matrices of the PCE-Galerkin system. 
Therefore, recalling the formulation of~\eqref{eq:PB2} given in Equation~\eqref{eq:GalerkinSystem}, we connect it to the properties of the infinite-dimensional operator $T_{\bar{\func{u}}}$, and we cast it as a root-finding problem in $\R^{n(N+1)}$.

To formalize this connection, we introduce two bridging operators. The first is the natural isomorphism $\mathcal{I}_n: \R^{n(N+1)} \to \Pi_N^n$, which maps a vector of coefficients to its corresponding polynomial in the finite-dimensional functional space $\Pi_N^n \subset L^2(I_\xi; \R^n)$ as
$$
    \mathcal{I}_n(\hat{\vect{u}}) \coloneq \sum_{j=0}^N \hat{\vect{u}}_j \Phi_j(\xi).
$$
The second is the standard $L^2$-orthogonal projection $\mathcal{P}_N: L^2(I_\xi; \R^n) \to \Pi_N^n$, which maps any square-integrable function to its best polynomial approximation in $\Pi_N^n$
$$
    \mathcal{P}_N(\func{v}) \coloneq \sum_{i=0}^N \frac{\langle \func{v}, \Phi_i \rangle_{L^2}}{\|\Phi_i\|_{L^2}^2} \Phi_i(\xi).
$$
We define, with a slight abuse of notation, a scaled equivalent Galerkin operator $\func{F}_N:\R^{n(N+1)}\to\R^{n(N+1)}$ as
$$
    \func{F}_N(\hat{\vect{u}})_i \coloneq \frac{\langle \mathcal{F}(\mathcal{I}_n(\hat{\vect{u}})), \Phi_i \rangle_{L^2}}{\|\Phi_i\|_{L^2}^2}, \quad i=0,\dots,N,
$$
which does not change the roots of the system in Equation~\eqref{eq:GalerkinSystem}, but it ensures that $\func{F}_N(\hat{\vect{u}})$ exactly matches the coefficients of its orthogonal projection on $\{\Phi_i\}_{i=0}^N$. In fact, denoting by $\mathcal{I}_n^{-1}: \Pi_N^n \to \R^{n(N+1)}$ the inverse mapping that extracts the basis coefficients of a given polynomial, we can write $\func{F}_N$ in operator form as
$$
    \func{F}_N(\hat{\vect{u}}) = \mathcal{I}_n^{-1} \left( \mathcal{P}_N(\mathcal{F}(\mathcal{I}_n(\hat{\vect{u}}))) \right).
$$
Consequently, its Jacobian matrix $\vect{JF}_N(\hat{\vect{u}}) \in \R^{n(N+1)\times n(N+1)}$ is the projection of the continuous Fr\'echet derivative $T_{\func{u}_N} \coloneq D\mathcal{F}(\func{u}_N)$ translated into the coefficient space
\begin{equation*}
    \vect{JF}_N(\hat{\vect{u}}) = \mathcal{I}_n^{-1} \circ \mathcal{P}_N \circ T_{\mathcal{I}_n(\hat{\vect{u}})} \circ \mathcal{I}_n,\quad\text{so that}\quad\left(\vect{JF}_N(\hat{\vect{u}})\right)_{ij}=\frac{\langle D_{\vect{u}}\func{f}(\mathcal{I}_n(\hat{\vect{u}}),\cdot)\Phi_j,\Phi_i\rangle_{L^2}}{\|\Phi_i\|_{L^2}^2},
\end{equation*}
where each $\left(\vect{JF}_N(\hat{\vect{u}})\right)_{ij}$ represents an $n \times n$ block of the full Jacobian matrix.

We equip $\R^{n(N+1)}$ with a weighted inner product and its induced norm
\begin{equation*}
    \langle \vect{a}, \vect{b} \rangle_* \coloneq \sum_{i=0}^N \langle \vect{a}_i, \vect{b}_i \rangle_{\R^n} \|\Phi_i\|_{L^2}^2, \qquad |\vect{b}|_* \coloneq \sqrt{\langle \vect{b}, \vect{b} \rangle_*},\qquad\forall\vect{a},\vect{b}\in\R^{n(N+1)},
\end{equation*}
that, by Parseval's identity, is isometric to the $L^2$ norm in the function space, i.e. $|\vect{b}|_* = \|\mathcal{I}_n(\vect{b})\|_{L^2}$ for all $\vect{b} \in \R^{n(N+1)}$. Moreover, for any $\func{w} \in L^2(I_\xi; \R^n)$ and $\vect{b} \in \R^{n(N+1)}$, this inner product provides a crucial bridge to the functional space:
\begin{equation}\label{eq:bridge_inner}
\begin{split}
    \langle \mathcal{I}_n^{-1}(\mathcal{P}_N(\func{w})), \vect{b} \rangle_* &= \sum_{i=0}^N \left\langle \frac{\langle \func{w}, \Phi_i \rangle_{L^2}}{\|\Phi_i\|_{L^2}^2}, \vect{b}_i \right\rangle_{\R^n} \|\Phi_i\|_{L^2}^2 \\
    &= \left\langle \func{w}, \sum_{i=0}^N \vect{b}_i \Phi_i \right\rangle_{L^2} = \langle \func{w}, \mathcal{I}_n(\vect{b}) \rangle_{L^2}.
\end{split}
\end{equation}
Thanks to this isometric norm framework, we can now prove the bounds on the Jacobian matrix required by the Newton--Kantorovich Theorem.

\begin{prop} 
Suppose that $\bar{\func{u}} \in H^s(I_\xi; \R^n)$ for $s \geq 1$~\eqref{ass:A2a} and that the Jacobian matrix satisfies the uniform positive definiteness condition~\eqref{eq:uniform_positive_definiteness} with $P = I_n$. Assume that the pointwise derivative $D\mathcal{F}(\func{u})$ satisfies the uniform local Lipschitz condition~\eqref{eq:LipLocUnif} with respect to $\func{u}$. Then, there exist an integer $N_0 > 0$ and an $L^\infty$-neighborhood $U$ of $\bar{\func{u}}$ such that for any $N \geq N_0$, the Jacobian evaluated at the least-squares approximation coefficients $\vect{a} = \mathcal{I}_n^{-1}(\func{u}_{LS,N})$ is invertible, and there exists a constant $K_{\bar{\func{u}}}(N) = \mathcal{O}(N)$ satisfying the following Lipschitz continuity bound in the $|\cdot|_*$ induced matrix norm:
\begin{gather*}
|(\vect{JF}_N(\vect{a}))^{-1}|_* \leq \frac{2}{\lambda_{\min}}, \\
|\vect{JF}_N(\hat{\vect{u}}_1) - \vect{JF}_N(\hat{\vect{u}}_2)|_* \leq K_{\bar{\func{u}}}(N)|\hat{\vect{u}}_1 - \hat{\vect{u}}_2|_* \quad \forall\hat{\vect{u}}_1, \hat{\vect{u}}_2\; \textnormal{ s.t. }\; \mathcal{I}_n(\hat{\vect{u}}_1), \mathcal{I}_n(\hat{\vect{u}}_2) \in U.
\end{gather*}
\end{prop}

\begin{proof}
Since in the proof of Theorem~\ref{thm:consistency} we show that $\func{u}_{LS,N}\to\bar{\func{u}}$ uniformly as $N\to\infty$, and by the continuity of the Jacobian matrix, we can choose $N_0$ large enough such that $\func{u}_{LS,N}$ falls into a sufficiently small neighborhood of $\bar{\func{u}}$. This guarantees that the strict positive definiteness is locally preserved, yielding
$$
\langle D_{\vect{u}}\func{f}(\func{u}_{LS,N}(z),\bar{\mu}+\sigma z)\vect{w}, \vect{w} \rangle_{\R^n} \geq \frac{\lambda_{\min}}{2} \|\vect{w}\|_{\R^n}^2
$$
uniformly on $I_\xi$ for all $N \geq N_0$ and all $\vect{w} \in \R^n$.

To bound the inverse of $\vect{JF}_N(\vect{a})$, we evaluate its coercivity using the $*$-inner product. Introducing $\vect{b}\in\R^{n(N+1)}$ and $\func{v}_N=\mathcal{I}_n(\vect{b})$, by using the isometric identity~\eqref{eq:bridge_inner}, we have
\begin{align*}
    \langle \vect{JF}_N(\vect{a})\vect{b}, \vect{b} \rangle_* &=\langle \mathcal{I}_n^{-1}\left(\mathcal{P}_N(T_{\func{u}_{LS,N}} \func{v}_N)\right), \vect{b} \rangle_* \\
    &=\langle T_{\func{u}_{LS,N}} \func{v}_N, \mathcal{I}_n(\vect{b}) \rangle_{L^2} \\
    &=\int_{I_\xi} \langle D_{\vect{u}}\func{f}(\func{u}_{LS,N}(z),\bar{\mu}+\sigma z)\func{v}_N(z), \func{v}_N(z) \rangle_{\R^n} \, \mathrm{d}z \\ 
    &\geq \frac{\lambda_{\min}}{2} \|\func{v}_N\|_{L^2}^2 = \frac{\lambda_{\min}}{2} |\vect{b}|_*^2.
\end{align*}
By the Cauchy--Schwarz inequality it holds $|\langle \vect{JF}_N(\vect{a})\vect{b}, \vect{b} \rangle_*| \leq |\vect{JF}_N(\vect{a})\vect{b}|_* |\vect{b}|_*$. Thus, dividing by $|\vect{b}|_*$, we conclude that $|\vect{JF}_N(\vect{a})\vect{b}|_* \geq \lambda_{\min}|\vect{b}|_*/2$, which implies strict invertibility and the bound $|(\vect{JF}_N(\vect{a}))^{-1}|_* \leq 2/\lambda_{\min}$.

Next, we address the Lipschitz continuity. Let $\hat{\vect{u}}_1, \hat{\vect{u}}_2 \in \R^{n(N+1)}$ be such that $\mathcal{I}_n(\hat{\vect{u}}_1), \mathcal{I}_n(\hat{\vect{u}}_2) \in U$ and let $\vect{b} \in \R^{n(N+1)}$ with $|\vect{b}|_* = 1$. By setting $\func{v}_N = \mathcal{I}_n(\vect{b})$ we have $\|\func{v}_N\|_{L^2}=1$. Using the fact that $\mathcal{I}_n^{-1}$ is an isometry from $\Pi_N^n \subset L^2(I_\xi; \R^n)$ to $\R^{n(N+1)}$ and $\mathcal{P}_N$ is an $L^2$-orthogonal projection (hence a contraction, $\|\mathcal{P}_N\|_{\mathcal{L}(L^2, L^2)} = 1$), and applying the standard bound for multiplication operators $\|A\func{v}\|_{L^2} \leq \|A\|_{L^\infty} \|\func{v}\|_{L^2}$, we find:
\begin{align*}
    | (\vect{JF}_N(\hat{\vect{u}}_1) - \vect{JF}_N(\hat{\vect{u}}_2))\vect{b} |_* &= \left|\mathcal{I}_n^{-1}\left(\mathcal{P}_N \left( (T_{\mathcal{I}_n(\hat{\vect{u}}_1)} - T_{\mathcal{I}_n(\hat{\vect{u}}_2)}) \func{v}_N \right)\right)\right|_* \\
    &= \left\| \mathcal{P}_N \left( (T_{\mathcal{I}_n(\hat{\vect{u}}_1)} - T_{\mathcal{I}_n(\hat{\vect{u}}_2)}) \func{v}_N \right) \right\|_{L^2} \\
    &\leq \| (T_{\mathcal{I}_n(\hat{\vect{u}}_1)} - T_{\mathcal{I}_n(\hat{\vect{u}}_2)}) (\func{v}_N) \|_{L^2} \\
    &\leq \| D_{\vect{u}}\func{f}(\mathcal{I}_n(\hat{\vect{u}}_1),\cdot) - D_{\vect{u}}\func{f}(\mathcal{I}_n(\hat{\vect{u}}_2),\cdot) \|_{L^\infty} \|\func{v}_N\|_{L^2} \\
    &\leq L_U \| \mathcal{I}_n(\hat{\vect{u}}_1) - \mathcal{I}_n(\hat{\vect{u}}_2) \|_{L^\infty},
\end{align*}
where $L_U$ is the local Lipschitz constant of the Jacobian matrix $D_{\vect{u}}\func{f}(\func{u}(z),\bar{\mu}+\sigma z)$ in $U$ with respect to the state variable. 

Finally, we invoke component-wise the standard spectral inverse inequality for polynomials, often referred to as Nikolskii's inequality \cite{canuto2006spectral,devore1993constructive,milovanovic1994topics}, which states there exists a constant $C > 0$ such that
$$
\|\func{p}_N\|_{L^\infty} \leq C N \|\func{p}_N\|_{L^2} \quad \forall\, \func{p}_N \in \Pi_N^n.
$$
Applying this to the difference $\mathcal{I}_n(\hat{\vect{u}}_1) - \mathcal{I}_n(\hat{\vect{u}}_2)$, and setting $K_{\bar{\func{u}}}(N) = L_U C N = \mathcal{O}(N)$, we complete the proof by obtaining
\begin{equation*}
    | (\vect{JF}_N(\hat{\vect{u}}_1) - \vect{JF}_N(\hat{\vect{u}}_2))\vect{b} |_* \leq L_U C N \| \mathcal{I}_n(\hat{\vect{u}}_1) - \mathcal{I}_n(\hat{\vect{u}}_2) \|_{L^2} = L_U C N |\hat{\vect{u}}_1 - \hat{\vect{u}}_2|_*.
\end{equation*}
\end{proof}

With all these preliminary results, we are ready to show that for sufficiently large $N$, we can find a solution of the PCE-Galerkin system that is close enough in the $L^2$ norm to the branch $\bar{\func{u}}$.

\begin{thm}[Convergence]\label{thm:existence_of_PC_ls_sol}
If $\bar{\func{u}}\in H^s(I_\xi; \R^n)$ for $s>1$~\eqref{ass:A2b} and satisfies the uniform positive definiteness condition~\eqref{eq:uniform_positive_definiteness} with $P = I_n$, and assuming that~\eqref{eq:LipLocUnif} holds for $D\mathcal{F}$, then there exists $N_{\text{conv}} \in \mathbb{N}$ such that, for any $N\geq N_{\text{conv}}$, there exists a unique solution $\func{u}_N$ to the Galerkin projection system $\func{F}_N(\hat{\vect{u}})=\vect{0}$ in a neighborhood of the least-squares approximation $\func{u}_{LS,N}$. Furthermore, this solution satisfies:
$$
\|\bar{\func{u}}-\func{u}_{N}\|_{L^2}=E(N)
$$
where $E(N)=\mathcal{O}(N^{-s})$ is the rate of convergence of the residual $\|\mathcal{F}(\func{u}_{LS,N})\|_{L^2}$.
\end{thm}

\begin{proof}
We apply the Newton--Kantorovich theorem (see Appendix~\ref{thm:NewtonKantorovich}) to the Galerkin operator $\func{F}_N$ starting from the initial guess given by the least-squares approximation coefficients $\vect{a} = \mathcal{I}_n^{-1}(\func{u}_{LS,N})$. 

From the previous propositions, for sufficiently large $N$, we can bound the quantities required by the Newton--Kantorovich theorem in the induced norm, i.e.\ $\beta_N$ the norm of the inverse Jacobian, $K_N$ the Lipschitz constant of the Jacobian, and $\eta_N$ the norm of the initial Newton step, as
\begin{align*}
    \beta_N &= |(\vect{JF}_N(\vect{a}))^{-1}|_* \leq \frac{2}{\lambda_{\min}}, \\
    K_N &= K_{\bar{\func{u}}}(N) = \mathcal{O}(N), \\
    \eta_N &=|(\vect{JF}_N(\vect{a}))^{-1} \func{F}_N(\vect{a})|_* \leq \beta_N|\func{F}_N(\vect{a})|_* \leq \frac{2}{\lambda_{\min}} E(N),
\end{align*}
where $E(N) = \|\mathcal{F}(\func{u}_{LS,N})\|_{L^2} = \mathcal{O}(N^{-s})$. 
We must verify that the Kantorovich condition $h_N = \beta_N K_N \eta_N \leq \frac{1}{2}$ holds for $N\geq N_{\text{conv}}$. Since $\bar{\func{u}}\in H^s(I_\xi; \R^n)$ with $s>1$, in the evaluation of $h_N$
$$
h_N = \beta_N K_N \eta_N \leq \left(\frac{2}{\lambda_{\min}}\right)^2 \mathcal{O}(N) \mathcal{O}(N^{-s}) = \mathcal{O}(N^{1-s}),
$$
the exponent $1-s$ is strictly negative, implying that $h_N \to 0$ as $N \to \infty$. Consequently, there exists an $N_{\text{conv}}$ such that $h_N \leq \frac{1}{2}$ for all $N \geq N_{\text{conv}}$. By the Newton--Kantorovich theorem, the sequence of Newton iterates converges to a unique solution $\func{u}_N$ of the PCE-Galerkin system.

To estimate the distance from the initial guess, as $N \to \infty$, $h_N \to 0$ and $\frac{1-\sqrt{1-2h_N}}{h_N}\to 1$; exploiting the Newton--Kantorovich bound for the convergence radius $r_N^*$, we obtain
$$
\|\func{u}_{LS,N}-\func{u}_N\|_{L^2} \leq r_N^* = \frac{1-\sqrt{1-2h_N}}{h_N} \eta_N  \sim \frac{2}{\lambda_{\min}} E(N).
$$
Finally, applying the triangle inequality, we conclude
$$
\|\bar{\func{u}}-\func{u}_N\|_{L^2}\leq \|\bar{\func{u}}-\func{u}_{LS,N}\|_{L^2}+\|\func{u}_{LS,N}-\func{u}_N\|_{L^2} \sim \left(1+\frac{2}{\lambda_{\min}}\right)E(N)=\mathcal{O}(N^{-s}).
$$
\end{proof}

In numerical implementations, solving the non-linear PCE-Galerkin system requires an iterative root-finding algorithm, such as Newton's method, which relies on a sufficiently accurate initial guess. A standard computational strategy is to use the computed solution $\func{u}_N$ at degree $N$ as the initial guess for the augmented system of degree $N+1$. The following corollary provides the theoretical foundation for this approach by proving that the distance between successive discrete approximations decays asymptotically. This guarantees that, as the polynomial degree increases, the lower-degree solution naturally falls within the convergence neighborhood of the higher-degree system, ensuring stable convergence toward the exact branch $\bar{\func{u}}$.

\begin{cor}\label{cor:successive_sol}
Let $\func{u}_N$ and $\func{u}_{N+1}$ be the unique local solutions of the PCE-Galerkin system~\eqref{eq:PB2} of degree $N$ and $N+1$, respectively, converging to the branch $\bar{\func{u}}$. Then, the distance between successive approximations satisfies
$$
\|\func{u}_N-\func{u}_{N+1}\|_{L^2}=\mathcal{O}(N^{-s}).
$$
\end{cor}

\begin{proof}
Using the $L^2$-isometry of the coefficient norm and the triangle inequality, we have
\begin{align*}
    \|\func{u}_N-\func{u}_{N+1}\|_{L^2} &\leq \|\func{u}_N-\func{u}_{LS,N}\|_{L^2}+\|\func{u}_{LS,N}-\func{u}_{LS,N+1}\|_{L^2}+\|\func{u}_{LS,N+1}-\func{u}_{N+1}\|_{L^2}.
\end{align*}
The first and last term have asymptotically the behavior respectively of $E(N)$ and $E(N+1)$. Since $\func{u}_{LS,N}$ is the $L^2$-orthogonal projection of $\bar{\func{u}}$, the middle term is bounded by the truncation error
$$
\|\func{u}_{LS,N}-\func{u}_{LS,N+1}\|_{L^2}^2 \leq \sum_{i=N+1}^{+\infty} \|\hat{\func{u}}_{LS,i} \Phi_i\|_{L^2}^2 = \|\bar{\func{u}}-\func{u}_{LS,N}\|_{L^2}^2 \leq E(N)^2.
$$
Since $E(N+1)<E(N)$, combining these results and calling $\bar{C}=1+\frac{2}{\lambda_{\min}}$, gives:
$$
\|\func{u}_N-\func{u}_{N+1}\|_{L^2}\sim (2\bar{C} + 1) E(N).
$$
\end{proof}

\begin{rmk}
Observe that, when a bifurcation occurs for a specific value of the parameter, the Jacobian matrix becomes singular, i.e., $D_{\vect{u}}\func{f}(\bar{\func{u}}(z),\bar{\mu}+\sigma z)$ has a zero eigenvalue for some $z \in I_\xi$. In such cases, if the critical parameter falls inside the stochastic domain, the uniform positive definiteness is lost ($\lambda_{\min}=0$), and the theoretical guarantees provided by the Newton--Kantorovich framework break down. However, we empirically observe that PCE-Galerkin approximations may still converge to the branches even near bifurcations, possibly with degraded rates due to the loss of invertibility of the Jacobian.

Furthermore, away from bifurcation points, the solution branches are typically not just $H^s(I_\xi; \R^n)$, but analytic in the parameter $\xi$. In these analytic regions, the convergence of the Polynomial Chaos expansion becomes super-polynomial. Specifically, the error is expected to decay exponentially as $\mathcal{O}(e^{-N})$ for some constant $c > 0$, as established in spectral approximation theory for analytic functions \cite{canuto2006spectral, trefethen2013approximation}.  This explains the extremely high precision achieved by the Galerkin method in stable regions of the bifurcation diagram.
\end{rmk}

\subsection{Well-posedness}
As a final step, we show how the weak formulation~\eqref{eq:PB2} inherits fundamental properties of the original problem~\eqref{eq:PB1} under suitable assumptions, thereby ensuring the well-posedness of the PCE-Galerkin approximations, regarding the non existence and uniqueness of solutions.

In a multi-dimensional setting, the absence of exact continuous roots for $\func{f}(\func{u}, \bar{\mu}+\sigma z) = \vect{0}$ does not automatically preclude the PCE-Galerkin problem from having spurious solutions. However, by imposing an orientational constraint, we can bypass this problem: if the vector field maintains a uniformly positive projection along a specific constant direction $\vect{c} \in \R^n$, the PCE-Galerkin system inherits the global non-existence property. Intuitively, since constant vectors always belong to the polynomial test space $\Pi_{N}^n$, projecting the Galerkin residual along $\vect{c}$ yields the integral of a strictly positive quantity.
This result is formalized in the following theorem.

\begin{thm}[Non-existence]\label{thm:non_existence}
Suppose that, in~\eqref{eq:PB1}, for every $z\in I_\xi$ the equation
$$
\func{f}(\vect{u},\bar{\mu}+\sigma z)=\vect{0}
$$
has no solution $\vect{u}\in\R^n$. 
Moreover, assume there exists a constant vector $\vect{c} \in \R^n \setminus \{\vect{0}\}$ such that the vector field is oriented along $\vect{c}$, i.e.
\begin{equation*}
    \langle \func{f}(\vect{u},\bar{\mu}+\sigma z), \vect{c} \rangle_{\R^n} > 0 \qquad \forall \vect{u} \in \R^n, \; \forall z \in I_\xi. 
\end{equation*}
Then, for every $N\in\mathbb{N}$, the projected problem~\eqref{eq:PB2} admits no solution.
\end{thm}

\begin{proof}
Consider the PCE-Galerkin problem~\eqref{eq:PB2}. Let $\Pi_N^n$ be the approximation space and assume, by contradiction, that there exists a solution $\func{u}_N \in \Pi_N^n$. Since constant vector functions belong to $\Pi_N^n$, we can choose $\func{v}_N \equiv \vect{c}$ as a test function. Because the projection along $\vect{c}$ is strictly positive everywhere by assumption, this yields
$$
\int_{I_\xi} \langle \mathcal{F}(\func{u}_N)(z), \vect{c} \rangle_{\R^n} \, \mathrm{d}z > 0,
$$
which directly contradicts the definition of the Galerkin projection $\int_{I_\xi} \langle \mathcal{F}(\func{u}_N)(\xi), \func{v}_N(\xi) \rangle_{\R^n} \,\dxi = 0$. This shows that~\eqref{eq:PB2} admits no solution for any $N \in \mathbb{N}$.
\end{proof}

\begin{rmk}
If $n=1$, the assumption $\langle \func{f}(\vect{u},\bar{\mu}+\sigma z), \vect{c} \rangle_{\R^n} > 0$ for all $\vect{u} \in \R^n$ and $z \in I_\xi$ is automatically satisfied: by the Intermediate Value Theorem, a continuous scalar function that never vanishes must maintain a constant sign, allowing one to simply choose $c=1$ or $c=-1$. 
However, for $n \ge 2$, the space $\R^n \setminus \{\vect{0}\}$ is connected. 
Without an orientational constraint, the vector field could rotate around the origin as the parameter $z$ varies, causing its integral over $I_\xi$ to average to exactly $0$. 
In such cases, the PCE-Galerkin system could incorrectly admit solutions despite the original continuous problem having none.
\end{rmk}

\begin{rmk}
The assumption that the approximation spaces $\Pi_N^n$ contain the constant functions is essential for the non-existence result to hold. 
If this property is lacking, the PCE-Galerkin system may admit solutions despite the absence of solutions to the original problem~\eqref{eq:PB1}.

To illustrate this, consider the case where $\func{f}(\vect{u}, \bar{\mu} + \sigma z) \equiv \vect{c}$ for all $(\vect{u}, z) \in \R^n \times I_\xi$, with $\vect{c} \neq \vect{0}$. 
Clearly,~\eqref{eq:PB1} has no solution. 
However, suppose we choose an approximation space $\mathbf{V}_N$ that excludes constants, such as $\mathbf{V}_N \coloneq (\mathrm{span}\{\Phi_1, \dots, \Phi_N\})^n$, where the basis functions $\{\Phi_i\}_{i=1}^N$ are $L^2_{\nu_\xi}$-orthogonal to constants, i.e., $\int_{I_\xi} \Phi_i(\xi) \,\dxi = 0$ for all $i \in \{1, \dots, N\}$. 
For any $\func{u}_N \in \mathbf{V}_N$ and any test function $\func{v}_N \in \mathbf{V}_N$, we have
$$
\int_{I_\xi} \langle \mathcal{F}(\func{u}_N)(z), \func{v}_N(z) \rangle_{\R^n} \, \mathrm{d}z = \int_{I_\xi} \langle \vect{c}, \func{v}_N(z) \rangle_{\R^n} \, \mathrm{d}z= 0.
$$
Consequently, every element $\func{u}_N \in \mathbf{V}_N$ becomes a solution to the projected problem~\eqref{eq:PB2}. This happens because the information preventing existence, e.g.\ the non--zero mean of the operator, is lost when the problem is projected onto a subspace orthogonal to the constant functions.
\end{rmk}

We now complement the preceding non-existence theorem by showing that uniformly bounded solutions to~\eqref{eq:PB2} cannot converge strongly in $L^2$ in a region where equilibria are absent on a set of parameter values of positive measure.

\begin{prop}[Localized non-existence]\label{prop:limits_galerkin}
Let $D\subset\R^n$ be closed and $J\subset I_\xi$ be measurable with positive measure. Suppose that, for every $z\in J$, the equation
$$
\func{f}(\vect{u},\bar{\mu}+\sigma z)=\vect{0}
$$
has no solution $\vect{u}\in D$. Then no sequence $\{\func{u}_N\}_{N\in\mathbb{N}}$, with $\func{u}_N\in\Pi_N^n$ solving~\eqref{eq:PB2} for each $N$, satisfying
$$
\sup_N\|\func{u}_N\|_{L^\infty}<+\infty,
\qquad
\func{u}_N(z)\in D\quad\text{for almost every }z\in J,
$$
converges strongly in $L^2(I_\xi;\R^n)$.
\end{prop}

\begin{proof}
Suppose, by contradiction, that $\func{u}_N\to\func{u}^*$ strongly in $L^2$. Up to extracting a subsequence $\{\func{u}_{N_k}\}_{k\in\N}$, we have $\func{u}_{N_k}(z) \to \func{u}^*(z)$ for almost every $z \in I_\xi$~\cite[Theorem~4.9]{brezis2011}.
Choose $M>0$ such that $\|\func{u}_N\|_{L^\infty}\leq M$ for all $N$. Since each $\func{u}_{N_k}$ is continuous, we have $\|\func{u}_{N_k}(z)\|_{\R^n} \le M$ for every $k\in\N$ and $z\in I_{\xi}$. Consequently, for almost every $z\in I_{\xi}$, passing to the limit gives $\|\func{u}^*(z)\|_{\R^n} = \lim_{k\to\infty}\|\func{u}_{N_k}(z)\|_{\R^n} \leq M$, implying that $\|\func{u}^*\|_{L^{\infty}}\leq M$.

Next, we denote by $F$ the set of points where pointwise convergence fails and by $E_k$ the set where the inclusion $\func{u}_{N_k}(z)\in D$ fails. Then, the set $E\coloneq \left(F\cup\bigcup_{k\in\N}E_k\right)$ is a countable union of zero-measure sets, thus it has measure zero itself. For every $z \in J \setminus E$, we simultaneously have $\func{u}_{N_k}(z)\in D$ for all $k$ and $\func{u}_{N_k}(z)\to\func{u}^*(z)$. Since $D$ is closed, it follows that $\func{u}^*(z)\in\overline{D}=D$ for almost every $z\in J$.

Set $\func{g}(z)\coloneq\func{f}(\func{u}^*(z),\bar{\mu}+\sigma z)$. By continuity of $\func{f}$ on $\bar{B}_M(\vect{0})\times I$, we have $\func{g}\in L^\infty\subset L^2$, and~\eqref{eq:LipLocUnif} yields
$$
\|\mathcal{F}(\func{u}_N)-\func{g}\|_{L^2}
\leq L_M\|\func{u}_N-\func{u}^*\|_{L^2}\longrightarrow0.
$$
For any fixed $\func{v}\in\Pi_m^n$, the Galerkin equations give $\langle\mathcal{F}(\func{u}_N),\func{v}\rangle_{L^2}=0$ whenever $N\geq m$. 
Passing to the limit as $N\to\infty$ gives $\langle\func{g},\func{v}\rangle_{L^2}=0$ so that, by density of the vector-valued polynomials in $L^2(I_\xi;\R^n)$, it must be that $\func{g}=\vect{0}$ almost everywhere. Thus $\func{u}^*(z)\in D$ solves $\func{f}(\func{u}^*(z),\bar{\mu}+\sigma z)=\vect{0}$ for almost every $z\in J$, contradicting the assumption since $J$ has positive measure.
\end{proof}

We finally establish a uniqueness result for the PCE-Galerkin approximation. Standard uniqueness proofs for non-linear Galerkin methods typically rely on global strict monotonicity, a condition requiring
$$
\langle \mathcal{F}(\func{w})(z)-\mathcal{F}(\func{u})(z), \func{w}(z)-\func{u}(z) \rangle_{\R^n}> 0 \qquad \forall z\in I_\xi, \quad \forall\, \func{w}(z) \neq \func{u}(z).
$$
However, in a multi-dimensional setting, this requirement is often violated by polynomial dynamical systems featuring complex rotations or chaotic regimes, such as the Lorenz equations. To bypass this problem and establish a more widely applicable result, we replace the global constraint with a combination of local and asymptotic geometric properties. 

Locally, the absence of bifurcation points within $I_\xi$ implies a uniformly positive definite Jacobian, yielding strict monotonicity in a neighborhood of the exact solution branch. Far from the root, we introduce two complementary conditions: a uniform radial coercivity bound, which ensures the vector field is strictly directed outwards (or inwards), and a relative growth bound. This latter constraint bounds the total magnitude of the field relative to its radial projection, preventing unbounded tangential rotations that could lead to spurious solutions. Together, these properties guarantee that the Galerkin formulation preserves the uniqueness of the continuous problem.

\begin{thm}[Uniqueness]\label{thm:uniqueness}
Suppose that~\eqref{eq:PB1} admits a unique solution in the interval $I$, whose re-parametrization satisfies~\eqref{ass:A2b} ($s>1$), that $\func{f}$ is continuously differentiable with respect to its first argument with both $\func{f}$ and $D\mathcal{F}$ satisfying~\eqref{eq:LipLocUnif}, and that the $P$-weighted uniform positive definiteness condition~\eqref{eq:uniform_positive_definiteness} and the asymptotic geometric constraints~\eqref{ass:A4a}--\eqref{ass:A4b} hold for a constant symmetric positive-definite matrix $P \in \R^{n \times n}$.
Then, there exists an integer $\bar{N}$ such that, for all $N \geq \bar{N}$, the PCE-Galerkin problem~\eqref{eq:PB2} admits exactly one solution $\func{u}_N\in\Pi_N^n$.
\end{thm}

\begin{proof}
The existence of a solution $\func{u}_N\in\Pi_N^n$ follows from Theorem~\ref{thm:existence_of_PC_ls_sol} for $N \geq N_{\text{conv}}$. We first prove that $\func{u}_N$ converges to $\bar{\func{u}}$ uniformly. Using the triangle inequality and the Nikolskii's inequality introduced earlier, we bound the $L^\infty$ error as:
$$
\|\func{u}_N - \bar{\func{u}}\|_{L^\infty} \leq \|\func{u}_N - \func{u}_{LS,N}\|_{L^\infty} + \|\func{u}_{LS,N} - \bar{\func{u}}\|_{L^\infty} \leq C N \|\func{u}_N - \func{u}_{LS,N}\|_{L^2} + \|\func{u}_{LS,N} - \bar{\func{u}}\|_{L^\infty}.
$$
From Theorem~\ref{thm:consistency} and Theorem~\ref{thm:existence_of_PC_ls_sol}, we know that $\|\func{u}_N - \func{u}_{LS,N}\|_{L^2}=\mathcal{O}(N^{-s})$ and $\|\func{u}_{LS,N} - \bar{\func{u}}\|_{L^\infty}=\mathcal{O}(N^{\frac{3}{4}+\delta-s})$. Since $s>1$, the terms decay asymptotically, yielding $\lim_{N\to\infty}\|\func{u}_N - \bar{\func{u}}\|_{L^\infty}=0$. Given that $\mathcal{F}$ is Lipschitz continuous on bounded sets in $L^\infty$ by Proposition~\ref{prop:LocalLipschitzNemytskii}, this uniform convergence also implies that $\lim_{N\to\infty}\|\mathcal{F}(\func{u}_N)\|_{L^\infty}=0$.

Now, recall that the non-degeneracy condition gives $\langle D_{\vect{u}}\func{f}(\bar{\func{u}}(z), \bar{\mu}+\sigma z)\vect{w}, \vect{w} \rangle_P \geq \lambda_{\min} \|\vect{w}\|_P^2 > 0$ for all $z \in I_\xi$. 
Since $D_{\vect{u}}\func{f}$ is continuous and $I_\xi$ is compact, there exists $\rho>0$ such that
$\|\vect{u}-\bar{\func{u}}(z)\|_P<\rho$ implies $\langle D_{\vect{u}}\func{f}(\bar{\vect{u}},\bar{\mu}+\sigma z)\vect{w}, \vect{w} \rangle_P \geq\frac{\lambda_{\min}}{2}\|\vect{w}\|_P^2\ \forall z\in I_\xi$.
Introducing the neighborhood of $\bar{\func{u}}$ as 
$$
U \coloneq \left\{ \func{u} \in L^\infty(I_\xi; \R^n) : \sup_{z \in I_\xi} \|\func{u}(z)-\bar{\func{u}}(z)\|_P < \rho \right\},
$$
then, for every $\func{u} \in U$ and every $z \in I_\xi$, $\langle D_{\vect{u}}\func{f}(\func{u}(z),\bar{\mu}+\sigma z)\vect{w}, \vect{w} \rangle_P \geq \frac{\lambda_{\min}}{2}\|\vect{w}\|_P^2$. 
Observe that $U$ is convex, since for all $z\in I_\xi$, given $t\in[0,1]$ and $\func{u}_3\coloneq t\func{u}_1+(1-t)\func{u}_2$ for $\func{u}_1,\func{u}_2\in U$, we have  
\begin{align*}
    \|\bar{\func{u}}(z)-\func{u}_3(z)\|_P&=\|t(\bar{\func{u}}(z)-\func{u}_1(z))+(1-t)(\bar{\func{u}}(z)-\func{u}_2(z))\|_P\\
    &\leq t\|\bar{\func{u}}(z)-\func{u}_1(z)\|_P+(1-t)\|\bar{\func{u}}(z)-\func{u}_2(z)\|_P\\
    &<t\rho+(1-t)\rho=\rho.
\end{align*}
Consequently, we can apply the integral Mean Value Theorem~\cite{cartan1971differential}: for any $z\in I_\xi$, and $\func{u}_1,\func{u}_2\in U$, we have
$$
\mathcal{F}(\func{u}_1)(z) - \mathcal{F}(\func{u}_2)(z) = \left( \int_0^1 D_{\vect{u}}\func{f}(\func{u}_2(z) + t(\func{u}_1(z)-\func{u}_2(z)), \bar{\mu} + \sigma z) \, dt \right) (\func{u}_1(z) - \func{u}_2(z)).
$$
Since the segment $\func{u}_3(t, z) \coloneq \func{u}_2(z) + t(\func{u}_1(z)-\func{u}_2(z))$ is entirely contained in $U$ for all $t \in [0,1]$, taking the $P$-inner product with $(\func{u}_1(z) - \func{u}_2(z))$ implies that:
$$
\langle \mathcal{F}(\func{u}_1)(z) - \mathcal{F}(\func{u}_2)(z), \func{u}_1(z) - \func{u}_2(z) \rangle_P \geq \frac{\lambda_{\min}}{2} \|\func{u}_1(z) - \func{u}_2(z)\|_P^2 \quad \forall\,z\in I_\xi.
$$
This directly implies that $\mathcal{F}$ is strictly pointwise monotone in $U$, since
$$
\langle \mathcal{F}(\func{u}_1)(z)-\mathcal{F}(\func{u}_2)(z), \func{u}_1(z)-\func{u}_2(z) \rangle_P \geq \frac{\lambda_{\min}}{2} \|\func{u}_1(z)-\func{u}_2(z)\|_P^2 \geq 0 \quad \forall\,\func{u}_1,\func{u}_2\in U.
$$
Integrating this expression over $I_\xi$ ensures strict monotonicity in the weak sense, i.e.\  we obtain that for any $\func{u}_1, \func{u}_2 \in U$ with $\func{u}_1 \neq \func{u}_2$ in a positive--measure set it holds
$$
\int_{I_{\xi}}\langle \mathcal{F}(\func{u}_1)(z)-\mathcal{F}(\func{u}_2)(z), \func{u}_1(z)-\func{u}_2(z) \rangle_P\,\mathrm{d}z > 0.
$$

Furthermore, by the uniform radial coercivity assumption, the radial projection of $\mathcal{F}$ is uniformly bounded away from $0$ outside of $U$. In fact, by setting the radius of the uniform radial coercivity to match $\rho$, there exists a constant $m_U > 0$ such that, for all $z \in I_\xi$ s.t. $\|\func{u}(z)-\bar{\func{u}}(z)\|_P\geq\rho$,
$$
\langle \mathcal{F}(\func{u})(z), \func{u}(z) - \bar{\func{u}}(z) \rangle_P \ge m_U \|\func{u}(z) - \bar{\func{u}}(z)\|_P.
$$ 

Next, we choose $\varepsilon \in (0, m_U/4)$ and define 
$$
V \coloneq \left\{ \func{u} \in \mathcal{C}(I_\xi; \R^n) : \sup_{z \in I_\xi}\|\mathcal{F}(\func{u})(z)\|_P < \varepsilon \right\},
$$ 
and the following neighborhood of $\bar{\func{u}}(z)$
$$
U^* \coloneq \left\{ \func{u} \in U : \sup_{z \in I_\xi}\|\func{u}(z)-\bar{\func{u}}(z)\|_P \leq \frac{1}{2C_\rho} \right\}.
$$
This is a tighter neighborhood such that $U^*\subsetneq U$. Indeed, applying the Cauchy--Schwarz inequality along with the radial coercivity and the relative growth bound for any $\vect{u}$ such that $\|\vect{u} - \bar{\func{u}}(z)\|_P \geq \rho$, we obtain 
\begin{align*}
\langle\func{f}(\vect{u},\bar{\mu}+\sigma z), \vect{u} - \bar{\func{u}}(z) \rangle_P &\leq \|\func{f}(\vect{u},\bar{\mu}+\sigma z)\|_P \|\vect{u} - \bar{\func{u}}(z)\|_P\\
&\leq C_\rho \langle \func{f}(\vect{u},\bar{\mu}+\sigma z), \vect{u} - \bar{\func{u}}(z) \rangle_P \|\vect{u}-\bar{\func{u}}(z)\|_P.
\end{align*}
Dividing by the inner product gives $1 \leq C_\rho \|\vect{u}-\bar{\func{u}}(z)\|_P$, that must hold for all $\|\vect{u}-\bar{\func{u}}(z)\|_P \geq \rho$; consequently, it is necessary that $C_\rho \geq 1/\rho$. Consequently, the radius of $U^*$ satisfies $\frac{1}{2C_\rho} \leq \frac{\rho}{2} < \rho$, ensuring $U^* \subsetneq U$.
Since $\|\mathcal{F}(\func{u}_N)\|_{L^\infty} \to 0$ and $\|\bar{\func{u}}-\func{u}_N\|_{L^\infty} \to 0$, and due to the equivalence of norms in $\R^n$, there exists an integer $\bar{N}$ such that $\func{u}_N \in U^* \cap V$ for all $N \geq \bar{N}$.

For any fixed $\func{w} \in U^* \cap V$ and any $\func{u} \in \mathcal{C}(I_\xi; \R^n)$ with $\func{w} \neq \func{u}$, the following strict global inequality holds:
\begin{equation}
\int_{I_\xi}\langle \mathcal{F}(\func{w})(z)-\mathcal{F}(\func{u})(z), \func{w}(z)-\func{u}(z) \rangle_P\,\mathrm{d}z > 0.
\label{eq:WeakIneq}
\end{equation}
To see this, consider two cases for each $z \in I_\xi$:
\begin{itemize}
    \item[i)] If $\func{u}(z)$ is such that $\|\func{u}(z)-\bar{\func{u}}(z)\|_P<\rho$, where the Jacobian matrix is uniformly positive definite, we apply the integral Mean Value Theorem as before: 
    $$
    \langle\mathcal{F}(\func{w})(z)-\mathcal{F}(\func{u})(z), \func{w}(z)-\func{u}(z)\rangle_P\geq \frac{\lambda_{\min}}{2}\|\func{w}(z)-\func{u}(z)\|_P^2\geq 0.
    $$
    \item[ii)] If $\func{u}(z)$ is outside this neighborhood, implying that $\|\func{u}(z)-\bar{\func{u}}(z)\|_P\geq\rho$, we expand $\langle \mathcal{F}(\func{u})(z) - \mathcal{F}(\func{w})(z), \func{u}(z) - \func{w}(z) \rangle_P$ by adding and subtracting $\bar{\func{u}}(z)$:
    \begin{equation}
    \langle \mathcal{F}(\func{u})(z), \func{u}(z) - \bar{\func{u}}(z) \rangle_P - \langle \mathcal{F}(\func{u})(z), \func{w}(z) - \bar{\func{u}}(z) \rangle_P- \langle \mathcal{F}(\func{w})(z), \func{u}(z) - \func{w}(z) \rangle_P.
    \label{eq:ScalarProdSplitting}  
    \end{equation}
    We bound the second term of this expression using the Cauchy--Schwarz inequality and the relative growth bound:
    \begin{align*}
        |\langle \mathcal{F}(\func{u})(z), \func{w}(z) - \bar{\func{u}}(z) \rangle_P|&\leq\|\mathcal{F}(\func{u})(z)\|_P \|\func{w}(z) - \bar{\func{u}}(z)\|_P \\
        &\leq C_\rho\langle \mathcal{F}(\func{u})(z), \func{u}(z) - \bar{\func{u}}(z) \rangle_P \|\func{w}(z) - \bar{\func{u}}(z)\|_P.
    \end{align*}
    Consequently, factoring out the radial projection, the first two terms of equation~\eqref{eq:ScalarProdSplitting} can be bounded from below by:
    $$
    \langle \mathcal{F}(\func{u})(z), \func{u}(z) - \bar{\func{u}}(z) \rangle_P \left(1 - C_\rho \|\func{w}(z) - \bar{\func{u}}(z)\|_P\right).
    $$
    Since $\func{w} \in U^*$, we have $\|\func{w}(z) - \bar{\func{u}}(z)\|_P \le \frac{1}{2C_\rho}$. Thus, the combination of the first two terms is strictly bounded from below by $\frac{1}{2}\langle \mathcal{F}(\func{u})(z), \func{u}(z) - \bar{\func{u}}(z) \rangle_P$. Using then the radial coercivity, we have that 
    $$
    \frac{1}{2}\langle \mathcal{F}(\func{u})(z), \func{u}(z) - \bar{\func{u}}(z) \rangle_P\ge \frac{m_U}{2} \|\func{u}(z) - \bar{\func{u}}(z)\|_P > 0.
    $$ 

    For the third term, we bound its magnitude using Cauchy--Schwarz and the triangle inequality:
    $$
    |\langle \mathcal{F}(\func{w})(z), \func{u}(z) - \func{w}(z) \rangle_P| \le \|\mathcal{F}(\func{w})(z)\|_P (\|\func{u}(z) - \bar{\func{u}}(z)\|_P + \|\func{w}(z) - \bar{\func{u}}(z)\|_P).
    $$
    Recall that, since $\func{w} \in V$, $\|\mathcal{F}(\func{w})(z)\|_P<\varepsilon$, and that $\|\func{u}(z) - \bar{\func{u}}(z)\|_P \ge \rho$. On the other hand, $\func{w}(z)$ is inside $U$, meaning $\|\func{w}(z) - \bar{\func{u}}(z)\|_P < \rho \le \|\func{u}(z) - \bar{\func{u}}(z)\|_P$. 
    Therefore, the expression in the parenthesis is strictly less than $2\|\func{u}(z) - \bar{\func{u}}(z)\|_P$, leading to
    $$
    |\langle \mathcal{F}(\func{w})(z), \func{u}(z) - \func{w}(z) \rangle_P|\leq 2\varepsilon\|\func{u}(z) - \bar{\func{u}}(z)\|_P.
    $$
    Consequently, it trivially follows that $-\langle \mathcal{F}(\func{w})(z), \func{u}(z) - \func{w}(z) \rangle_P\geq -2\varepsilon\|\func{u}(z) - \bar{\func{u}}(z)\|_P$. Combining these bounds,
    $$
    \langle \mathcal{F}(\func{u})(z) - \mathcal{F}(\func{w})(z), \func{u}(z) - \func{w}(z) \rangle_P \geq \left( \frac{m_U}{2} - 2\varepsilon \right) \|\func{u}(z) - \bar{\func{u}}(z)\|_P.
    $$
    Because we initially chose $\varepsilon < m_U/4$, this lower bound remains strictly positive everywhere outside $U$.
\end{itemize}
Thus, as claimed, the integrand is non-negative everywhere and strictly positive on the set where $\func{w} \neq \func{u}$, ensuring that inequality~\eqref{eq:WeakIneq} holds.

Now, fix $N \geq \bar{N}$ and let $\func{u}_N \in U^* \cap V \cap \Pi_N^n$ be the solution of the PCE-Galerkin system. Suppose there exists another global solution $\func{v}_N \in \Pi_N^n$ to the PCE-Galerkin system, with $\func{v}_N \neq \func{u}_N$. Because the difference $(\func{u}_N - \func{v}_N)$ belongs to $\Pi_N^n$ and $P$ is a constant matrix, $P(\func{u}_N - \func{v}_N)$ also belongs to $\Pi_N^n$. We can therefore use it as a test function for both PCE-Galerkin solutions, obtaining:
$$
\int_{I_\xi} \langle \mathcal{F}(\func{u}_N)(z), P(\func{u}_N(z) - \func{v}_N(z)) \rangle_{\R^n}\,\mathrm{d}z = 0, \quad \int_{I_\xi} \langle \mathcal{F}(\func{v}_N)(z), P(\func{u}_N(z) - \func{v}_N(z)) \rangle_{\R^n}\,\mathrm{d}z = 0.
$$
By the definition of the weighted inner product, subtracting these two equations yields:
$$
\int_{I_\xi} \langle \mathcal{F}(\func{u}_N)(z) - \mathcal{F}(\func{v}_N)(z), \func{u}_N(z) - \func{v}_N(z) \rangle_P\,\mathrm{d}z = 0.
$$
However, applying Equation~\eqref{eq:WeakIneq} to $\func{w}=\func{u}_N$ and $\func{u}=\func{v}_N$, since $\func{u}_N \neq \func{v}_N$, the integral must be strictly positive. This is a contradiction, leading to $\func{v}_N = \func{u}_N$.
\end{proof}

\begin{rmk}\label{rmk:uniquenss1d}
We emphasize that in the scalar case, where $n=1$, the multi-dimensional assumptions naturally reduce to standard uniform boundedness conditions, by choosing also the matrix $P=1$. If $\func{f}$ is continuous, admits a unique root at $\bar{\func{u}}(z)$ with a strictly positive derivative $\partial_u \func{f} \ge \lambda_{\min} > 0$, and satisfies $|\func{f}(u,\bar{\func{\mu}}+\sigma z)| \ge m$ for $|u| \ge R$, the topology forces $\func{f}(u)$ to be strictly positive for $u > \bar{\func{u}}(z)$ and strictly negative for $u < \bar{\func{u}}(z)$. 
Consequently, the inner product becomes 
$$
\func{f}(u,\bar{\func{\mu}}+\sigma z)(u-\bar{\func{u}}(z)) = |\func{f}(u,\bar{\func{\mu}}+\sigma z)||u-\bar{\func{u}}(z)|.
$$ 

In this scenario, the uniform boundedness away from zero automatically guarantees the radial coercivity condition, yielding $f(u,\bar{\mu}+\sigma z)(u-\bar{u}(z)) \ge m|u-\bar{u}(z)|$. 
Furthermore, the relative growth bound $|f(u,\bar{\mu}+\sigma z)| \le C_\rho f(u,\bar{\mu}+\sigma z)(u-\bar{u}(z))$ simplifies to $1 \le C_\rho |u-\bar{u}(z)|$, which is trivially satisfied for any distance $|u-\bar{u}(z)| \ge \rho$ by choosing $C_\rho = 1/\rho$. 

For $n \ge 2$, however, the vector field could theoretically have a large norm while rotating in a tangential direction around root, making the radial coercivity and relative growth constraints necessary to prevent the loss of uniqueness.
\end{rmk}
\section{Numerical Results}\label{sec:numerical-results}

In this section, we validate the theoretical framework established in Section~\ref{sec:theoretical-results} through numerical experiments. Our primary goals are to provide an efficient strategy for identifying branch-approximating solutions of the PCE-Galerkin system, and to numerically verify the theoretical convergence rates derived in Section~\ref{sec:theoretical-results} under varying degrees of analytical regularity.

As shown in Section~\ref{sec:ode-bifurcations}, the PCE-Galerkin system admits a vast set of algebraic roots, most of which are highly oscillatory. To isolate the branch-approximating solutions, we employ a degree continuation algorithm, summarized in Algorithm~\ref{alg:degree_continuation}, which bases its convergence properties on Corollary~\ref{cor:successive_sol}.

To capture multiple physical branches, the degree continuation algorithm is executed across several random initializations. For each execution, the procedure initializes at degree $N=0$, where the PCE-Galerkin system reduces to finding the constant roots of the averaged vector field, and the solver is seeded with a random initial guess $\vect{u}_{\text{init}} \in \R^n$. For $N=0$, the Galerkin projection resembles the deterministic problem evaluated at the mean parameter $\bar{\mu}$. Consequently, the Newton updates converge to a root that closely approximates an exact steady state at $\bar{\mu}$. In particular, if multiple initializations at $N=0$ reach a theoretical bound on the number of roots of the system, such as Bezout's Theorem~\ref{thm:Bezout} for polynomials fields, we are guaranteed to have found all the possible constant roots. If these remain distinct across the subsequent degree-continuation steps while satisfying convergence properties such as Corollary~\ref{cor:successive_sol}, all continuous branches over $I$ are recovered. 

\begin{algorithm}[b]
\caption{Degree continuation algorithm}
\label{alg:degree_continuation}
\begin{algorithmic}[1]
\Statex{\textbf{Input:} Galerkin residual maps $\{\func{F}_N:\R^{n(N+1)}\to\R^{n(N+1)}\}_{N=0}^{N_{\max}}$, maximum degree $N_{\max}$,} \LineComment{interval midpoint $\bar{\mu}$, standard deviation $\sigma$, random initial guess $\vect{u}_{\text{init}} \in \R^n$}
\Statex{\textbf{Output:} Coefficient vector $\hat{\vect{u}}_{N_{\max}}$ defining a branch-approximating PC solution}
\Statex
\State{$N \gets 0$}
\State{$\hat{\vect{u}}_0 \gets \mathtt{root}\big(\func{F}_0(\cdot; \bar{\mu}, \sigma), \vect{u}_{\text{init}}\big)$}\Comment{Baseline initialization}
\While{$N < N_{\max}$}
    \State $N \gets N + 1$
    \State {$\vect{u}_{\text{init}} \gets [\hat{\vect{u}}_{N-1}, \vect{0}]$}\Comment{Predictor}
    \State{$\hat{\vect{u}}_N \gets \mathtt{root}\big(\func{F}_N(\cdot; \bar{\mu}, \sigma), \vect{u}_{\text{init}}\big)$}\Comment{Corrector}
\EndWhile
\end{algorithmic}
\end{algorithm} 

To increase accuracy, we raise the polynomial degree incrementally. According to Corollary~\ref{cor:successive_sol}, the distance between successive approximations behaves as $\|\func{u}_N - \func{u}_{N+1}\|_{L^2} = \mathcal{O}(N^{-s})$. This theoretical guarantee justifies using the computed coefficients at degree $N$, appended with a zero for the highest-degree term, as the exact initial guess for the Newton solver at the augmented system of degree $N+1$. This initialization ensures that the predictor falls strictly within the convergence neighborhood of the higher-degree system, guaranteeing stable convergence toward the exact physical branch $\bar{\func{u}}$.

We note that different random seeds may converge to the same solution; therefore, we only collect unique roots, and if some a-priori knowledge on the number of branches is given, it can be used as a stopping criterion until reaching the maximum number of initializations, closely resembling the procedure adopted for the convergence of the deflation method~\cite{FarrellDeflationTechniquesFinding2015}. In this way, branches are recovered globally inside the interval $I$ with the accuracy of degree-$N$ approximations, bypassing the root proliferation of the Galerkin system ($d^{N+1}$ in the case of degree $d$ algebraic systems). As with deflation, recovering all branches cannot be guaranteed for general nonlinear systems; however, rigorous certifications are granted under hypotheses of Theorems~\ref{thm:non_existence}--\ref{thm:uniqueness} and Proposition~\ref{prop:limits_galerkin} or when the distinct roots reach a theoretical bound as noted above.

The first two examples, the pitchfork and S-shaped normal forms, serve as fundamental scalar benchmarks to test the spectral convergence rates and the uniqueness guarantees. The subsequent two examples extend our framework to more complex, multi-dimensional models of physical and biological interest: the genetic toggle switch and the Lorenz chaotic system. Finally, the last example investigates the method's performance on a degenerate system exhibiting an infinite number of steady states within a bounded parameter regime.

\subsection{Pitchfork Bifurcation}\label{subsec:pitchfork}
We first apply this methodology to a scalar example, the supercritical pitchfork normal form, defined by the field
\begin{equation}
    f(u,\mu)=-u^3+\mu u,\qquad \partial_uf(u,\mu)=-3u^2+\mu.
    \label{eq:pitchfork}
\end{equation}

As it appears in Figure~\ref{fig:pitchfork}, this system undergoes a bifurcation at $\mu=0$, where the stable branch $u=0$ becomes unstable and two stable branches $u=\pm\sqrt{\mu}$ emerge for $\mu>0$. 

\begin{figure}[t]
    \centering
    \includegraphics[width=0.5\linewidth]{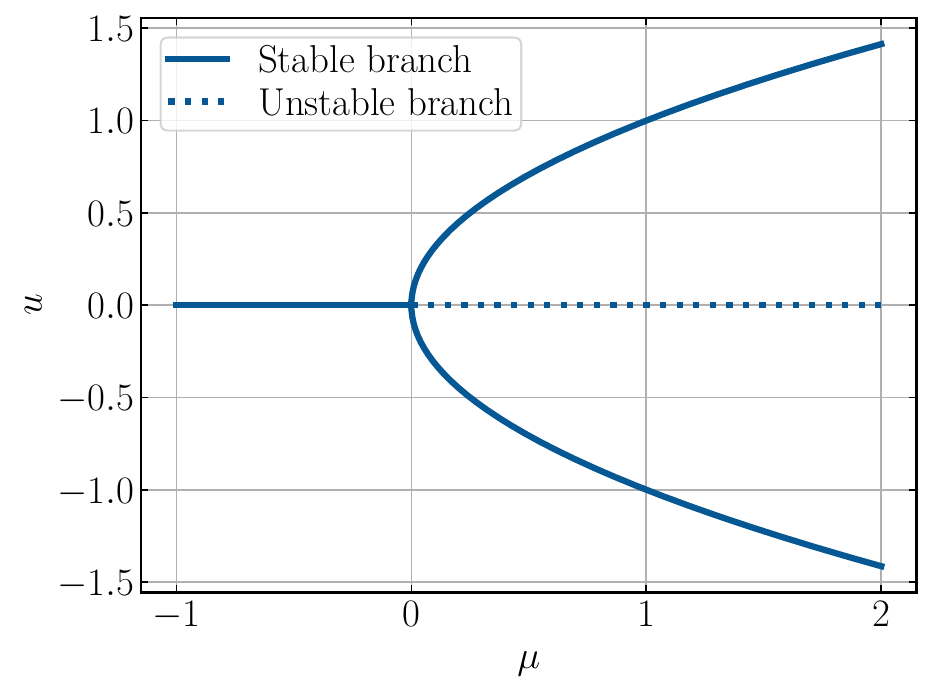}
    \caption{Pitchfork normal form bifurcation diagram.}
    \label{fig:pitchfork}
\end{figure}

By formulating the parameter as $\mu(\xi)=\bar{\mu}+\sigma\xi$, we explore three distinct parametric intervals to validate the theoretical bounds on the convergence rate $E(N)$, which depends on the Sobolev regularity $s$ of the exact branch within the evaluated interval.

First, we consider a \emph{smooth region}, where no bifurcation singularities appear. We choose the parameters $\bar{\mu}$ and $\sigma$ such that the domain entirely lies in the positive real numbers, yielding, for example, the interval $I=[0.2, 1]$, where the polynomial field $f$ satisfies~\eqref{eq:LipLocUnif} and the branches $\bar{u}(\mu)=\pm\sqrt{\mu}$ are analytical~\eqref{ass:A2b}. From now on, we focus on the branch $\bar{u}(\mu)=\sqrt{\mu}$, though the same can be applied to the negative one. Since the derivative of the field is $\partial_u f(\bar{u}(\mu),\mu)=-2\mu$, which is a decreasing function, we observe that $\partial_u f(\bar{u}(\mu),\mu)\leq-0.4$ for any $\mu\in I$~\eqref{eq:uniform_negative_definiteness}. This ensures uniform negative definiteness, satisfying the theoretical hypotheses of the Newton--Kantorovich convergence framework. 

As illustrated in Figure~\ref{fig:no_bifurcation}, the spectral approximations align perfectly with the exact branch. From low degrees (blue curves) up to $N=30$ (dashed orange line), the numerical solutions rapidly converge to the exact branch (solid black line); notably, the solution for $N=5$ is already indistinguishable from the one obtained with $N=30$. In the resulting error tracking, the $L^2$-norm of the residual and the $L^2$ and $L^\infty$ norms of the difference $u_N-u_{LS,N}$ decay exponentially, confirming the optimal spectral accuracy of the PC formulation in analytic regions, perfectly matching the theoretical reference bound $\mathcal{O}(e^{-N})$ (dash-dotted black line).

\begin{figure}[t]
    \centering
    \begin{minipage}{0.48\textwidth}
        \centering
        \includegraphics[width=\linewidth]{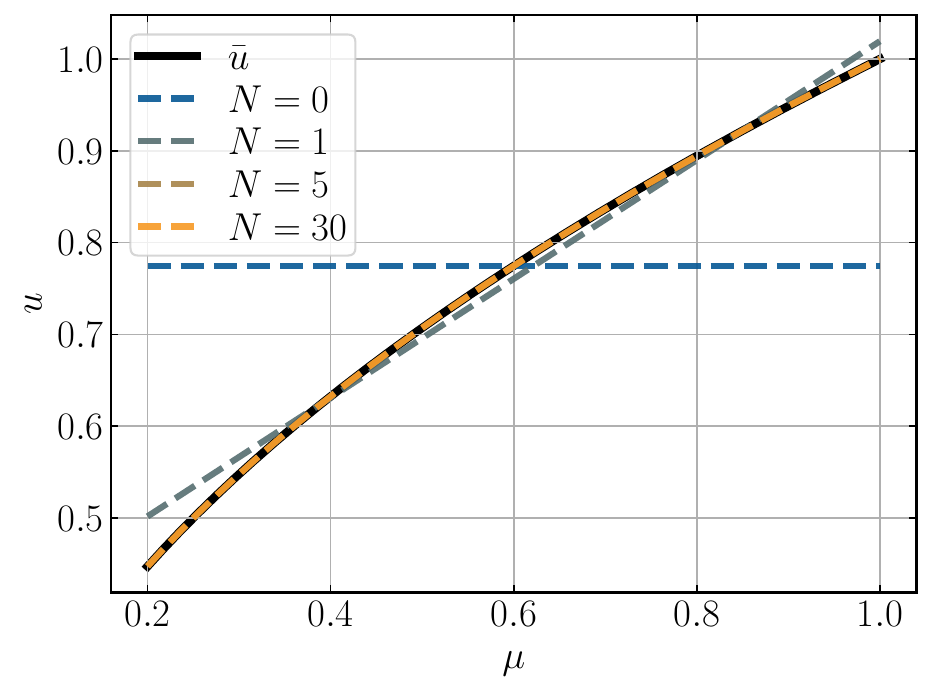}
    \end{minipage}\hfill
    \begin{minipage}{0.48\textwidth}
        \centering
        \includegraphics[width=\linewidth]{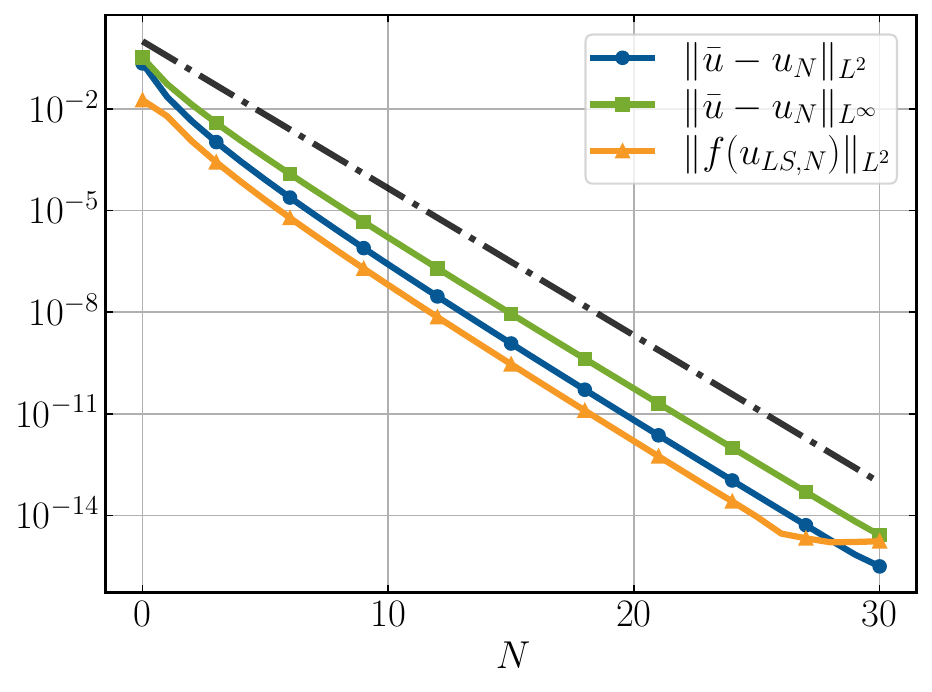}
    \end{minipage}
    \caption{PCE-Galerkin approximations (\emph{left}) and convergence analysis (\emph{right}) in the \emph{smooth regime}, $I=[0.2, 1]$.}
    \label{fig:no_bifurcation}
\end{figure}

Next, we shift the interval to a \emph{boundary singularity region}, setting $I=[0, 1]$. Here, the exact branch $\bar{u}(\mu) = \sqrt{\mu}$ exhibits a loss of differentiability at $\mu=0$, violating~\eqref{ass:A2a}. The derivative of the system approaches zero, violating the uniform negative definiteness condition~\eqref{eq:uniform_negative_definiteness} since $\lambda_{\min}=0$. As predicted, this singularity degrades the convergence rate from exponential to algebraic. Furthermore, the infinite derivative at $\mu=0$ makes uniform approximation challenging, causing the $L^\infty$ error to be higher than the $L^2$ error due to the difficulty of matching the exact solution near the singularity. Despite the theoretical breakdown, the solver retains the branch without diverging.

This stability and the specific algebraic decay rates align with Remark~\ref{rmk:superconvergence}. In fact, since the fractional singularity $\alpha=1/2$ is located at the domain's endpoint, the uniform error decays at the rate of $\mathcal{O}(N^{-2\alpha}) = \mathcal{O}(N^{-1})$. As shown in Figure~\ref{fig:bifurcation_extremum}, the PCE approximations (up to $N=30$) and the empirical metrics reflect these bounds, demonstrating robust convergence despite the lack of strict Sobolev regularity. Notably, the errors conform to the theoretical algebraic behavior of $\mathcal{O}(N^{-1})$ (dashed black lines), and the Galerkin residual achieves a faster decay rate of $\mathcal{O}(N^{-2})$ (dotted black line).

\begin{figure}[t]
    \centering
    \begin{minipage}{0.48\textwidth}
        \centering
        \includegraphics[width=\linewidth]{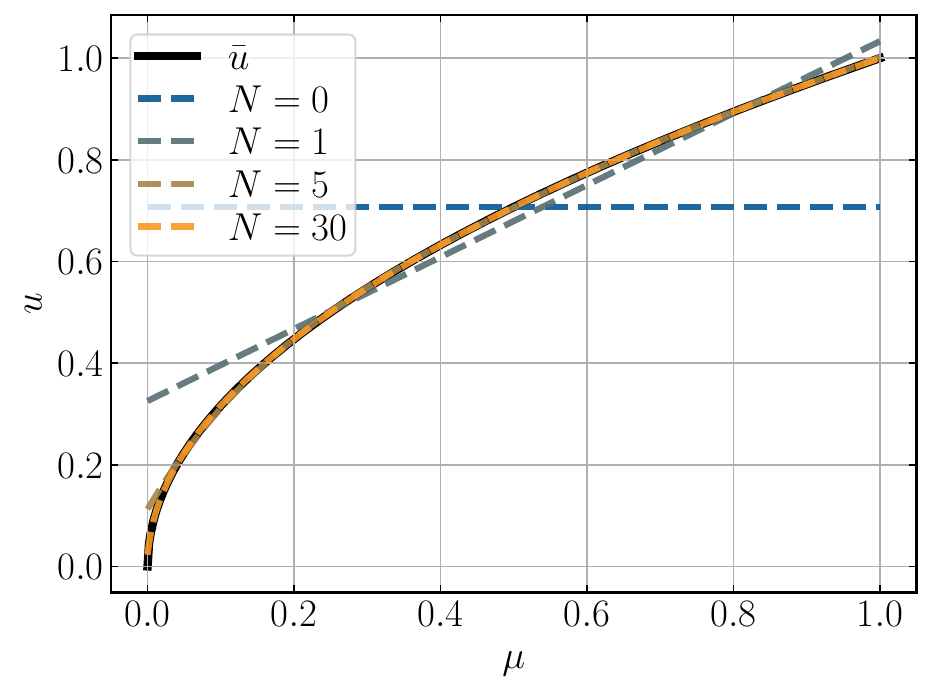}
    \end{minipage}\hfill
    \begin{minipage}{0.48\textwidth}
        \centering
        \includegraphics[width=\linewidth]{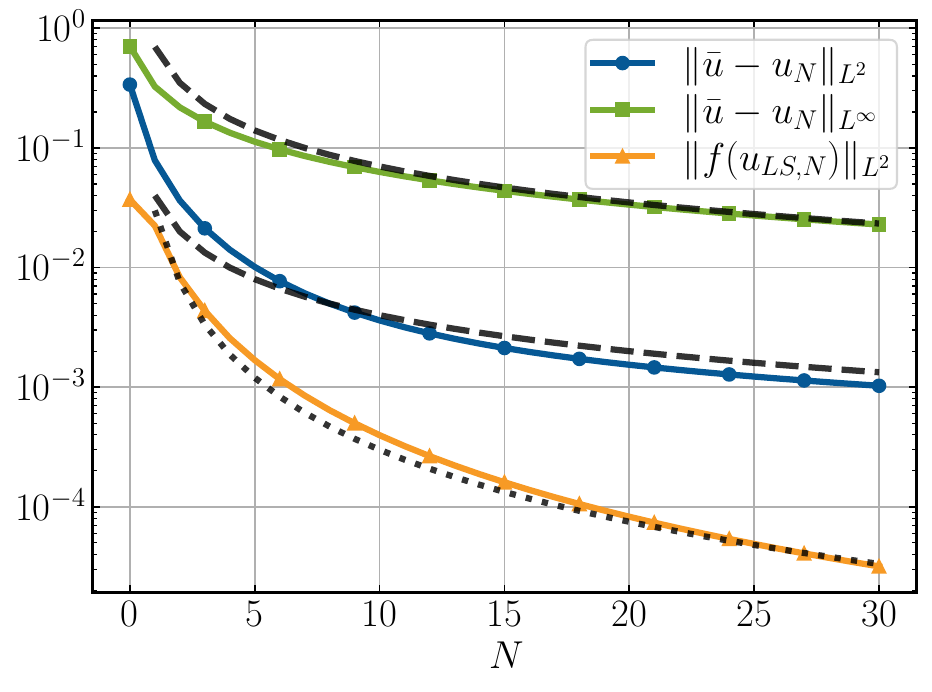}
    \end{minipage}

    \caption{PCE-Galerkin approximations (\emph{left}) and convergence analysis (\emph{right}) in the \emph{boundary singularity region}, $I=[0, 1]$.}
    \label{fig:bifurcation_extremum}
\end{figure}

Finally, we consider the \emph{interior singularity regime}, where the bifurcation point is situated strictly within the parameter interval chosen as $I=[-1, 3]$. In this scenario, the branches do not terminate at the bifurcation point, and our method captures them as continuous curves extended over the entire interval. By denoting the indicator function with $\chi_{[0, \infty)}$, the three continuous branches can be expressed as $\bar{u}_{\pm}(\mu) = \pm \sqrt{\mu} \, \chi_{[0, 3)}(\mu)$, and $\bar{u}_0(\mu) = 0$.

Focusing on the branch $\bar{u}_+$, the eigenvalue of the linearized system crosses zero inside the parametric interval, thereby violating local non-degeneracy assumption~\eqref{eq:uniform_negative_definiteness} as well as~\eqref{ass:A2a} at $\mu=0$. Despite this theoretical breakdown, applying degree continuation allows the PC Galerkin formulation to seamlessly capture this branch, as well as the other two branches $\bar{u}_0$ and $\bar{u}_-$. However, as illustrated in Figure~\ref{fig:bifurcation_interval}, the PCE approximations (shown up to $N=30$) reveal that the non-smooth transition at $\mu=0$ induces localized oscillations. This is a classic manifestation of the Gibbs phenomenon in spectral methods~\cite{gottlieb1997gibbs, hesthaven2007spectral}.

This internal loss of regularity severely restricts the convergence speed. Because the fractional singularity $\alpha=1/2$ is now located in the interior of the domain rather than at the boundary, the effects discussed in Remark~\ref{rmk:superconvergence} are significantly weakened. Consequently, the theoretical decay of the $L^2$ error, $L^\infty$ error, and projection residual is bounded by $\mathcal{O}(N^{-1/2})$. Furthermore, the observed decay is not strictly monotonic with respect to the polynomial degree $N$. Although the theoretical upper bound is $\mathcal{O}(N^{-1/2})$ (dash-dotted black line), the residual exhibits an asymptotic decay rate closer to $\mathcal{O}(N^{-1})$ (dashed black line).

\begin{figure}[t]
    \centering
    \begin{minipage}{0.48\textwidth}
        \centering
        \includegraphics[width=\linewidth]{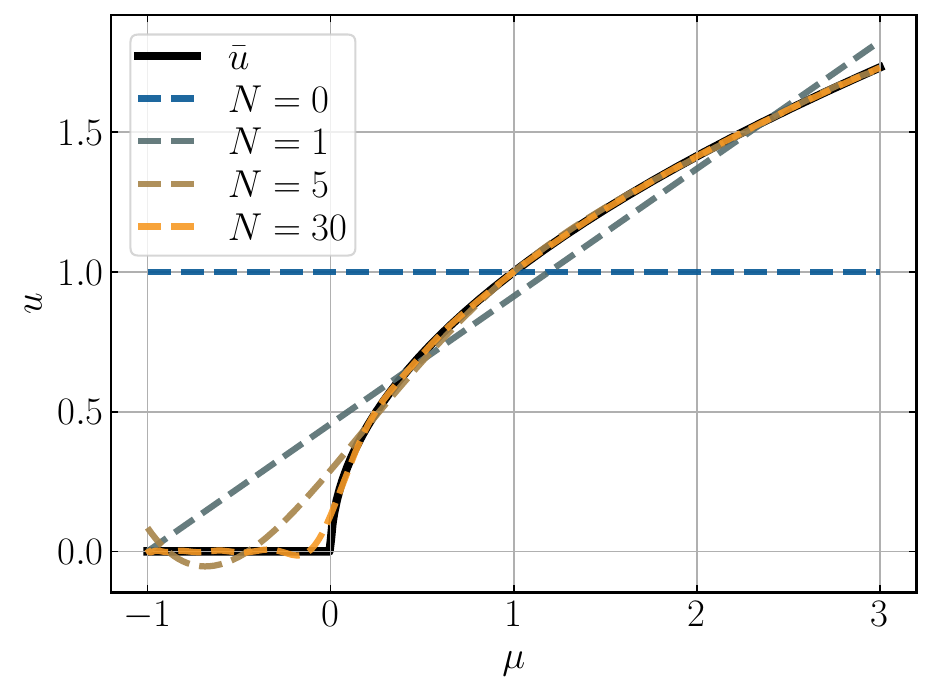}
    \end{minipage}\hfill
    \begin{minipage}{0.48\textwidth}
        \centering
        \includegraphics[width=\linewidth]{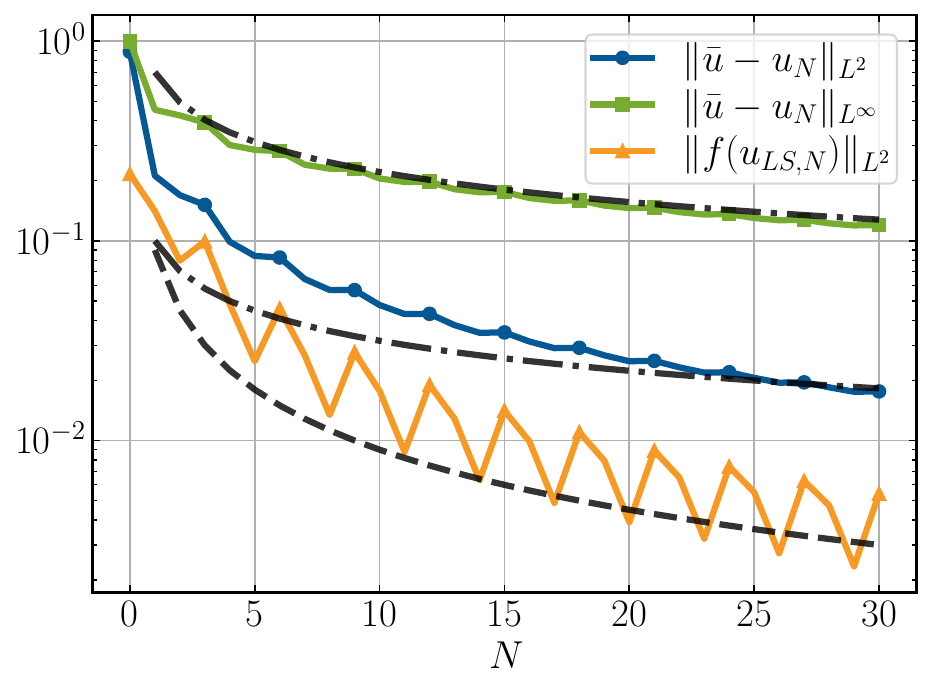}
    \end{minipage}

    \caption{PCE-Galerkin approximations (\emph{left}) and convergence analysis (\emph{right}) in the \emph{interior singularity regime}, $I=[-1, 3]$.}
    \label{fig:bifurcation_interval}
\end{figure}

\begin{rmk}
The Newton--Kantorovich condition $h_N = \beta_N L_N \eta_N \leq 0.5$ is sufficient but not necessary. Near singularities, the unbounded growth of the inverse derivative causes $h_N$ to exceed this threshold. Nevertheless, our numerical tests show that the continuation converges to the correct branches, demonstrating empirically that the basin of attraction of the Galerkin formulation may extend beyond the local theoretical bounds provided by the theorem.
\end{rmk}

Finally, we assess the computational efficiency of the proposed framework in the \emph{interior singularity regime} $I=[-1,3]$ for the branch $\bar{u}_+$. Figure~\ref{fig:pitchfork_efficiency} compares the runtime of Algorithm~\ref{alg:degree_continuation} against classical \emph{simple continuation} and \emph{pseudo-arc-length continuation}, whose step-sizes are respectively $\Delta\mu$ and $\Delta s$, both starting from $\mu=-1$ and equipped with branch switching when $\partial_u f$ changes sign at $\mu=0$. For all methods, we report the median runtime together with the 5th--95th percentile bands (shaded areas) over 200 independent runs, accounting for hardware timing fluctuations due to the very short execution times; as visible in the plots, the narrowness of the bands demonstrates very low time variance. To ensure a fair comparison, each timing measurement covers the full execution of Algorithm~\ref{alg:degree_continuation} from degree $0$ to $N\in\{0,\dots,30\}$, including the generation of the Gauss--Legendre quadrature rule and polynomial basis.

Classical methods are faster only at coarse step sizes, when $\Delta\mu, \Delta s \in\{10^{-1}, 10^{-2}\}$, with their computational cost scaling inversely with the step size. Specifically, the PCE runtime crosses the $\Delta\mu = 10^{-2}$ and $\Delta s = 10^{-2}$ thresholds around $N = 1$ and $N = 5$, respectively, and remains faster than simple continuation at $\Delta\mu = 10^{-3}$ up to $N = 21$. Across the tested range up to $N=30$, the PCE formulation stays below pseudo-arc-length at $\Delta s = 10^{-3}$ and achieves speedups of roughly $6\times$ and $17\times$ over simple continuation and pseudo-arc-length at fine step sizes $\Delta\mu,\Delta s=10^{-4}$, while providing a global representation over $I$ of the branch without requiring explicit branch switching and step-size tuning.

\begin{figure}[t]
  \centering
  \includegraphics[width=0.55\linewidth]{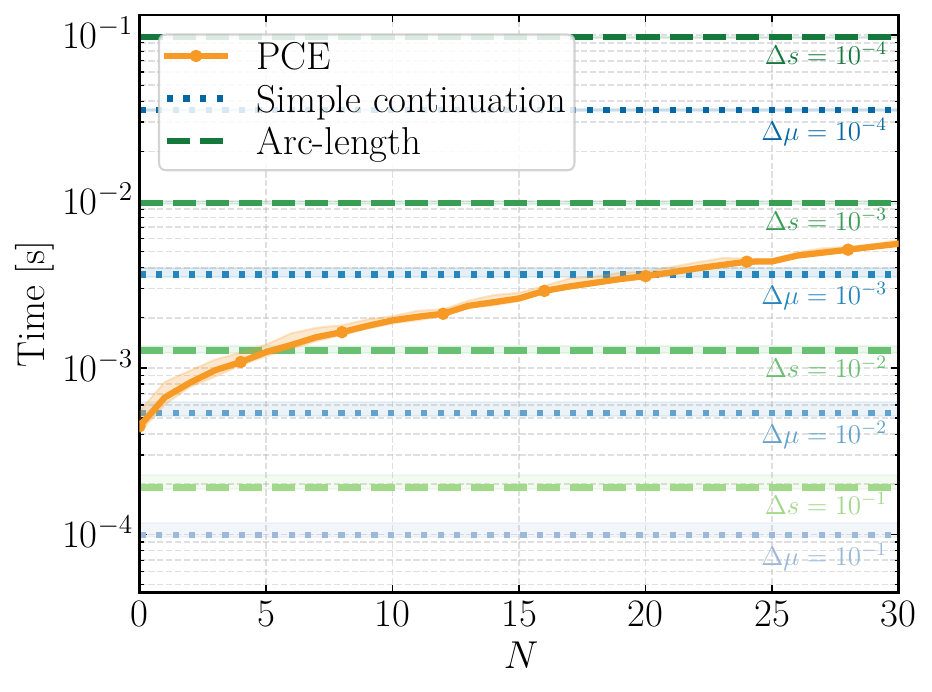}
  \caption{Runtime comparison on $I=[-1,3]$ for $\bar{u}_+$: PCE-Galerkin versus simple and pseudo-arc-length continuation.}
  \label{fig:pitchfork_efficiency}
\end{figure}

\subsection{S-Shaped Bifurcation}
To explicitly demonstrate the application of the global uniqueness guarantees established in Theorem~\ref{thm:uniqueness} in a nontrivial uniqueness region of a bifurcating system, we consider the S-shaped bifurcation scalar system
\begin{equation}
    f(u,\mu)=-u^3+u+\mu,\qquad\partial_u f(u,\mu)=-3u^2+1.
    \label{eq:sshaped}
\end{equation}

By setting $\partial_u f = 0$, we identify two saddle-node bifurcation points corresponding to the state-parameter pairs $(u_{\mp}, \mu_{\pm}) = \left(\mp \frac{1}{\sqrt{3}}, \pm \frac{2}{3\sqrt{3}}\right)$. Inside the interval $\mu \in (\mu_-, \mu_+)$, the system exhibits three distinct coexisting branches, as it appears from Figure~\ref{fig:s_shaped}. However, for $|\mu| > \frac{2}{3\sqrt{3}}$, it consists of a single, unique real root.
\begin{figure}[t]
    \centering
    \includegraphics[width=0.5\linewidth]{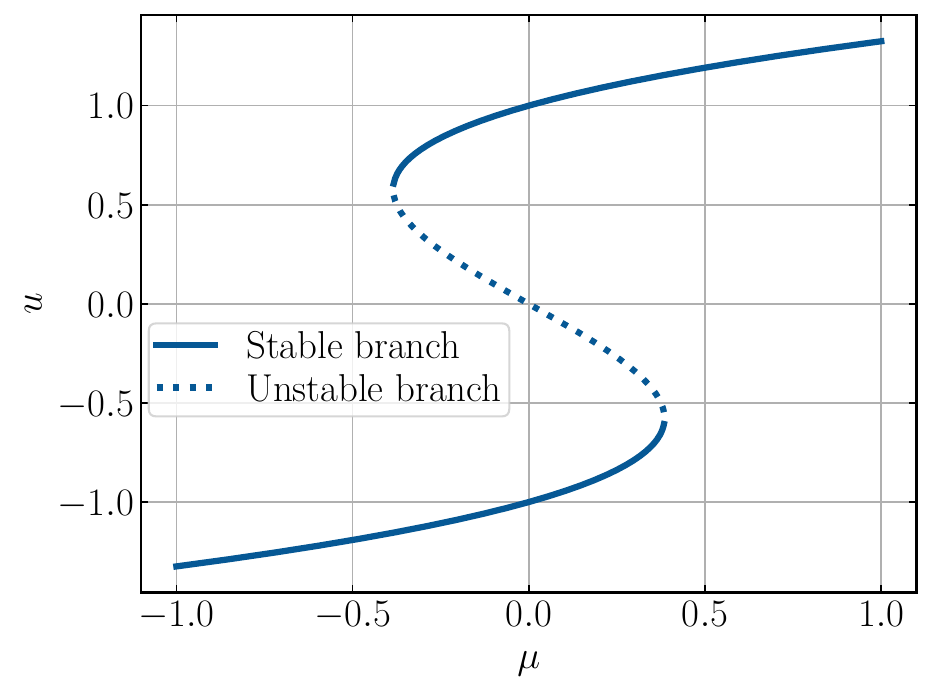}

    \caption{S-shaped dynamical system bifurcation diagram.}
    \label{fig:s_shaped}
\end{figure}

We consider a parameter domain that lies in this uniqueness region, choosing for instance $I=[0.5, 1.5]$. There, $f$ satisfies~\eqref{eq:LipLocUnif} and the continuous problem admits a unique, analytical solution branch $\bar{u}(\mu)$~\eqref{ass:A2b}. Furthermore, since $\bar{u}(\mu)>\frac{1}{\sqrt{3}}$ for all $\mu \in I$, the derivative is uniformly negative definite evaluated along the branch, i.e.\ $\partial_u f(\bar{u}(\mu), \mu) = -3\bar{u}(\mu)^2 + 1 \le \lambda_{\min} < 0,\ \forall \mu \in I$; this satisfies assumption ~\eqref{eq:uniform_negative_definiteness} with $P=1$ in Theorem~\ref{thm:uniqueness}.

Then, since $I$ is bounded and $f(u, \mu)$ is a scalar polynomial, as $|u| \to \infty$, the magnitude of the field $|f(u, \mu)|$ diverges to $+\infty$. Considering Remark~\ref{rmk:uniquenss1d}, this asymptotic polynomial growth satisfies the uniform radial coercivity constraint $|f| \ge m$ outside any neighborhood of the exact root $\bar{u}(\mu)$ for an appropriate $m$~\eqref{ass:A4a}--\eqref{ass:A4b}. Consequently, all the hypotheses of the uniqueness Theorem~\ref{thm:uniqueness} are satisfied. This provides a powerful theoretical fact: there exists an integer $\bar{N}$ such that for all $N \ge \bar{N}$, the finite-dimensional PCE-Galerkin problem~\eqref{eq:PB2} admits \emph{exactly one} polynomial solution $u_N \in \Pi_N^1$. As shown in Figure~\ref{fig:s_shaped_uniqueness}, the computed PCE approximations (shown up to $N=20$) converge to the continuous branch $\bar{u}$ at an exponential rate. The empirical error and Galerkin residual decay perfectly match the theoretical reference bound $\mathcal{O}(e^{-N})$ (dash-dotted black line), achieving machine precision by $N\approx17$.
\begin{figure}[t]
    \centering
    \begin{minipage}{0.48\textwidth}
        \centering
        \includegraphics[width=\linewidth]{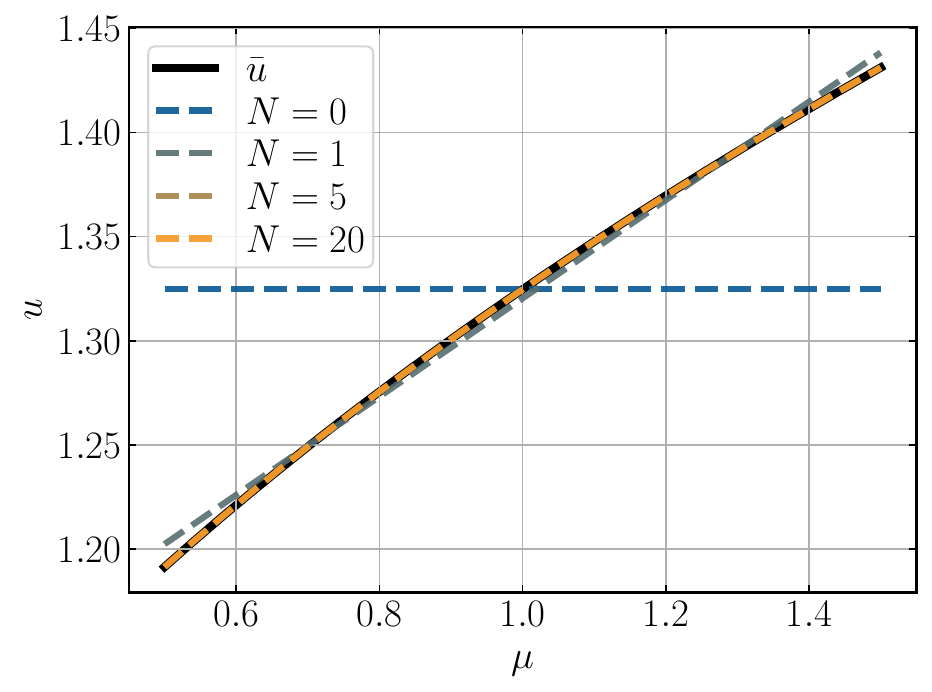}
    \end{minipage}\hfill
    \begin{minipage}{0.48\textwidth}
        \centering
        \includegraphics[width=\linewidth]{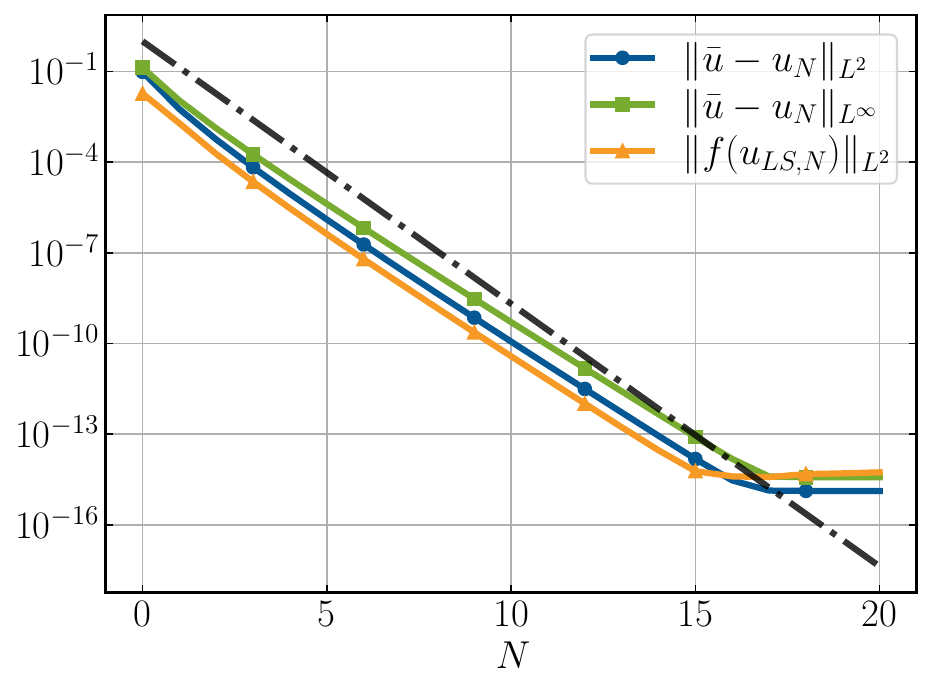}
    \end{minipage}
    \caption{PCE-Galerkin approximations (\emph{left}) and convergence analysis (\emph{right}) in a region of uniqueness, $I=[0.5, 1.5]$.}
    \label{fig:s_shaped_uniqueness}
\end{figure}

Moreover, we use this benchmark to examine the behavior and limitations of the proposed framework when dealing with disconnected or isolated branches, complementing the theoretical obstruction established in Proposition~\ref{prop:limits_galerkin}. While the consistency and convergence results of Section~\ref{sec:theoretical-results} hold for any branch $\overline{u}\in\mathcal{C}(I;\mathbb{R}^n)$ supported over the entire interval $I$, regardless of whether it is connected to other branches via bifurcations, the situation changes when a branch originates or terminates inside $I$, as occurs for the folds of the S-shaped curve at $\mu_{\pm} = \pm\frac{2}{3\sqrt{3}}$. In this case, no global continuous branch exists on $I$ for those equilibria, and Proposition~\ref{prop:limits_galerkin} denies the existence of any uniformly bounded sequence of Galerkin solutions converging in a neighborhood of the fold. 

This is confirmed in Figure~\ref{fig:s_isolated} for the interval $I=[0,0.4]$, where only the upper branch is defined across $I$ and the lower two terminate at $\mu_+ = \frac{2}{3\sqrt{3}}$: different initializations for Algorithm~\ref{alg:degree_continuation} capture all three equilibria at $N=0$ and track the two lower branches up to $N=4$ by extrapolating past $\mu_+$; at $N=5$, however, the solver can no longer sustain the lower fold, and all three trajectories initialized from degree $N=4$ converge to the unique entirely supported branch in the interval $I$. Consequently, while disconnected or isolated structures may be partially extrapolated at low polynomial degrees, restricting or partitioning the parameter domain into sub-intervals where the target branches are globally defined is required to capture them reliably. At the same time, this finite low-degree tracking provides a useful diagnostic indicator of an internal fold, signaling where the parameter interval should be subdivided to accurately resolve the bifurcating structure.

\begin{figure}[t]
  \centering
  \includegraphics[height=3.52cm]{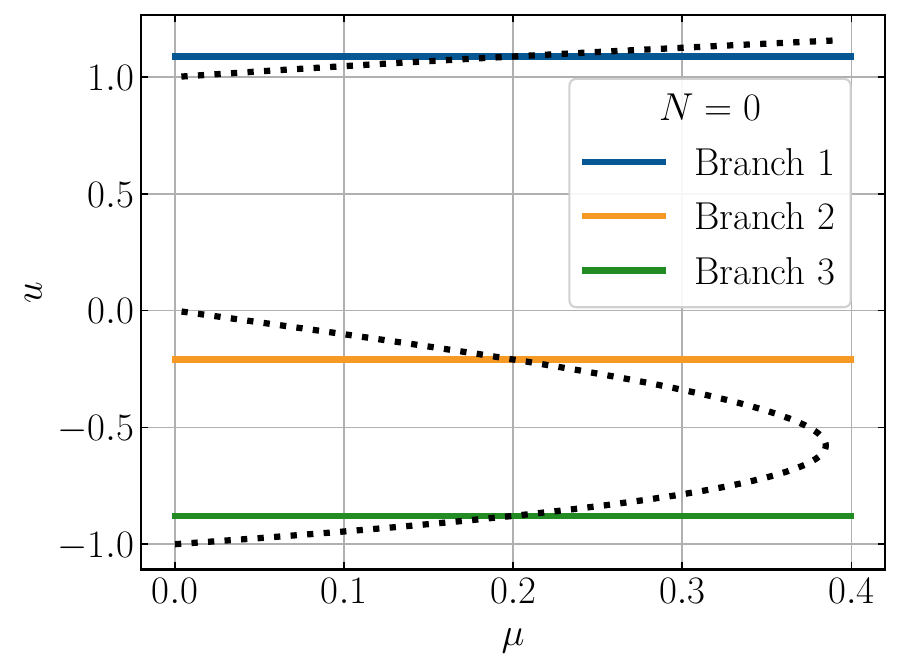}
  \hfill
  \includegraphics[height=3.52cm]{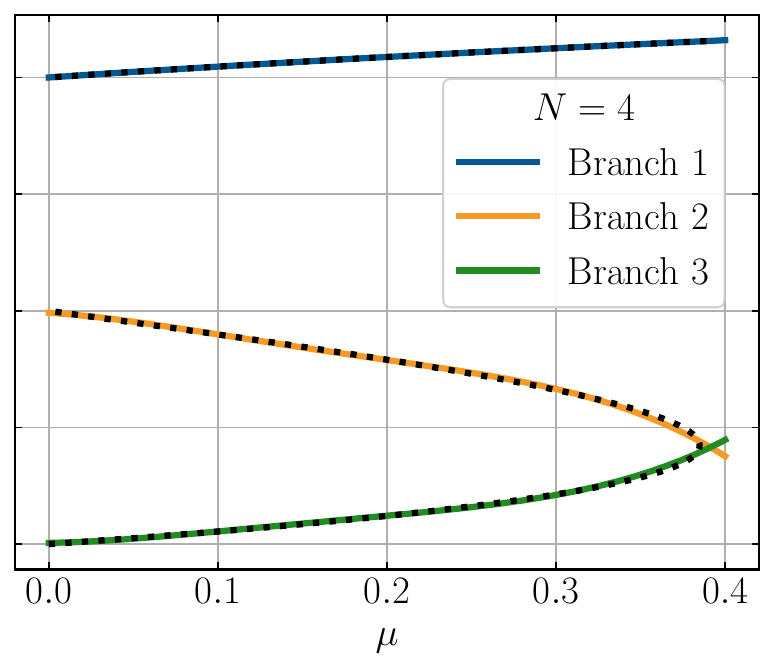}
  \hfill
  \includegraphics[height=3.52cm]{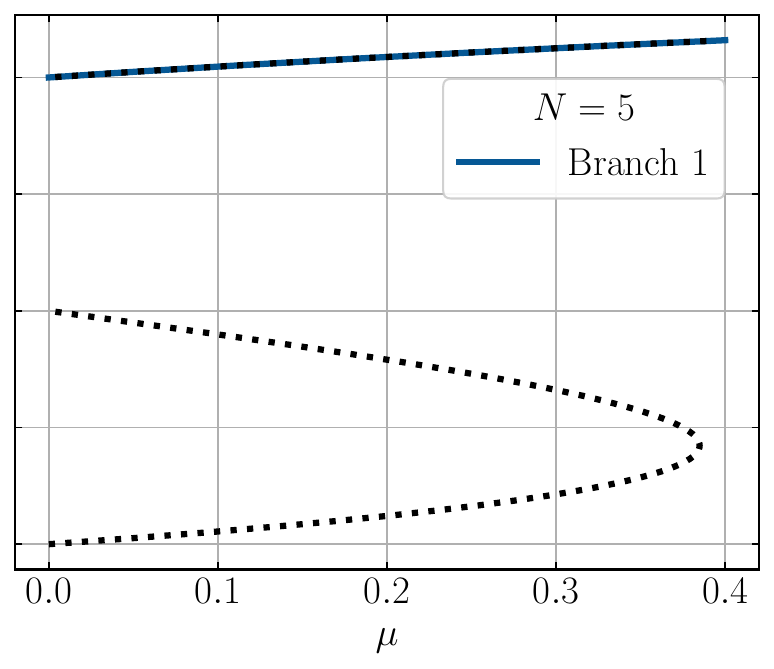}
  \caption{Three PCE-Galerkin approximations for $N=0,4,5$ in a region of a disconnected branch.}
  \label{fig:s_isolated}
\end{figure}

\subsection{Genetic Toggle Switch system}

To demonstrate the effectiveness of our framework in a more complex setting, we extend our validation beyond single-bifurcation scalar problems by considering the genetic toggle switch model~\cite{gardner2000}. This multi-dimensional system represents a bistable memory unit, consisting of two mutually inhibitory components whose dynamics depend fundamentally on cooperative non-linear repression and balanced synthesis rates~\cite{bollenbach2023}. 

Focusing on its mathematical properties, we consider a symmetric formulation of the classical continuous model, whose equations and steady states are defined by the vector field $\func{f}: \mathbb{R}^2 \times I \to \mathbb{R}^2$:
$$
\begin{cases}
    \dot{x} = -x + \frac{\mu}{1+y^2} \\
    \dot{y} = -y + \frac{\mu}{1+x^2}
\end{cases}
\quad\implies\quad \func{f}(\vect{u}, \mu) = \begin{pmatrix}
    -x + \frac{\mu}{1+y^2} \\
    -y + \frac{\mu}{1+x^2}
\end{pmatrix}.
$$
Here $\vect{u} = (x,y)^\top$ represents the vector of repressor concentrations, the quadratic rational terms model the mutual cooperative repression between the two genetic components, and we treat the effective transcription rate $\mu$ as a stochastic parameter $\mu(\xi) = \bar{\mu} + \sigma \xi$.
While the biological context restricts the transcription rate to $\mu \ge 0$, we extend our analysis to $I = [-6, 15]$ to demonstrate our method's capacity to handle multiple bifurcations.

To compute the equilibrium manifold, we set the vector field to zero and subtract the second equation from the first,  and obtain the following factorized equation $(x - y)(1 - xy) = 0$.
First, assuming $x = y$, the derivative of the cubic equation is strictly positive, thus the equation has exactly one real root $\bar{x}_0(\mu)$ for any $\mu \in \mathbb{R}$, defining the symmetric equilibrium branch $\bar{\func{u}}_0(\mu) = (\bar{x}_0(\mu), \bar{x}_0(\mu))$.

Second, assuming $x \neq y$ implies $xy = 1$, which, substituted back into the dynamics $\func{f}$ yields a quadratic equation providing the asymmetric equilibrium states $x_{\pm} = (\mu \pm \sqrt{\mu^2 - 4})/2$, that exist if and only if $|\mu| \ge 2$. Consequently, the system undergoes two pitchfork bifurcations at the critical values $\mu_\pm= \pm2$. Finally, noting that $1/x_{\pm} = x_{\mp}$, the solutions satisfy $y_{\pm} = x_{\mp}$, allowing us to define the two asymmetric branches as $\bar{\func{u}}_\pm(\mu) = (x_\pm(\mu), x_\mp(\mu))$.

To reconstruct the bifurcation diagram in an interval containing both bifurcations we apply the polynomial degree continuation Algorithm~\ref{alg:degree_continuation}, and observe that the non-linear solver tracks both the symmetric branch $\bar{\func{u}}_{0}$ and the asymmetric branches. Indeed, as illustrated in Figure~\ref{fig:toggle_pitchfork}, the branch-approximating solutions of the PCE-Galerkin system (shown as colored curves for $N=20$) successfully reconstruct the multi-dimensional double-pitchfork structure, precisely overlapping the true analytical branches (solid black lines). The first panel displays the projection of the bifurcation diagram in the $(\mu,x)$ plane, while the second panel represents the bifurcation in the full space. Notice that, due to the underlying system symmetry, the $x_{\pm}$ branches exactly mirror the $y_{\mp}$ branches.

Regarding the tracking of these solutions, it is worth noting that the parameter interval $I$ is partitioned by the bifurcations into three regions, characterized by three, one, and three steady states, respectively. Theoretically, this structure allows for up to nine distinct continuous branch combinations over the global domain. However, by starting the continuation algorithm at $N=0$ with the three roots of the deterministic system at $\bar{\mu}$, the Newton updates successfully isolate three of these nine branches, effectively reconstructing the multiple bifurcating structure.

\begin{figure}[t]
    \centering
    \begin{minipage}{0.48\textwidth}
        \centering
        \includegraphics[width=\linewidth]{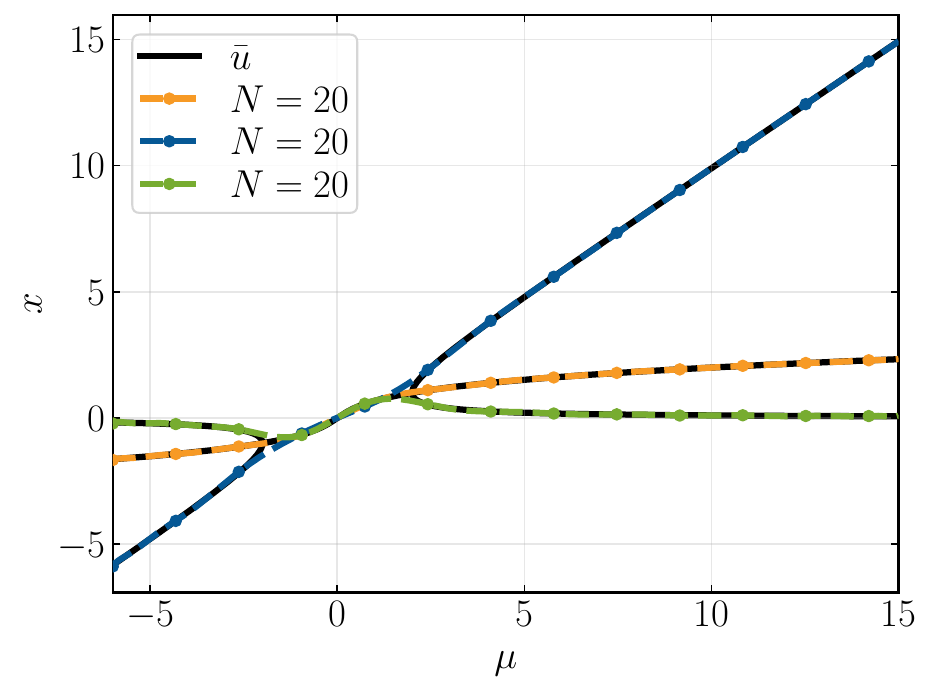}
    \end{minipage}\hfill
    \begin{minipage}{0.48\textwidth}
        \centering
        \includegraphics[width=\linewidth]{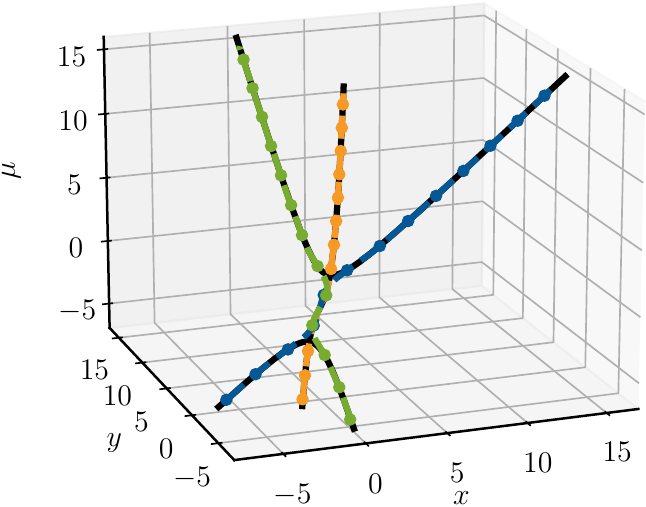}
    \end{minipage}

    \caption{PCE-Galerkin approximations in $I=[-6,15]$, with two pitchfork bifurcations.}
    \label{fig:toggle_pitchfork}
\end{figure}

Finally, we aim to verify the global uniqueness of the PCE-Galerkin approximation within parameter regimes where the continuous problem is uniquely solvable. Thus, let us verify the hypotheses of Theorem~\ref{thm:uniqueness} for any parameter compact interval contained within $I = (-2, 2)$, where the smooth rational structure of $\func{f}$ and the analyticity of the symmetric branch $\bar{\func{u}}_0$ immediately ensure~\eqref{eq:LipLocUnif}--\eqref{ass:A2b}. Fixing $\bar{\mu}$ and $\sigma$ such that $\bar{\mu}+\sigma z\in I$ for all $z \in I_\xi$, we compute the Jacobian matrix evaluated along the unique symmetric branch $\bar{\func{u}}_0$:
\begin{equation*}
    J(z) = D_{\vect{u}}\func{f}(\bar{\func{u}}_0(z), \mu(z)) = \begin{pmatrix}
    -1 & -k(z) \\
    -k(z) & -1
    \end{pmatrix}, \qquad \text{where } k(z) = \frac{2\mu(z) \bar{x}_0(z)}{(1+\bar{x}_0(z)^2)^2}.
\end{equation*}
By substituting the steady-state relation $\mu(z) = \bar{x}_0(z)(1+\bar{x}_0(z)^2)$ from the dynamics, the off-diagonal term simplifies to $k(z) = \frac{2\bar{x}_0(z)^2}{1+\bar{x}_0(z)^2}$. The eigenvalues $\lambda_{\pm} = -1 \pm k(z)$ of $J(z)$ are both negative, and since this matrix is symmetric, this directly proves negative definiteness of the matrix, bypassing the need to analyze its symmetric part or introduce weighted $P$-norms.

Using the relation $\mu(z)=\bar{x}_0(z)^3+\bar{x}_0(z)$, the bounds imposed on $\bar{x}_0(z)$ map to the parameter space, enforcing $\mu(z) \in (-2, 2)$. Therefore, for any compact interval contained within $I$, the maximum eigenvalue of the Jacobian remains negative and uniformly bounded away from zero, ensuring that the local monotonicity condition~\eqref{eq:uniform_negative_definiteness}, with $P=I_2$, is satisfied.

Regarding the remaining conditions—namely, uniform radial coercivity~\eqref{ass:A4a} and the relative growth bound~\eqref{ass:A4b}—Theorem~\ref{thm:uniqueness} provides sufficient hypotheses. Enforcing these analytical bounds for the toggle switch vector field would restrict the certified uniqueness region to a narrower sub-interval $J \subsetneq (-2, 2)$.
Nevertheless, we empirically observe that the numerical method behaves robustly across the entire interval $I = (-2,2)$. As before, the first panel of Figure~\ref{fig:toggle_unique} displays the $(\mu,x)$ plane, while the second panel represents the full space. The degree-continuation solver converges to the unique symmetric branch without encountering any spurious artifacts. This evidence suggests that the theorem's sufficient hypotheses are conservative for this particular system and could potentially be weakened, demonstrating that the practical basin of attraction of the PCE-Galerkin formulation extends well beyond the boundaries dictated by the strict theoretical analysis.

\begin{figure}[t]
    \centering
    \begin{minipage}{0.48\textwidth}
        \centering
        \includegraphics[width=\linewidth]{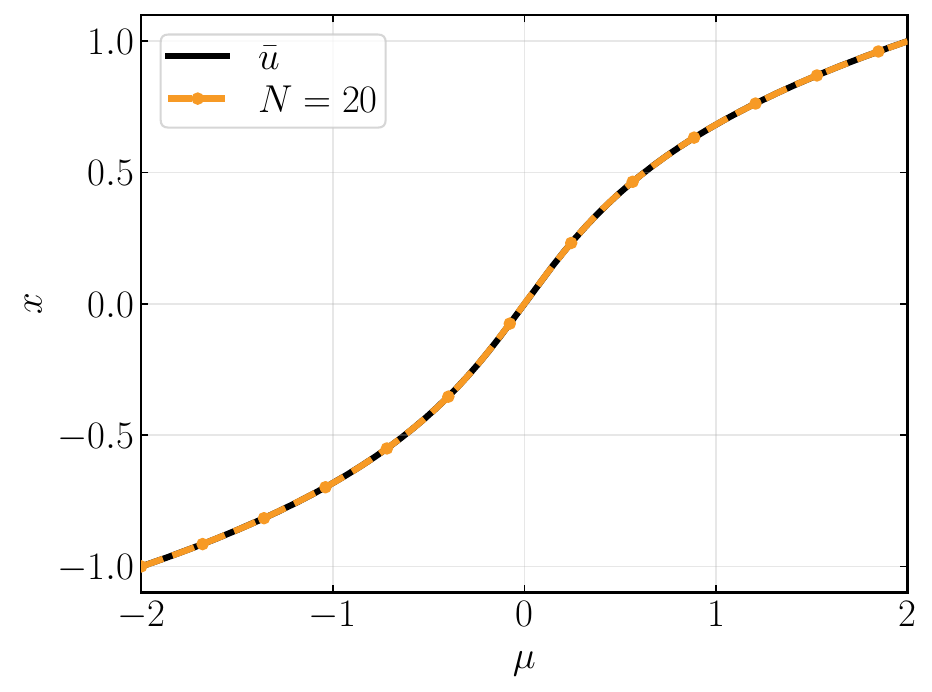}
    \end{minipage}\hfill
    \begin{minipage}{0.48\textwidth}
        \centering
        \includegraphics[width=\linewidth]{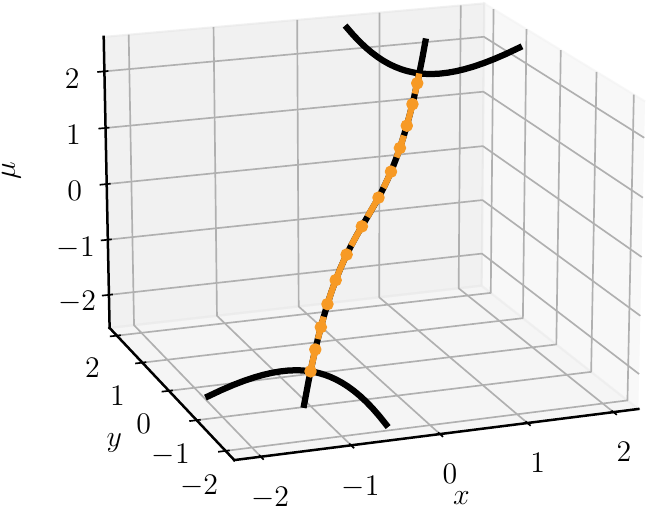}
    \end{minipage}

    \caption{PCE-Galerkin approximation in $I=[-2,2]$, with a  uniqueness regime.}
    \label{fig:toggle_unique}
\end{figure}

\subsection{Lorenz system}
In this section, we demonstrate the robust performance of the proposed framework for the multi-dimensional Lorenz system~\cite{lorenz1963deterministic, sparrow1982lorenz}, a well-known non-linear dynamics benchmark with complex bifurcation structures and chaotic regimes, which describes a simplified model for atmospheric convection. The parametrized system, in terms of the Prandtl number $\gamma$, the Rayleigh number $\rho$, and the geometric factor $\theta$, is given by:
$$
\begin{cases}
    \dot{x}=\gamma(y-x)\\
    \dot{y}=x(\rho-z)-y\\
    \dot{z}=xy-\theta z,
\end{cases} 
\quad\implies\quad
    \func{f}(\vect{u},\gamma,\rho,\theta)=\begin{pmatrix} 
    \gamma(y-x)\\ 
    x(\rho-z)-y\\ 
    xy-\theta z 
    \end{pmatrix}.
$$
where $\vect{u}=(x,y,z)^\top\in\mathbb{R}^3$, $x$ represents the intensity of the convective motion, $y$ and $z$ correspond to spatial temperature variations, and we find the steady-state manifolds by solving the non-linear equation $\func{f}(\vect{u},\rho)=\vect{0}$, where the vector field $\func{f}: \mathbb{R}^3\times I\to\mathbb{R}^3$. Here, we treat the Rayleigh parameter $\rho$, which dictates the instability of the system, as a stochastic input. We define $\rho(\xi)\coloneq\bar{\rho}+\sigma \xi$, while keeping $\gamma=10$ and $\theta=\frac{8}{3}$ as deterministic constants~\cite{lorenz1963deterministic, sparrow1982lorenz}.

We first investigate the behavior of the PCE-Galerkin bifurcating solutions for the problem. It is well-known that for $\rho<1$, the origin $\func{\bar{u}}_0(\rho)=(0,0,0)$ is the unique stable equilibrium. As the Rayleigh parameter crosses the critical value $\rho = 1$, the origin loses stability and the system undergoes a pitchfork bifurcation. Since the condition $\dot{x}=0$ confines the equilibria to the plane $x=y$, two equilibrium branches emerge as $\func{\bar{u}}_{\pm}(\rho) = \left( \pm \sqrt{\theta(\rho-1)}, \pm \sqrt{\theta(\rho-1)}, \rho-1 \right)$, with $\rho > 1$.

To numerically reconstruct the bifurcation diagram, we span the stochastic interval by setting $\bar{\rho} = \frac{3}{2}$ and $\sigma=\frac{1}{2\sqrt{3}}$, corresponding to the domain $I_\rho = [1.0, 2.0]$. By applying the degree continuation, the non-linear solver correctly tracks both the symmetric branch $\bar{\func{u}}_{0}$ and the asymmetric branches $\bar{\func{u}}_{\pm}$. As illustrated in the right panel of Figure~\ref{fig:lorenz_pitchfork}, which captures the three-solution regime, the branch-approximating solutions successfully reconstruct the multi-dimensional pitchfork structure. In this extended three-dimensional space $(\rho,x,y)$, the two asymmetric branches are clearly separated, exhibiting a symmetric configuration with respect to the origin in the $(x,y)$ plane for any given value of the parameter $\rho$.

\begin{figure}[t]
    \centering
    \begin{minipage}{0.48\textwidth}
        \centering
        \includegraphics[width=\linewidth]{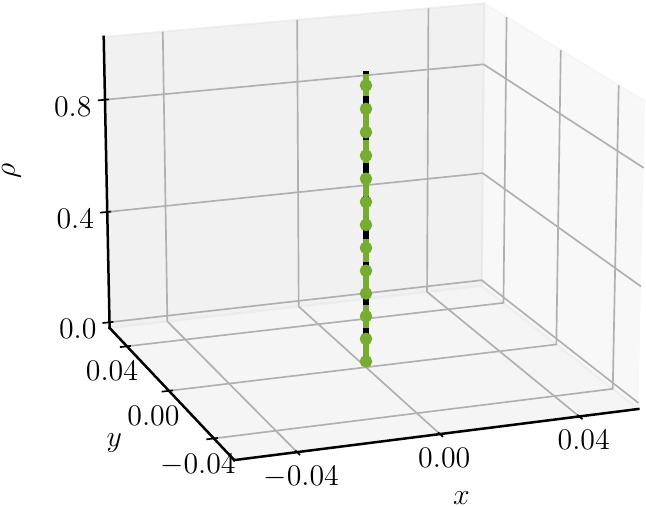}
    \end{minipage}
    \hfill
    \begin{minipage}{0.48\textwidth}
        \centering
        \includegraphics[width=\linewidth]{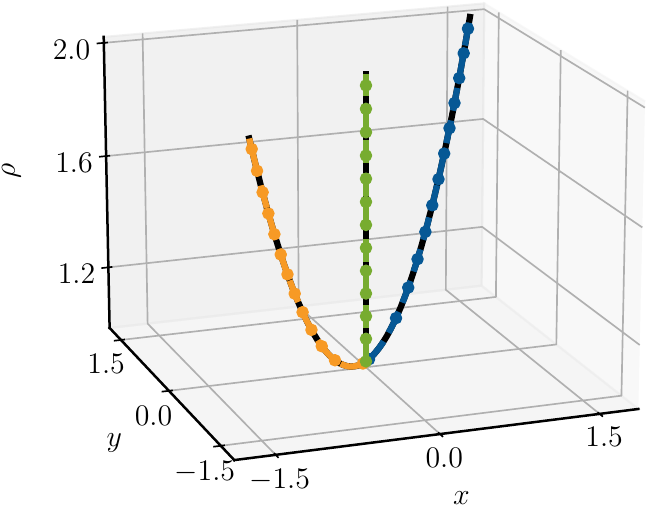}
    \end{minipage}
    \caption{PCE-Galerkin approximation for $N=20$ in the $(\rho,x,y)$ space, portraying the uniqueness regime (\emph{left}) and the three-solution regime (\emph{right}).}
    \label{fig:lorenz_pitchfork}
\end{figure}

Next, we shift our focus to the case where $\rho<1$, i.e., before the bifurcation occurs, aiming to apply the global uniqueness result established in Theorem~\ref{thm:uniqueness}. Since $\func{f}$ is polynomial and the trivial equilibrium branch $\bar{\func{u}}_0(\rho)=\vect{0}$ is smooth, conditions~\eqref{eq:LipLocUnif}--\eqref{ass:A2b} are naturally satisfied. Due to the dissipative properties of the Lorenz system, we verify the reversed inequality conditions for inward-pointing vector fields.

Unlike the toggle switch model, the Jacobian of the Lorenz system is non-symmetric. Therefore, to satisfy the uniqueness hypotheses and verify that the PCE-Galerkin system exhibits a unique solution for sufficiently large $N$, we must equip the space $\mathbb{R}^3$ with an appropriately weighted inner product $\langle \vect{v}, \vect{w} \rangle_P \coloneq \vect{v}^\top P \vect{w}$, with $\|\vect{w}\|_P \coloneq \sqrt{\langle \vect{w}, \vect{w} \rangle_P}$, by defining the diagonal positive-definite (since $\gamma>0$) matrix $P = \text{diag}(\frac{1}{\gamma},1,1)$.
We assume that the entire parameter interval lies in the pre-bifurcation region, i.e.\ $\rho(z)<1$ for all $z \in I_\xi$. The computation of the Jacobian matrix evaluated at $\bar{\func{u}}_0=\vect{0}$ yields:
\begin{equation*}
    J(z) \coloneq D_{\vect{u}}\func{f}(\bar{\func{u}}_0(z), \rho(z))= \begin{pmatrix} 
    -\gamma & \gamma & 0 \\ 
    \rho(z) & -1 & 0 \\ 
    0 & 0 & -\theta 
    \end{pmatrix}.
\end{equation*}
Since for any square matrix, its quadratic form evaluates to zero over its skew-symmetric part, the negative definiteness condition for the quadratic form of the weighted inner product  $\langle J(z)\vect{w}, \vect{w} \rangle_P = \vect{w} (P J(z)) \vect{w}$ is governed exclusively by the symmetric part of the matrix product $P J(z)$, defined as $S(z) \coloneq \frac{1}{2}\left(P J(z) + (J(z) P)^\top\right)$. To understand if the eigenvalues of $S(z)$ are strictly negative, and thus the quadratic form is negative definite, we compute the product $P J(z)$ and its symmetric part $S(z)$ as:
\begin{equation*}
    P J(z) = 
    \begin{pmatrix} 
    -1 & 1 & 0 \\ 
    \rho(z) & -1 & 0 \\ 
    0 & 0 & -\theta 
    \end{pmatrix}, \qquad
    S(z) = \begin{pmatrix} 
    -1 & \frac{1+\rho(z)}{2} & 0 \\ 
    \frac{1+\rho(z)}{2} & -1 & 0 \\ 
    0 & 0 & -\theta 
    \end{pmatrix}.
\end{equation*}
The constraints on the eigenvalues are thus given by $-\theta<0$ and $-3<\rho(z)<1$. Since physical applications of the Lorenz model require $\rho\ge0$ and $\theta\ge0$, and the uniqueness regime holds true for $\rho<1$, these conditions are naturally met. Since $I_\xi$ is compact, there exists a uniform constant $\lambda_{\min}>0$ bounding the maximum eigenvalue of $S(z)$, so that the \emph{negative definiteness} condition in equation~\eqref{eq:uniform_negative_definiteness} is satisfied.

\begin{rmk}
The lower bound of $-3$ is not a limitation of the system, but rather depends on our choice of the matrix $P$. Given $r>0$ and defining a weight matrix $P_r \coloneq \mathrm{diag}(r, 1, 1)$, the top-left block of the new symmetric matrix $S_r(z)$ imposes the condition $4r\gamma > (r\gamma + \rho(z))^2$. The roots of this parabola guarantee negative definiteness for any parameter $\rho(z)$ in any interval strictly contained in
$I_r\coloneq[-r\gamma - 2\sqrt{r\gamma}, -r\gamma + 2\sqrt{r\gamma}]$.
Setting $r=1/\gamma$ maximizes the upper bound, which is the bifurcation point $\rho=1$, effectively recovering our initial matrix $P$. Conversely, increasing $r$ allows the stability region to extend arbitrarily into the negative domain $\rho<0$, ensuring that a valid positive-definite matrix $P_r$ exists also for ranges in the negative domain.
\end{rmk}

To satisfy the hypotheses, it remains to verify the  conditions of uniform inward radial coercivity and relative growth bound. We begin by evaluating the radial  projection of the non-linear vector field. Let $\varrho > 0$ be an arbitrary radius. For any $\vect{u} \in \mathbb{R}^3$, we compute the weighted inner product:
\begin{align*}
    \langle \func{f}(\vect{u}, \rho(z)), \vect{u} \rangle_P &= \frac{1}{\gamma} x\big(\gamma(y-x)\big) + y\big(\rho(z) x - y - xz\big) + z\big(xy - \theta z\big) \\
    &= -x^2 + xy + \rho(z) xy - y^2 - xyz + xyz - \theta z^2\\
    &= -x^2 + (1+\rho(z))xy - y^2 - \theta z^2 \\
    &= \vect{u} S(z) \vect{u},
\end{align*}
which corresponds to the quadratic form of the symmetric matrix $S(z)$ evaluated previously. Thus, applying the bound from negative definiteness, we obtain:
\begin{equation*}
    \langle \func{f}(\vect{u}, \rho(z)), \vect{u} \rangle_P \leq -\lambda_{\min} \|\vect{u}\|_P^2.
\end{equation*}
For vectors such that $\|\vect{u}\|_P \ge \varrho$, we can decompose $\|\vect{u}\|_P^2 = \|\vect{u}\|_P \|\vect{u}\|_P \ge \varrho \|\vect{u}\|_P$. By defining $m_\varrho \coloneq \lambda_{\min} \varrho>0$, the uniform inward radial coercivity condition~\eqref{ass:A4a} is strictly satisfied:
\begin{equation*}
    \langle \func{f}(\vect{u}, \rho(z)), \vect{u} \rangle_P \leq -m_\varrho \|\vect{u}\|_P.
\end{equation*}
Finally, we must prove that the total magnitude of the vector field $\|\func{f}\|_P$ is bounded by its radial projection. The field can be split into its linear and non-linear components:
$$
\func{f}(\vect{u}, \rho(z)) = J(z)\vect{u} + \func{f}_{\text{nl}}(\vect{u}),\qquad \func{f}_{\text{nl}}(\vect{u}) = (0, -xz, xy).
$$
Given $\|\vect{u}\|_P^2=\frac{1}{\gamma}x^2+y^2+z^2\ge \varrho$, $\|\vect{u}\|_P\ge \varrho$ and the triangle inequality, we obtain:
\begin{align*}
    \|\func{f}(\vect{u}, \rho(z))\|_P &\leq \|J(z)\|_P \|\vect{u}\|_P + \|\func{f}_{\text{nl}}(\vect{u})\|_P\\
    &= \|J(z)\|_P \|\vect{u}\|_P + \sqrt{x^2(y^2+z^2)}\\
    &\leq \|J(z)\|_P \|\vect{u}\|_P + \sqrt{\gamma \|\vect{u}\|_P^2 \|\vect{u}\|_P^2}\\
    &\leq \frac{\|J(z)\|_P}{\varrho} \|\vect{u}\|_P^2 + \sqrt{\gamma}\|\vect{u}\|_P^2 \\
    &= \left(\frac{\|J(z)\|_P}{\varrho} + \sqrt{\gamma}\right)\|\vect{u}\|_P^2\\
    &\leq - \left( \frac{\|J(z)\|_P + \varrho\sqrt{\gamma}}{\varrho \lambda_{\min}} \right) \langle \func{f}(\vect{u}, \rho(z)), \vect{u} \rangle_P,\\
\end{align*}
where in the last inequality we used the lower bound established by the radial coercivity step.
Taking the maximum of the matrix norm over the compact set $I_\xi$, we define the uniform constant $C_\varrho \coloneq \max_{z \in I_\xi} \frac{\|J(z)\|_P + \varrho\sqrt{\gamma}}{\varrho \lambda_{\min}}>0$, and the relative growth bound~\eqref{ass:A4b} is satisfied.

Since all hypotheses of the uniqueness Theorem~\ref{thm:uniqueness} are satisfied under the $P$-weighted metric, there exists a sufficiently large $\bar{N}$ such that, if $N\geq\bar{N}$, the PCE-Galerkin problem associated with the stochastic Lorenz system admits a unique solution. As shown in the left panel of Figure~\ref{fig:lorenz_pitchfork} for the pre-bifurcation regime $\rho \in [0,1]$, the solver reliably computes the unique approximation of the exact solution $\bar{\func{u}}_0$, confirming that the numerical projection remains immune to numerical artifacts and spurious roots.

\subsection{Spherical surface of Steady States}
As a final scenario, we investigate the robustness of the PCE-Galerkin framework when applied to a highly degenerate system. Unlike the previous examples characterized by a finite set of isolated branches, we consider a planar system exhibiting a two-dimensional manifold of steady states:
\begin{equation}
\begin{cases}
    \dot{x}=(x - \mu^3)(1 - x^2 - y^2 - \mu^2)\\
    \dot{y}=y(1 - x^2 - y^2 - \mu^2) 
\end{cases} 
\quad\implies\quad
    \func{f}(\vect{u},\mu) = \begin{pmatrix} 
    (x - \mu^3)(1 - x^2 - y^2 - \mu^2) \\ 
    y(1 - x^2 - y^2 - \mu^2) 
    \end{pmatrix},
    \label{eq:spherical_system}
\end{equation}
where $\vect{u} = (x,y)^\top$ and $\mu = \bar{\mu} + \sigma\xi$ acts as the stochastic parameter.

By setting the vector field to zero, we identify two distinct sets of equilibria. The first set is derived from the linear factors, yielding an isolated analytical branch $\bar{\func{u}}_0(\mu) = (\mu^3, 0)$, which can be captured exactly for any expansion degree $N \ge 3$. The second set emerges from the roots of the quadratic factor $1 - x^2 - y^2 - \mu^2 = 0$. Within the parametric domain $I \subset [-1, 1]$, this equation defines a continuum of steady states forming a perfect spherical surface of parameter-dependent radius $R(\mu) = \sqrt{1 - \mu^2}$ in the extended space $(x, y, \mu)$. Note that while the polynomial field $\func{f}$ satisfies~\eqref{eq:LipLocUnif}, the non-degeneracy assumptions~\eqref{eq:uniform_positive_definiteness}--\eqref{eq:uniform_negative_definiteness} fail along the spherical manifold, since there the Jacobian is singular.

\subsubsection{Random Initialization}
Applying a non-linear solver directly to a high expansion degree $N$ without progressive continuation reveals the vast number of solutions generated by the PCE-Galerkin system. 
As illustrated in Figure~\ref{fig:spherical_system_random}, where the left panel displays the projection in the $(\mu,x)$ plane and the right panel shows the full $(\mu,x,y)$ space, these polynomial curves spread across the spherical manifold without concentrating on any privileged zones. 
They sometimes oscillate between the spherical surface, whose boundaries are outlined by gray lines, and the cubic branch, represented by the solid black line.
Although these solutions exhibit oscillatory behavior and do not act as branch-approximating curves, they remain highly informative regarding the topology of the equilibrium manifold.
Since the underlying equilibrium manifold is a continuous two-dimensional surface, this ensemble of distinct algebraic roots can be interpreted as a cloud of numerical samples outlining the true geometric shape of the sphere. 
Even without a continuation path, direct random initializations successfully recover the overall topology and boundaries of the exact steady state surface.

\begin{figure}[t]
    \centering
    \begin{minipage}{0.4\textwidth}
        \includegraphics[width=\linewidth]{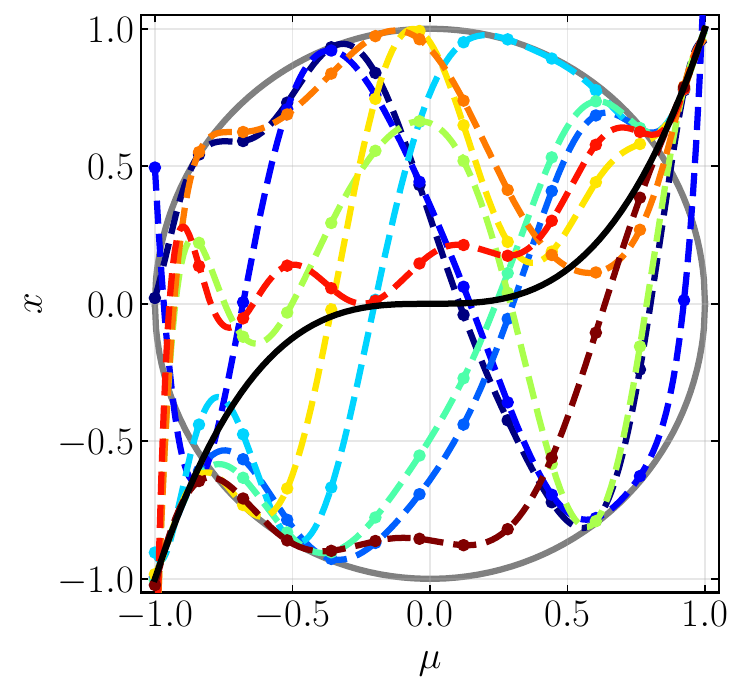}
    \end{minipage}\hfill
    \begin{minipage}{0.48\textwidth}
        \includegraphics[width=\linewidth]{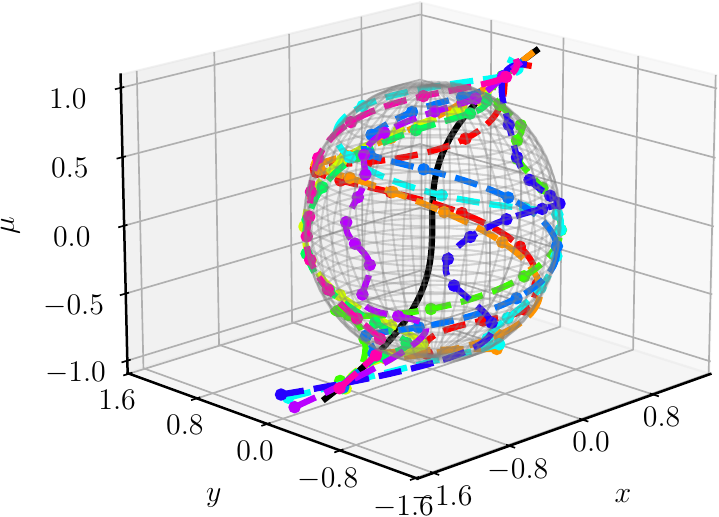}
    \end{minipage}
    \caption{PCE-Galerkin solutions with random initializations for $N=10$ and $I=[-1,1]$.}
    \label{fig:spherical_system_random}
\end{figure}

\subsubsection{Degree Continuation}
We now investigate the behavior of the branch-approximating solutions when applying the degree continuation detailed in Algorithm~\ref{alg:degree_continuation}, consistent with the approach used in the previous sections. At the baseline step $N=0$, the system evaluates the expected value of the vector field. In this specific example, numerical experiments reveal a structural instability already at this baseline level $N=0$.

Through our tests, we observe that once the solver is initialized with random coefficients, the solutions do not reliably remain uniformly distributed on the circular section. Instead, the computed roots empirically tend to collapse onto the $x$-axis, where $y=0$. This specific directional flattening is driven by the structure of the system, where the second equation explicitly vanishes for $y=0$. This creates a numerical preference that drives the trajectories toward the horizontal axis. 

This phenomenon appears to depend on the parameter interval width, as illustrated by the sequence in Figure~\ref{fig:spherical_system_continuation}. For narrow domains, e.g.\ $I=[-0.1,0.1]$ (left panel), the perturbation effect is minor and the algorithm successfully perceives a uniformly distributed continuous ring. However, as the interval widens (e.g.\ $I=[-0.3,0.3]$, middle panel), numerical results show that the random initializations are forced to cluster at the extreme points along the tracks. For even wider uncertainty intervals, ranging from $I=[-0.5, 0.5]$ up to $I=[-1, 1]$ (right panel), we empirically observe that the non-linear solver fails to converge to more than three distinct solutions, all of which strictly collapse onto the $x$-axis.

\begin{figure}[t]
    \centering
    \begin{minipage}{0.32\textwidth}
        \centering
        \includegraphics[width=\linewidth]{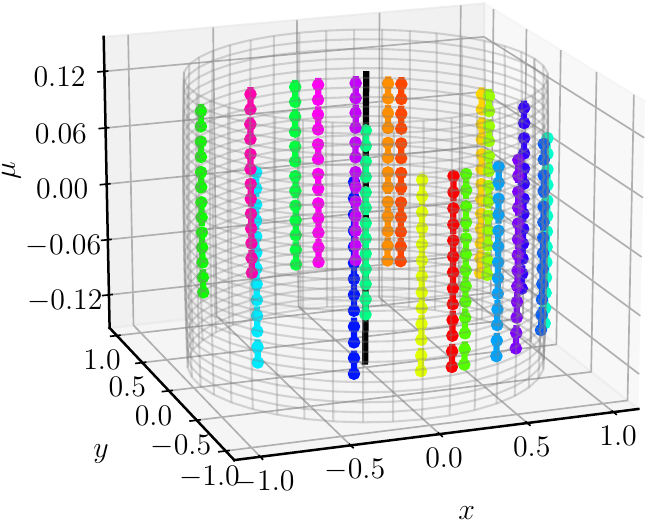}
    \end{minipage}\hfill
    \begin{minipage}{0.32\textwidth}
        \centering
        \includegraphics[width=\linewidth]{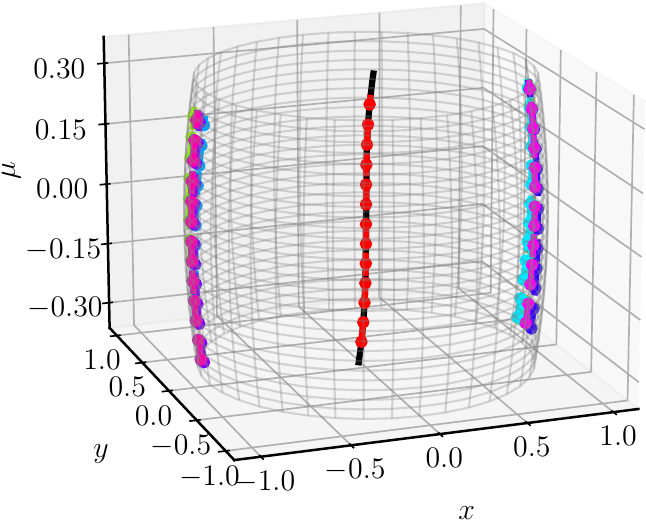}
    \end{minipage}\hfill
    \begin{minipage}{0.32\textwidth}
        \centering
        \includegraphics[width=\linewidth]{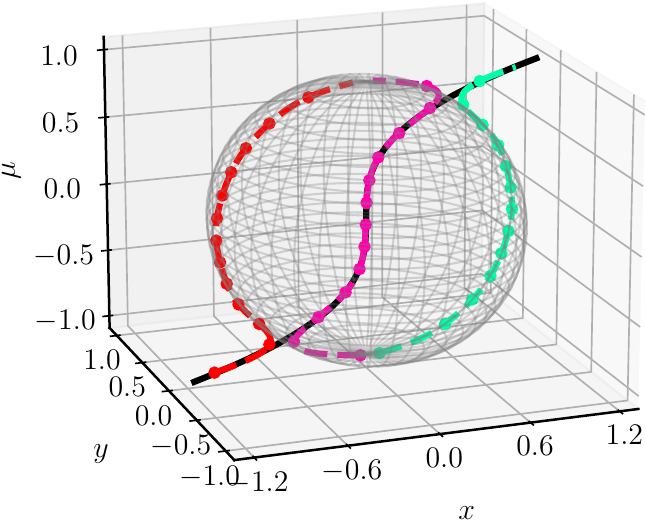}
    \end{minipage}\hfill
    \caption{PCE-Galerkin degree continuation ($N=10$) across domains of increasing width.}
    \label{fig:spherical_system_continuation}
\end{figure}

By testing a simplified version of the system, removing the linear and cubic multipliers and reducing the field to $\func{f}(\vect{u},\mu) = (1 - x^2 - y^2 - \mu^2, 1 - x^2 - y^2 - \mu^2)^\top$, numerical experiments show that the degree-continuation algorithm tracks the solutions across wide intervals. The random initializations follow the curvature of the sphere and remain uniformly distributed across the entire surface. This confirms that the manifold is captured by the PCE-Galerkin framework, and that the previous flattening toward $y=0$ is induced by the coordinate-specific action of the cubic coupling term.
\section{Conclusions and Perspectives}\label{sec:conclusions}
In this work, we have proposed and analyzed a global reconstruction framework for the steady-state manifolds of parameter-dependent dynamical systems. 
By shifting the perspective from local, point-wise continuation methods to a global weak formulation based on Polynomial Chaos Expansion (PCE), we demonstrated that it is possible to recover entire deterministic solution branches simultaneously across the parameter space. 

In particular, our theoretical investigation established the consistency and well-posedness of the PC Galerkin projection. 
We proved that the discrete system successfully inherits the qualitative features of the underlying solution (such as non-existence and uniqueness) provided that specific geometric constraints, namely uniform radial coercivity and local monotonicity, are satisfied. A crucial result of our analysis is the rigorous characterization of convergence rates: while optimal spectral accuracy is achieved in smooth regions of the equilibrium manifold, we also provided theoretical bounds for the algebraic decay observed in the proximity of bifurcation singularities, where the loss of Sobolev regularity limits the polynomial approximation.

From a numerical standpoint, the proposed degree-continuation strategy addresses the inherent challenge of algebraic root proliferation in the Galerkin system. By leveraging the hierarchical nature of the polynomial basis, we successfully isolated the \emph{branch-approximating} solutions from the much larger set of oscillating roots. The robustness of this approach was validated through a diverse collection of examples, ranging from classic normal forms to complex multidimensional systems of practical interest, such as the genetic toggle switch and the chaotic Lorenz attractor. 
Furthermore, the analysis of a highly degenerate continuous spherical manifold highlighted the framework's ability to sample infinite sets of steady states.

In this work, we demonstrate that the global nature of the PCE approach offers significant advantages over traditional sequential algorithms, particularly in providing a functional and efficient surrogate model for many-query sensitivity analysis and uncertainty quantification.
However, we also highlight that the dimension $n(N+1)$ of the resulting non-linear algebraic system grows rapidly with both the state dimension $n$ and the polynomial degree $N$, requiring robust solvers or more advanced strategies to track the branches.
This suggests that the method's scalability could be enhanced, for example by employing multi-element PCE~\cite{wan2005multi}, which partitions the parameter domain to isolate steep gradients. 
On the other hand, the computational bottleneck associated with the high-dimensional Galerkin projection could be mitigated through ROM techniques to compress the state dimension, ad-hoc block preconditioning strategies, or adaptive refinement algorithms~\cite{PichiDeflationbasedCertifiedGreedy2025a}.

Future research directions will primarily focus on extending the theoretical and computational framework developed in this work for ODEs to high-dimensional systems governed by Partial Differential Equations (PDEs). Building upon the initial evidence of PCE applicability in fluid dynamics benchmarks such as the Coand\u{a} effect~\cite{gonnella2024stochastic}, we aim at extending this investigation to more complex non-linear PDEs in computational mechanics, performing a one-shot reconstruction of branch-approximating solutions to efficiently obtain the bifurcation diagram and  discover new branches.
Furthermore, the global approximations obtained via PCE could be coupled with traditional local continuation methods, leveraging the evaluated polynomials as highly accurate initializations to identify the branches in challenging regions of the parameter space.
Finally, a systematic investigation of the analytical properties of the \emph{oscillating solutions} represents a highly promising mathematical frontier. Indeed, as observed for the spherical manifold, these roots can provide deeper insights into the detection of hidden bifurcating structures or the topology of the equilibrium manifold. 

\section*{Acknowledgments and declarations}

\noindent\textbf{Author contribution} G.~Venier contributed to methodology, formal investigation, software development, and drafting the original manuscript. I.C.~Gonnella and F.~Pichi contributed to methodology, formal investigation, review and
editing of the manuscript, and supervision. G.~Rozza contributed to supervision and resources. All authors read and approved the final manuscript.\\

\noindent\textbf{Funding} The authors acknowledge the support provided by INdAM-GNCS (CUP E53C25002010001) and the European Union - NextGenerationEU, in the framework of the iNEST - Interconnected Nord-Est Innovation Ecosystem (iNEST ECS00000043 - CUP G93C22000610007) consortium and its CC5 Young Researchers initiative.\\

\noindent\textbf{Data availability} No datasets were generated or analyzed during the current study. The code reproducing the numerical test cases is publicly available at \url{https://github.com/GiacomoVenier/PCE-Bifurcation-Analysis.git}.\\

\noindent\textbf{Conflict of interest} The authors declare that they have no conflict of interest.\\

\bibliographystyle{abbrv}
\bibliography{bib}

@book{wiggins2003,
  author    = {Stephen Wiggins},
  title     = {Introduction to Applied Nonlinear Dynamical Systems and Chaos},
  edition   = {2nd},
  publisher = {Springer, New York},
  year      = {2003},
  address   = {New York}
}

@article{debusschere2005,
  author  = {Bert J. Debusschere and Habib N. Najm and Philippe P. Pébay and Omar M. Knio
             and Roger G. Ghanem and Olivier P. Le Maître},
  title   = {Numerical challenges in the use of Polynomial Chaos representations
             for stochastic processes},
  journal = {SIAM Journal on Scientific Computing},
  volume  = {26},
  number  = {2},
  pages   = {698--719},
  year    = {2005},
  doi     = {10.1137/S1064827503427741}
}

@book{xiu2010,
  author    = {Dongbin Xiu},
  title     = {Numerical methods for stochastic computations:
               A spectral method approach},
  publisher = {Princeton University Press},
  year      = {2010},
  address   = {Princeton, NJ}
}

@book{brezis2011,
  author    = {Haïm Brezis},
  title     = {Functional Analysis, Sobolev Spaces and Partial Differential Equations},
  publisher = {Springer},
  year      = {2011},
  address   = {New York, NY}
}

@book{boyd2001chebyshev,
  title={Chebyshev and Fourier Spectral Methods},
  author={John P. Boyd},
  year={2001},
  publisher={Dover Publications}
}

@book{canuto2006spectral,
  title={Spectral Methods: Fundamentals in Single Domains},
  author={Claudio Canuto and M.Y. Hussaini and Alfio Quarteroni and Thomas A. Zang},
  year={2006},
  publisher={Springer}
}

@unpublished{gonnella2024stochastic,
  author    = {Gonnella, Isabella Carla and Khamlich, Moaad and Pichi, Federico and Rozza, Gianluigi},
  title     = {A Stochastic Perturbation Approach to Nonlinear Bifurcating Problems},
  note      = {Preprint, arXiv:2402.16803v4},
  year      = {2024},
}

@book{kantorovich1964functional,
  author    = {Kantorovich, Leonid V. and Akilov, Gleb P.},
  title     = {Functional Analysis in Normed Spaces},
  publisher = {Pergamon Press},
  address   = {Oxford},
  year      = {1964},
  note      = {Translated from Russian (1959)}
}

@article{ortega1968newton,
  author    = {Ortega, James M.},
  title     = {The Newton–Kantorovich Theorem},
  journal   = {The American Mathematical Monthly},
  volume    = {75},
  number    = {6},
  pages     = {658--660},
  year      = {1968},
  doi       = {10.2307/2313800},
  note      = {Short note stating the Newton–Kantorovich theorem}  
}

@book{rudin1991functional,
  author    = {Rudin, Walter},
  title     = {Functional Analysis},
  publisher = {McGraw-Hill},
  year      = {1991},
  edition   = {2}
}

@book{trefethen2013approximation,
  author    = {Trefethen, Lloyd N.},
  title     = {Approximation Theory and Approximation Practice},
  publisher = {SIAM},
  year      = {2013},
  address   = {Philadelphia},
  edition   = {1}
}

@book{Rozza2014,
  author       = {Quarteroni, Alfio and Rozza, Gianluigi},
  title        = {Reduced order methods for modeling and computational reduction},
  publisher    = {Springer},
  year         = {2014},
  doi          = {10.1007/978-3-319-02090-7}
}

@Book{Seydel2010,
  author    = {R{\"u}diger Seydel},
  title     = {Practical Bifurcation and Stability Analysis},
  edition   = {3},
  series    = {Interdisciplinary Applied Mathematics},
  volume    = {5},
  publisher = {Springer, New York, NY},
  year      = {2010},
  isbn      = {978-1-4419-1739-3},
  doi       = {10.1007/978-1-4419-1740-9},
}

@book{Kuznetsov2023,
  author       = {Yuri A. Kuznetsov},
  title        = {Elements of Applied Bifurcation Theory},
  edition      = {4},
  series       = {Applied Mathematical Sciences},
  volume       = {112},
  publisher    = {Springer, Cham / Switzerland},
  year         = {2023},
  isbn         = {978-3-031-22006-7},
  doi          = {10.1007/978-3-031-22007-4},
}

@article{pintore2020efficient,
  author  = {Moreno Pintore and Federico Pichi and Martin W. Hess and Gianluigi Rozza and Claudio Canuto},
  title   = {Efficient computation of bifurcation diagrams with a deflated approach to reduced basis spectral element method},
  journal = {Advances in Computational Mathematics},
  volume  = {47},
  number  = {1},
  pages   = {1--39},
  year    = {2021},
  doi     = {10.1007/s10444-020-09827-6}
}

@book{QuarteroniReducedBasisMethods2016,
  title = {Reduced {{Basis Methods}} for {{Partial Differential Equations}}: {{An Introduction}}},
  shorttitle = {Reduced {{Basis Methods}} for {{Partial Differential Equations}}},
  author = {Quarteroni, Alfio and Manzoni, Andrea and Negri, Federico},
  year = {2016},
  series = {La {{Matematica}} per Il 3+2, 92},
  edition = {1st ed. 2016.},
  publisher = {Springer International Publishing},
  address = {Cham},
  doi = {10.1007/978-3-319-15431-2},
  isbn = {978-3-319-15431-2}
}

@article{KuehnUncertaintyQuantificationAnalysis2024,
  title = {Uncertainty Quantification Analysis of Bifurcations of the {{Allen}}--{{Cahn}} Equation with Random Coefficients},
  author = {Kuehn, Christian and Piazzola, Chiara and Ullmann, Elisabeth},
  year = {2024},
  journal = {Physica D: Nonlinear Phenomena},
  volume = {470},
  pages = {134390},
  doi = {10.1016/j.physd.2024.134390}
}

@article{PichiArtificialNeuralNetwork2023,
  title = {An Artificial Neural Network Approach to Bifurcating Phenomena in Computational Fluid Dynamics},
  author = {Pichi, Federico and Ballarin, Francesco and Rozza, Gianluigi and Hesthaven, Jan S.},
  year = {2023},
  journal = {Computers \& Fluids},
  volume = {254},
  pages = {105813},
  doi = {10.1016/j.compfluid.2023.105813}
}

@article{PichiDeflationbasedCertifiedGreedy2025a,
  title = {Deflation-Based Certified Greedy Algorithm and Adaptivity for Bifurcating Nonlinear {{PDEs}}},
  author = {Pichi, Federico and Strazzullo, Maria},
  year = {2025},
  journal = {Communications in Nonlinear Science and Numerical Simulation},
  volume = {149},
  pages = {108941},
  doi = {10.1016/j.cnsns.2025.108941}
}

@article{TomadaSparseIdentificationBifurcating2025,
  title = {Sparse {{Identification}} for Bifurcating Phenomena in {{Computational Fluid Dynamics}}},
  author = {Tomada, Lorenzo and Khamlich, Moaad and Pichi, Federico and Rozza, Gianluigi},
  year = 2025,
  journal = {Computers \& Fluids},
  volume = {302},
  pages = {106841},
  doi = {10.1016/j.compfluid.2025.106841},
  urldate = {2025-10-20}
}

@book{Fulton1998,
  author    = {Fulton, William},
  title     = {Intersection Theory},
  publisher = {Springer-Verlag},
  year      = {1998},
  edition   = {2nd},
  series    = {Ergebnisse der Mathematik und ihrer Grenzgebiete. 3. Folge},
  doi       = {10.1007/978-1-4612-1700-8},
  isbn      = {978-0-387-98549-7}
}

@book{SommeseWampler2005,
  author    = {Sommese, Andrew J. and Wampler, Charles W.},
  title     = {The Numerical Solution of Systems of Polynomials Arising in Engineering and Science},
  publisher = {World Scientific},
  year      = {2005},
  address   = {Singapore},
  doi       = {10.1142/5763},
  isbn      = {978-981-256-040-7}
}

@article{wang2021how,
  title={How much faster does the best polynomial approximation converge than Legendre projection?},
  author={Wang, Haiyong},
  journal={Constructive Approximation},
  volume={54},
  number={3},
  pages={481--504},
  year={2021},
  publisher={Springer}
}

@book{shen2011spectral,
  title={Spectral Methods: Algorithms, Analysis and Applications},
  author={Shen, Jie and Tang, Tao and Wang, Li-Lian},
  volume={41},
  series={Springer Series in Computational Mathematics},
  year={2011},
  publisher={Springer-Verlag},
  address={Berlin Heidelberg}
}

@book{adams2003sobolev,
  title={Sobolev Spaces},
  author={Adams, Robert A. and Fournier, John J. F.},
  edition={2},
  volume={140},
  series={Pure and Applied Mathematics},
  year={2003},
  publisher={Academic Press},
  address={New York}
}

@book{brenner2008mathematical,
  title={The Mathematical Theory of Finite Element Methods},
  author={Brenner, Susanne C. and Scott, L. Ridgway},
  edition={3},
  series={Texts in Applied Mathematics},
  volume={15},
  year={2008},
  publisher={Springer},
  address={New York}
}

@book{bergh1976interpolation,
  title={Interpolation Spaces: An Introduction},
  author={Bergh, J{\"o}ran and L{\"o}fstr{\"o}m, J{\"o}rgen},
  series={Grundlehren der mathematischen Wissenschaften},
  volume={223},
  year={1976},
  publisher={Springer-Verlag},
  address={Berlin Heidelberg}
}

@book{tartar2007sobolev,
  title={An Introduction to Sobolev Spaces and Interpolation Spaces},
  author={Tartar, Luc},
  series={Lecture Notes of the Unione Matematica Italiana},
  volume={3},
  year={2007},
  publisher={Springer},
  address={Berlin Heidelberg}
}

@book{devore1993constructive,
  title={Constructive Approximation},
  author={DeVore, Ronald A. and Lorentz, George G.},
  volume={303},
  series={Grundlehren der mathematischen Wissenschaften},
  year={1993},
  publisher={Springer-Verlag},
  address={Berlin, Heidelberg}
}

@book{milovanovic1994topics,
  title={Topics in Polynomials: Extremal Problems, Inequalities, Zeros},
  author={Milovanovi{\'c}, Gradimir V. and Mitrinovi{\'c}, Dragoslav S. and Rassias, Themistocles M.},
  year={1994},
  publisher={World Scientific},
  address={Singapore}
}

@book{allgower1990numerical,
  title={Numerical Continuation Methods: An Introduction},
  author={Allgower, Eugene L and Georg, Kurt},
  volume={13},
  series={Springer Series in Computational Mathematics},
  year={1990},
  publisher={Springer-Verlag Berlin Heidelberg}
}

@book{dini1878,
  title={Fondamenti per la teorica delle funzioni di variabili reali},
  author={Dini, Ulisse},
  year={1878},
  publisher={T. Nistri},
  address={Pisa}
}

@book{rudin1976,
  title={Principles of Mathematical Analysis},
  author={Rudin, Walter},
  edition={3rd},
  year={1976},
  publisher={McGraw-Hill},
  address={New York}
}

@book{appell1990nonlinear,
  title={Nonlinear Superposition Operators},
  author={Appell, J{\"u}rgen and Zabrejko, Petr P.},
  volume={95},
  year={1990},
  publisher={Cambridge University Press},
  address={Cambridge},
  isbn={0521361028}
}

@article{cameron1947orthogonal,
  title={The orthogonal development of non-linear functionals in series of Fourier-Hermite functionals},
  author={Cameron, Robert H and Martin, William T},
  journal={Annals of Mathematics},
  volume={48},
  pages={385--413},
  year={1947},
  publisher={JSTOR}
}

@book{ghanem1991stochastic,
  title={Stochastic Finite Elements: A Spectral Approach},
  author={Ghanem, Roger G and Spanos, Pol D},
  year={1991},
  publisher={Springer},
  address={New York, NY}
}

@article{xiu2002wiener,
  title={The Wiener-Askey polynomial chaos for stochastic differential equations},
  author={Xiu, Dongbin and Karniadakis, George Em},
  journal={SIAM Journal on Scientific Computing},
  volume={24},
  number={2},
  pages={619--644},
  year={2002},
  publisher={SIAM}
}

@article{xiu2003modeling,
  title={Modeling uncertainty in flow simulations via generalized polynomial chaos},
  author={Xiu, Dongbin and Karniadakis, George Em},
  journal={Journal of Computational Physics},
  volume={187},
  number={1},
  pages={137--167},
  year={2003},
  publisher={Elsevier}
}

@article{brezis2018gagliardo,
  title={Gagliardo--Nirenberg inequalities and non-inequalities: The full story},
  author={Brezis, Ha{\"\i}m and Mironescu, Petru},
  journal={Annales de l'Institut Henri Poincar{\'e} C},
  volume={35},
  number={5},
  pages={1355--1376},
  year={2018},
  publisher={Elsevier},
  doi={10.1016/j.anihpc.2017.11.007}
}

@book{cartan1971differential,
  title={Differential Calculus},
  author={Cartan, Henri},
  year={1971},
  publisher={Hermann / Houghton Mifflin},
  address={Paris / Boston}
}

@article{gottlieb1997gibbs,
  title={On the Gibbs phenomenon and its resolution},
  author={Gottlieb, David and Shu, Chi-Wang},
  journal={SIAM Review},
  volume={39},
  number={4},
  pages={644--668},
  year={1997},
  publisher={SIAM}
}

@book{hesthaven2007spectral,
  title={Spectral Methods for Time-Dependent Problems},
  author={Hesthaven, Jan S and Gottlieb, Sigal and Gottlieb, David},
  year={2007},
  publisher={Cambridge University Press},
  address={Cambridge}
}

@article{lorenz1963deterministic,
  title={Deterministic nonperiodic flow},
  author={Lorenz, Edward N},
  journal={Journal of the Atmospheric Sciences},
  volume={20},
  number={2},
  pages={130--141},
  year={1963}
}

@book{sparrow1982lorenz,
  title={The Lorenz Equations: Bifurcations, Chaos, and Strange Attractors},
  author={Sparrow, Colin},
  year={1982},
  publisher={Springer-Verlag},
  address={New York}
}

@article{gardner2000,
  title={Construction of a genetic toggle switch in Escherichia coli},
  author={Gardner, Timothy S and Cantor, Charles R and Collins, James J},
  journal={Nature},
  volume={403},
  number={6767},
  pages={339--342},
  year={2000},
  publisher={Nature Publishing Group}
}

@manual{bollenbach2023,
  title={Genetic toggle switch},
  author={Bollenbach, T.},
  organization={Advanced Practical Course M Biophysics, University of Cologne},
  year={2023}
}

@book{arnold1992ordinary,
  title={Ordinary Differential Equations},
  author={Arnold, Vladimir I.},
  year={1992},
  publisher={Springer-Verlag},
  address={Berlin, Heidelberg},
  edition={3}
}

@article{wan2005multi,
  title={Multi-element generalized polynomial chaos for arbitrary probability measures},
  author={Wan, Xiaoliang and Karniadakis, George Em},
  journal={SIAM Journal on Scientific Computing},
  volume={28},
  number={3},
  pages={901--928},
  year={2006},
  publisher={SIAM}
}

@unpublished{KumarBifurcationCurveDetection2026,
  title = {Bifurcation Curve Detection with Deflation for Multiparametric {{PDEs}}},
  author = {Kumar, Nitin and Pichi, Federico and Rozza, Gianluigi},
  year = 2026,
  number = {arXiv:2602.12940},
  eprint = {2602.12940},
  primaryclass = {math},
  publisher = {arXiv},
  doi = {10.48550/arXiv.2602.12940},
  note = {Preprint, arXiv:2602.12940},

  archiveprefix = {arXiv}
}

@article{PichiDrivingBifurcatingParametrized2022,
  title = {Driving Bifurcating Parametrized Nonlinear {{PDEs}} by Optimal Control Strategies: Application to {{Navier}}--{{Stokes}} Equations with Model Order Reduction},
  shorttitle = {Driving Bifurcating Parametrized Nonlinear {{PDEs}} by Optimal Control Strategies},
  author = {Pichi, Federico and Strazzullo, Maria and Ballarin, Francesco and Rozza, Gianluigi},
  year = 2022,
  journal = {ESAIM: Mathematical Modelling and Numerical Analysis},
  volume = {56},
  number = {4},
  pages = {1361--1400},
  publisher = {EDP Sciences},
  doi = {10.1051/m2an/2022044}
}

@article{DengLoworderModelSuccessive2020a,
  title = {Low-Order Model for Successive Bifurcations of the Fluidic Pinball},
  author = {Deng, Nan and Noack, Bernd R. and Morzy{\'n}ski, Marek and Pastur, Luc R.},
  year = 2020,
  journal = {Journal of Fluid Mechanics},
  volume = {884},
  pages = {A37},
  doi = {10.1017/jfm.2019.959}
}

@article{HerreroRBReducedBasis2013,
  title = {{{RB}} ({{Reduced}} Basis) for {{RB}} ({{Rayleigh}}--{{B\'enard}})},
  author = {Herrero, H. and Maday, Y. and Pla, F.},
  year = 2013,
  journal = {Computer Methods in Applied Mechanics and Engineering},
  volume = {261--262},
  pages = {132--141},
  doi = {10.1016/j.cma.2013.02.018}
}

@unpublished{PiaSurrogateNormalformsNumerical2025,
  title = {Surrogate Normal-Forms for the Numerical Bifurcation and Stability Analysis of Navier-Stokes Flows via Machine Learning},
  author = {Pia, Alessandro Della and Patsatzis, Dimitrios G. and Rozza, Gianluigi and Russo, Lucia and Siettos, Constantinos},
  year = 2025,
  number = {arXiv:2506.21275},
  eprint = {2506.21275},
  primaryclass = {physics},
  publisher = {arXiv},
  note = {Preprint, arXiv:2506.21275},

  doi = {10.48550/arXiv.2506.21275},
  archiveprefix = {arXiv}
}

@article{BruntonDiscoveringGoverningEquations2016,
  title = {Discovering Governing Equations from Data by Sparse Identification of Nonlinear Dynamical Systems},
  author = {Brunton, Steven L. and Proctor, Joshua L. and Kutz, J. Nathan},
  year = 2016,
  journal = {Proceedings of the National Academy of Sciences},
  volume = {113},
  number = {15},
  pages = {3932--3937},
  publisher = {Proceedings of the National Academy of Sciences},
  doi = {10.1073/pnas.1517384113}
}

@article{ContiReducedOrderModeling2023,
  title = {Reduced Order Modeling of Parametrized Systems through Autoencoders and {{SINDy}} Approach: Continuation of Periodic Solutions},
  shorttitle = {Reduced Order Modeling of Parametrized Systems through Autoencoders and {{SINDy}} Approach},
  author = {Conti, Paolo and Gobat, Giorgio and Fresca, Stefania and Manzoni, Andrea and Frangi, Attilio},
  year = 2023,
  journal = {Computer Methods in Applied Mechanics and Engineering},
  volume = {411},
  pages = {116072},
  publisher = {Elsevier}
}

@article{FarrellDeflationTechniquesFinding2015,
  title = {Deflation {{Techniques}} for {{Finding Distinct Solutions}} of {{Nonlinear Partial Differential Equations}}},
  author = {Farrell, P. E. and Birkisson, {\'A}. and Funke, S. W.},
  year = 2015,
  journal = {SIAM Journal on Scientific Computing},
  volume = {37},
  number = {4},
  pages = {A2026-A2045},
  publisher = {{Society for Industrial and Applied Mathematics}},
  doi = {10.1137/140984798}
}

@article{PichiGraphConvolutionalAutoencoder2024,
  title = {A Graph Convolutional Autoencoder Approach to Model Order Reduction for Parametrized {{PDEs}}},
  author = {Pichi, Federico and Moya, Beatriz and Hesthaven, Jan S.},
  year = 2024,
  journal = {Journal of Computational Physics},
  volume = {501},
  pages = {112762},
  doi = {10.1016/j.jcp.2024.112762}
}

@article{LiDatadrivenModelingBifurcation2025,
  title = {Data-Driven Modeling of Bifurcation Systems by Learning the Bifurcation Parameter Generalization},
  author = {Li, Shanwu and Yang, Yongchao},
  year = 2025,
  journal = {Nonlinear Dynamics},
  volume = {113},
  number = {2},
  pages = {1163--1174},
  doi = {10.1007/s11071-024-10304-8}
}

@article{BoulleBifurcationAnalysisTwodimensional2022,
  title = {Bifurcation Analysis of Two-Dimensional {{Rayleigh--B}}\'enard Convection Using Deflation},
  author = {Boull{\'e}, Nicolas and Dallas, Vassilios and Farrell, Patrick E.},
  year = 2022,
  journal = {Physical Review E},
  volume = {105},
  number = {5},
  eprint = {2102.10576},
  primaryclass = {physics},
  pages = {055106},
  doi = {10.1103/PhysRevE.105.055106},
  archiveprefix = {arXiv}
}

@article{OlshanskiiApproximatingBranchSolutions2025,
  title = {Approximating a Branch of Solutions to the {{Navier}}--{{Stokes}} Equations by Reduced-Order Modeling},
  author = {Olshanskii, Maxim A. and Rebholz, Leo G.},
  year = 2025,
  journal = {Journal of Computational Physics},
  volume = {524},
  pages = {113728},
  doi = {10.1016/j.jcp.2025.113728},
  urldate = {2025-01-21}
}

@book{UeckerNumericalContinuationBifurcation2021,
  title = {Numerical {{Continuation}} and {{Bifurcation}} in {{Nonlinear PDEs}}},
  author = {Uecker, Hannes},
  year = 2021,
  series = {Other {{Titles}} in {{Applied Mathematics}}},
  publisher = {{Society for Industrial and Applied Mathematics}},
  doi = {10.1137/1.9781611976618},
  isbn = {978-1-61197-660-1}
}

@article{ShahabNeuralNetworksBifurcation2025,
  title = {Neural Networks for Bifurcation and Linear Stability Analysis of Steady States in Partial Differential Equations},
  author = {Shahab, Muhammad Luthfi and Susanto, Hadi},
  year = 2024,
  journal = {Applied Mathematics and Computation},
  volume = {483},
  pages = {128985},
  doi = {10.1016/j.amc.2024.128985}
}

@article{VenturiStochasticBifurcationAnalysis2010,
  title = {Stochastic Bifurcation Analysis of {{Rayleigh}}--{{B\'enard}} Convection},
  author = {Venturi, Daniele and Wan, Xiaoliang and Karniadakis, George Em},
  year = 2010,
  journal = {Journal of Fluid Mechanics},
  volume = {650},
  pages = {391--413},
  publisher = {Cambridge University Press},
  doi = {10.1017/S0022112009993685}
}

@article{lemaitre2009,
  title={A Newton method for the resolution of steady stochastic Navier-Stokes equations},
  author={Le Ma{\^\i}tre, Olivier P.},
  journal={Computers \& Fluids},
  volume={38},
  number={8},
  pages={1566--1579},
  year={2009},
  publisher={Elsevier},
  doi={10.1016/j.compfluid.2009.01.009}
}

@article{patil2023,
  title={Reduced-Order Modeling with Time-Dependent Bases for PDEs with Stochastic Boundary Conditions},
  author={Patil, Prerna and Babaee, Hessam},
  journal={SIAM/ASA Journal on Uncertainty Quantification},
  volume={11},
  number={2},
  pages={727--756},
  year={2023},
  publisher={SIAM},
  doi={10.1137/21M1468097}
}

\appendix
\section{}\label{appendix}

\subsubsection*{Bézout's Theorem}
In Section~\ref{sec:ode-bifurcations}, if $n=1$ and the function $f$ is polynomial in $u$, the Galerkin projection is formulated as a system of $N_{PC}+1$ polynomial equations. To bound the maximum number of potential solutions and justify the emergence of spurious, oscillating roots, we rely on the real affine version of Bézout's Theorem~\cite{Fulton1998, SommeseWampler2005}.

\begin{thm}\label{thm:Bezout}
    Consider $m$ affine hypersurfaces $M_i \subset \R^m$, each defined by a polynomial $f_i \in \R[x_1,\dots,x_m]$ of degree $d_i$. If these do not have any common component and $M_1 \cap \dots \cap M_m$ consists of a finite number of points, then the number of intersection points, counted with multiplicity $m_P$, satisfies:
    \begin{equation*}
        \sum_{P \in (M_1 \cap \dots \cap M_m)} m_P \leq d_1 \cdot \ldots \cdot d_m.
    \end{equation*}
\end{thm}

\subsubsection*{Nemytskii Operators}
The theoretical analysis of the weak formulation requires the Nemytskii operator $\mathcal{F}$ to be well-behaved over continuous functions. The following proposition rigorously establishes its local Lipschitz continuity, used to derive the error bounds in Theorem~\ref{thm:consistency} and to prove the uniqueness result.

\begin{prop}\label{prop:LocalLipschitzNemytskii}
If equation~\eqref{eq:LipLocUnif} holds, the Nemytskii operator $\mathcal{F}$ is locally Lipschitz on $\mathcal{C}(I; \R^n)$ with respect to the supremum norm $\|\cdot\|_\infty$. 
\end{prop}
\begin{proof}
Given $\mathbf{u}_0 \in \mathcal{C}(I; \R^n)$ and $r > 0$, define $R \coloneq \|\mathbf{u}_0\|_\infty + r$. For any $\mathbf{u}, \mathbf{v} \in \mathcal{C}(I; \R^n)$ such that $\|\mathbf{u} - \mathbf{u}_0\|_\infty \le r$ and $\|\mathbf{v} - \mathbf{u}_0\|_\infty \le r$, it follows that $\|\mathbf{u}\|_\infty, \|\mathbf{v}\|_\infty \le R$. Thus, $\mathbf{u}(\mu), \mathbf{v}(\mu)$ lie within the closed ball $\bar{B}_R(\mathbf{0}) \subset \R^n$ for all $\mu \in I$. 
By the uniform Lipschitz assumption~\eqref{eq:LipLocUnif} on $\mathbf{f}$:
$$
\|\mathcal{F}(\mathbf{u})(\mu) - \mathcal{F}(\mathbf{v})(\mu)\|_{\R^n} = \|\mathbf{f}(\mathbf{u}(\mu), \mu) - \mathbf{f}(\mathbf{v}(\mu), \mu)\|_{\R^n} \le L_R \|\mathbf{u}(\mu) - \mathbf{v}(\mu)\|_{\R^n}.
$$
Taking the supremum over $\mu \in I$ yields $\|\mathcal{F}(\mathbf{u}) - \mathcal{F}(\mathbf{v})\|_\infty \le L_R \|\mathbf{u} - \mathbf{v}\|_\infty$, proving the operator is Lipschitz on the ball of radius $r$ centred at $\mathbf{u}_0$.
\end{proof}

\subsubsection*{Newton--Kantorovich Theorem}
To prove Theorem~\ref{thm:existence_of_PC_ls_sol}, we cast the projection as a root-finding task in a finite-dimensional space. The Newton--Kantorovich Theorem~\cite{kantorovich1964functional, ortega1968newton} provides the foundation to guarantee both the existence of these discrete roots and their local uniqueness around the true branch.

\begin{thm}\label{thm:NewtonKantorovich}
    Let $X$ and $Y$ be Banach spaces and $F \colon D \subset X \to Y$. Suppose that on an open convex set $D_0 \subset D$, $F$ is Fréchet differentiable and
    $$
    \|F'(x) - F'(y)\| \leq K \|x - y\|, \quad x, y \in D_0.
    $$
    For some $x_0 \in D_0$, assume that $\Gamma_0 := [F'(x_0)]^{-1}$ is defined on all $Y$ and that $h := \beta K \eta \leq \frac{1}{2}$ where $\|\Gamma_0\| \leq \beta$ and $\|\Gamma_0 F(x_0)\| \leq \eta$. Set
    $$
    r^* = \frac{1}{\beta K} \left(1 - \sqrt{1 - 2h} \right), \quad r^{**} = \frac{1}{\beta K} \left(1 + \sqrt{1 - 2h} \right),
    $$
    and suppose that $S := \left\{ x \;\middle|\; \|x - x_0\| \leq r^* \right\} \subset D_0$. Then, the Newton iterates 
    $$
    x_{k+1} = x_k - [F'(x_k)]^{-1} F(x_k), \quad k = 0, 1, \dots
    $$
    are well defined, lie in $S$, and converge to a solution $x^*$ of $F(x) = 0$ which is unique in 
    $$
    D_0 \cap \left\{ x \;\middle|\; \|x_0 - x\| < r^{**} \right\}.
    $$
    Moreover, if $h < \frac{1}{2}$ the order of convergence is at least quadratic.
\end{thm}

\subsubsection*{Convergence of Legendre Projections}
The strong error bounds of the least-squares approximation $u_{LS,N}$ in Theorem~\ref{thm:consistency} depend on the spectral convergence properties of Legendre polynomials. Details about this theorem and the relative proofs can be found in \cite{xiu2010, boyd2001chebyshev, canuto2006spectral}

\begin{thm}\label{thm:convergence_ls}
    Let $\mathbf{f}\in L^2(I; \R^n)$ and $\tilde{\mathbf{f}}_N$ be its best $L^2$-approximation in $(\Pi_N)^n$, written with respect to the orthogonal basis of Legendre polynomials. Then,
    
    \begin{enumerate}
        \item if $\mathbf{f}\in H^s(I; \R^n)$, there exists a constant $C$ independent of $N$ such that
        $$
        \|\mathbf{f}-\tilde{\mathbf{f}}_N\|_{L^2(I; \R^n)}\leq \frac{C}{N^s}\|\mathbf{f}\|_{H^s(I; \R^n)};
        $$
        \item if $\mathbf{f}$ is of class $\mathcal{C}^\infty$, the convergence is super-polynomial; moreover, if $\mathbf{f}$ is analytical on a Bernstein ellipse, with $\rho>1$, then
        $$
        \|\mathbf{f}-\tilde{\mathbf{f}}_N\|_{L^2(I; \R^n)}=\mathcal{O}(\rho^{-N});
        $$
    \end{enumerate}
    
\end{thm}

\end{document}